\documentclass[11pt]{article}

\usepackage[utf8]{inputenc}
\usepackage[letterpaper,margin=1.0in]{geometry}

\usepackage{amsmath}\allowdisplaybreaks
\usepackage{amsfonts,bm}
\usepackage{amssymb}
\usepackage{bbm}
\usepackage{dsfont}

\def\eqref#1{equation~(\ref{#1})}

\def\floor#1{\left\lfloor #1 \right\rfloor}
\def\1{\bf{1}}

\newcommand{\Norm}[1]{\left\| #1 \right\|}
\newcommand{\norm}[1]{\left\| #1 \right\|_2}
\def\inner#1#2{\left\langle #1, #2 \right\rangle}

\def\eps{{\varepsilon}}

\def\vzero{{\bf{0}}}
\def\vone{{\bf{1}}}

\def\vb{{\bf{b}}}
\def\vc{{\bf{c}}}

\def\ve{{\bf{e}}}

\def\vu{{\bf{u}}}
\def\vv{{\bf{v}}}
\def\vw{{\bf{w}}}
\def\vx{{\bf{x}}}
\def\vy{{\bf{y}}}

\def\fA{{\mathcal{A}}}
\def\fB{{\mathcal{B}}}

\def\fD{{\mathcal{D}}}

\def\fF{{\mathcal{F}}}

\def\fH{{\mathcal{H}}}

\def\fL{{\mathcal{L}}}

\def\fO{{\mathcal{O}}}
\def\fP{{\mathcal{P}}}

\def\fX{{\mathcal{X}}}
\def\fY{{\mathcal{Y}}}

\def\BN{{\mathbb{N}}}

\def\BP{{\mathbb{P}}}

\def\BR{{\mathbb{R}}}

\def\mA {{\bf A}}
\def\mB {{\bf B}}

\def\mD {{\bf D}}

\def\mI {{\bf I}}

\newcommand{\E}{\mathbb{E}}

\newcommand{\Var}{\mathrm{Var}}

\DeclareMathOperator*{\argmin}{arg\,min}
\DeclareMathOperator*{\argminmax}{arg\,min\,max}

\DeclareMathOperator{\diag}{diag}
\DeclareMathOperator{\diam}{diam}
\DeclareMathOperator{\spn}{span}
\DeclareMathOperator{\prox}{prox}
\DeclareMathOperator{\geo}{Geo}

\usepackage{amsthm}
\theoremstyle{plain}
\newtheorem{theorem}{Theorem}[section]
\newtheorem{definition}[theorem]{Definition}

\newtheorem{lemma}[theorem]{Lemma}

\newtheorem{remark}[theorem]{Remark}
\newtheorem{corollary}[theorem]{Corollary}
\newtheorem{proposition}[theorem]{Proposition}

\makeatletter
\def\Ddots{\mathinner{\mkern1mu\raise\p@
                \vbox{\kern7\p@\hbox{.}}\mkern2mu
                \raise4\p@\hbox{.}\mkern2mu\raise7\p@\hbox{.}\mkern1mu}}
\makeatother

\newcommand{\pr}[2][]{\BP_{#1}\left[ #2 \right]}

\makeatletter
\newcommand*{\rom}[1]{\expandafter\@slowromancap\romannumeral #1@}
\makeatother

\newcommand{\bone}{\mathbbm{1}}

\usepackage{mathrsfs}
\usepackage{multirow}
\usepackage[dvipsnames]{xcolor}
\def\todo#1 {\textcolor{red}{Todo: #1}}
\usepackage[numbers]{natbib}
\usepackage{hyperref}
\usepackage{makecell}

\usepackage{tikz}
\usetikzlibrary{positioning, shapes.geometric}

\begin{document}

\title{Lower Complexity Bounds of Finite-Sum Optimization Problems: The Results and Construction}

\author{
  Yuze Han 
  \thanks{Equal Contribution. }
  \thanks{School of Mathematical Sciences,
  Peking University; email: 
  \texttt{hanyuze97@pku.edu.cn}.}
  \and
  Guangzeng Xie 
  \footnotemark[1]
  \thanks{Academy for Advanced Interdisciplinary Studies,
  Peking University; email: \texttt{smsxgz@pku.edu.cn}.}
  \and
  Zhihua Zhang 
  \thanks{School of Mathematical Sciences,
  Peking University; email:
  \texttt{zhzhang@math.pku.edu.cn}.}
}


\maketitle

\begin{abstract}
In this paper we study the lower complexity bounds for finite-sum optimization problems, where the objective is the average of $n$ individual component functions.
We consider Proximal Incremental First-order (PIFO) algorithms which have access to the gradient and proximal oracles for each component function. To incorporate loopless methods, we also allow PIFO algorithms to obtain the full gradient infrequently.
We develop a novel approach to constructing the hard instances, which partitions the tridiagonal matrix of classical examples into $n$ groups. This construction is friendly to the analysis of PIFO algorithms.
Based on this construction,
we establish the lower complexity bounds for finite-sum minimax optimization problems when the objective is convex-concave or nonconvex-strongly-concave and the class of component functions is $L$-average smooth. Most of these bounds are nearly matched by existing upper bounds up to log factors.
We can also derive similar lower bounds for finite-sum minimization problems as previous work under both smoothness and average smoothness assumptions.
Our lower bounds imply that proximal oracles for smooth functions are not much more powerful than gradient oracles.
\end{abstract}


\section{Introduction}
We consider the following optimization problem
\begin{align}\label{prob:main}
    \min_{\vx \in \fX} \max_{\vy \in \fY} f(\vx, \vy) \triangleq \frac{1}{n} \sum_{i=1}^n f_i(\vx, \vy),
\end{align}
where 
the feasible sets  $\fX \subseteq \BR^{d_x}$ and $\fY \subseteq \BR^{d_y}$ are closed and convex. 
This formulation contains several popular machine learning applications such as matrix games~\citep{carmon2019variance,carmon2020coordinate,ibrahim2019linear}, 
regularized empirical risk minimization~\citep{zhang2017stochastic,tan2018stochastic},
AUC maximization~\citep{joachims2005support,ying2016stochastic,shen2018towards}, 
robust optimization~\citep{ben2009robust,yan2019stochastic} 
and reinforcement learning~\citep{du2017stochastic,dai2018sbeed}. 

A popular approach for solving minimax problems is the first-order algorithm which iterates with gradient and proximal point operation~\citep{chambolle2011first,chambolle2016ergodic,mokhtari2019proximal,mokhtari2019unified,thekumparampil2019efficient,luo2019stochastic}. Along this line,
\citet{zhang2019lower} and \citet{ibrahim2019linear} presented tight lower bounds for solving strongly-convex-strongly-concave minimax problems by first-order algorithms. 
\citet{ouyang2018lower} studied a more general case that the objective function is only convex-concave.
However, these analyses~\citep{ouyang2018lower,zhang2019lower,ibrahim2019linear} do not consider the specific finite-sum structure as in Problem~(\ref{prob:main}). They only considered the deterministic first-order algorithms which are based on the full gradient and exact proximal point iteration.

In big data regimes, the number of components $n$ in Problem~(\ref{prob:main}) could be very large and we would like to devise randomized optimization algorithms that avoid accessing the full gradient frequently. For example, \citet{palaniappan2016stochastic} used stochastic variance reduced gradient (SVRG) algorithms to solve Problem~(\ref{prob:main}). Similar to convex optimization, one can accelerate it by catalyst~\citep{lin2018catalyst,yang2020catalyst} and proximal point  techniques~\citep{defazio2016simple,luo2019stochastic}. 
Note that SVRG is a double-loop algorithm, where the full gradient is calculated periodically with a constant interval.
There are also some loopless algorithms where a coin flip decides whether to calculate the full gradient at each iteration \citep{alacaoglu2021stochastic, luo2021near}.
Although randomized algorithms are widely used for solving minimax problems, the study of their lower bounds  is still open. All of the existing lower bound analysis 
focuses on convex or nonconvex minimization problems~\citep{agarwal2015lower,woodworth2016tight,arjevani2016dimension,lan2017optimal,hannah2018breaking,fang2018spider}.  

This paper considers randomized PIFO algorithms for solving Problem~(\ref{prob:main}), which are formally defined in Definition \ref{def:alg}. 
These algorithms have access to the Proximal Incremental First-order Oracle (PIFO)
\begin{align}\label{eq:oracle}
    h^{\mathrm{PIFO}}_{f_i}(\vx, \vy, \gamma) 
    \triangleq & \big[ f_i(\vx,\vy), \nabla_\vx f_i(\vx,\vy),
    -\nabla_\vy f_i(\vx,\vy),
    \prox^\gamma_{f_i}(\vx,\vy)
    \big],
\end{align}
where $i\in\{1,\dots,n\}$, $\gamma>0$, and the proximal operator is defined as
\begin{align*}
    \prox^\gamma_{f_i}(\vx, \vy) \triangleq 
   \argminmax_{\vu \in \BR^{d_x},\vv \in \BR^{d_y}} \left\{ f_i(\vu, \vv)
   +\frac{1}{2\gamma}\norm{\vx-\vu}^2 -\frac{1}{2\gamma} \norm{\vy-\vv}^2 \right\}.
\end{align*}
Compared to Incremental First-order Oracle (IFO), which is defined as
$h^{\mathrm{IFO}}_{f_i} (\vx, \vy) 
=
[f_i(\vx, \vy),$
$\nabla_\vx f_i (\vx, \vy), - \nabla_\vy f_i (\vx, \vy) ]$, 
PIFO additionally provides the proximal oracle of the component function.
To incorporate loopless methods, we also allow PIFO algorithms to access the full gradient infrequently with the interval obeying geometric distributions.

We consider the general setting where $f(\vx, \vy)$ is $L$-smooth and $(\mu_x, \mu_y)$-convex-concave, i.e., the function $f(\cdot, \vy) - \frac{\mu_x}{2} \norm{\cdot}^2$ is convex for any $\vy \in \fY$ and the function $-f(\vx, \cdot) - \frac{\mu_y}{2} \norm{\cdot}^2$ is convex for any $\vx \in \fX$. 
When $\mu_x, \mu_y \ge 0$, our goal is to find an $\eps$-suboptimal solution $(\hat{\vx}, \hat{\vy})$ to Problem~(\ref{prob:main}) such that the primal-dual gap is less than $\eps$, i.e.,
\begin{align*}
    \max_{\vy \in \fY} f(\hat{\vx}, \vy) - \min_{\vx \in \fX} f(\vx, \hat\vy) < \eps.
\end{align*}
On the other hand, when $\mu_x < 0, \mu_y > 0$, $f(\vx, \vy)$ is called a nonconvex-strongly-concave function, which has been widely studied in \citep{rafique2018non, lin2020near, ostrovskii2020efficient, luo2020stochastic}.
In this case, our goal is instead to find an $\eps$-stationary point $\hat{\vx}$ of $\phi_f(\vx) \triangleq \max_{\vy \in \fY} f(\vx, \vy)$, which is defined as
\begin{align*}
    \norm{\nabla \phi_f(\hat{\vx})} < \eps.
\end{align*}
It is worth noting that
by setting the feasible set of $\vy$ as a singleton, the minimax problem becomes a minimization problem.
Then we can omit the dependence of $f$ on $\vy$ and rewrite the function as $f(\vx)$ with some abuse of notation. When $f(\vx)$ is convex, our goal is to find an $\eps$-suboptimal solution $\hat{\vx}$ such that $f(\hat{\vx}) - \min_{\vx \in \fX} f(\vx) < \eps$, while when $f(\vx)$ is nonconvex, our goal is to find an $\eps$-stationary point $\hat{\vx}$ such that $\norm{\nabla f(\hat{\vx})} < \eps$.
\subsection{Contributions}
Our contributions are summarized as follows.
\begin{enumerate}
    \item We propose a novel construction framework to analyze lower complexity bounds for finite-sum optimization problems. Different from previous work, we decompose
    the classical tridiagonal matrix in \citet{nesterov2013introductory} into $n$ groups and each component function is defined in terms of only one group.
    Such a construction facilitates the analysis for both IFO and PIFO algorithms (see Definition~\ref{def:alg}).
    \item We establish the lower complexity bounds for finite-sum minimax problems when $f$ is convex-concave or nonconvex-strongly-concave and $\{f_i\}_{i=1}^n$ is $L$-average smooth (see Definition~\ref{def:average}). When $f$ is convex-concave, our lower bounds nearly match existing upper bounds up to log factors. The results are summarized in Table~\ref{table-minimax-average}\footnote{The work of \citet{zhang2021complexity} appeared on arXiv during the review process of our work.}.
    \item  For finite-sum minimization problems, we derive similar lower bounds as \citet{woodworth2016tight, hannah2018breaking, zhou2018stochastic} when each $f_i$ is $L$-smooth or $\{ f_i \}_{i=1}^n$ is $L$-average smooth.
    The results are summarized in Tables \ref{table-min} and \ref{table-min-average} in Section \ref{sec:min}.
    Compared to previous work, our framework provides more intuition about the optimizing process and requires fewer dimensions to construct the hard instances.
    \item For most cases, our lower bounds are nearly matched by IFO algorithms. This implies that the proximal oracles for smooth functions are not much more powerful than gradient oracles, which is consistent with the observation in \citet{woodworth2016tight}.
\end{enumerate}

\renewcommand{\arraystretch}{1.8}
\begin{table}[t]
\caption{Upper and lower bounds with the assumption that $\{f_i\}_{i=1}^n$ is $L$-average smooth and $f$ is $(\mu_x, \mu_y)$-convex-concave. When $\mu_x \ge 0$ and $\mu_y \ge 0$, the goal is to find an $\eps$-suboptimal solution with $\diam(\fX) \le 2 R_x, \diam(\fY) \le 2 R_y$. And when $\mu_x < 0$, the goal is to find an $\eps$-stationary point of the function $\phi_f(\vx) \triangleq \max_{\vy \in \fY} f(\cdot, \vy)$ with $\Delta = \phi_f(\vx_0) - \min_{\vx}\phi_f(\vx)$ and $\fX = \BR^{d_x}, \fY = \BR^{d_y}$. The condition numbers are defined as $\kappa_x = L / \mu_x$ and $\kappa_y = L / \mu_y$ when $\mu_x, \mu_y > 0$. }
\label{table-minimax-average}
\footnotesize
\setlength{\tabcolsep}{3pt}
\begin{center}
 \begin{tabular}{|c|c|c|} 
 \hline
 Cases & Upper or Lower bounds & References \\ 
  \hline
  \multirow{2}{*}{$\mu_x > 0, \mu_y > 0$} & $ \tilde{\fO} \left( \sqrt{ n \left(\sqrt{n} + \kappa_x
  \right)\left(\sqrt{n} + \kappa_y \right)} \log(1/\eps) \right) $ & \citet{luo2021near} \\
  \cline{2-3}
 & $\Omega\left( \sqrt{ n \left(\sqrt{n} + \kappa_x \right)\left(\sqrt{n} + \kappa_y \right)} \log(1/\eps)\right)$ & Theorem \ref{thm:average:strongly-strongly}  \\
 \hline
 \multirow{2}{*}{ $\mu_x = 0, \mu_y > 0$ } & $ \tilde{\fO} \Big( \big( n + R_x n^{3/4} \sqrt{\frac{L}{\eps}} +
 R_x \sqrt{ \frac{n L \kappa_y }{\eps} }
 + n^{3/4}\! \sqrt{\kappa_y } \big) \log \left( \frac{1}{\eps} \right) \Big) $ & \citet{luo2021near} \\
 \cline{2-3}
 & $\Omega\Big(n + R_x n^{3/4} \sqrt{\frac{L}{\eps}} + 
 R_x \sqrt{ \frac{n L \kappa_y }{\eps} }
 + n^{3/4}\! \sqrt{\kappa_y } \log \left( \frac{1}{\eps} \right) \Big)$ & Theorem \ref{thm:average:convex-strongly} \\
 \hline
 \multirow{2}{*}{ $\mu_x = 0, \mu_y = 0$ } & $\tilde{\fO} \Big( \big(n + \frac{\sqrt{n} L R_x R_y}{\eps} + (R_x {+} R_y) n^{3/4} \sqrt{\frac{ L}{\eps}} \big) \log \left( \frac{1}{\eps} \right)  \Big) $ & \citet{luo2021near} \\
 \cline{2-3}
 & $\Omega\Big(n + \frac{\sqrt{n} L R_x R_y}{\eps} + (R_x {+} R_y) n^{3/4} \sqrt{\frac{ L}{\eps}}\Big)$ & Theorem \ref{thm:average:convex-concave} \\
 \hline
 \multirow{2}{*}{ \makecell[c]{ $\mu_x < 0, \mu_y > 0$, \\ $\kappa_y = \Omega(n)$ } } 
 & $\tilde{\fO} \left( (n + n^{3/4} \sqrt{\kappa_y} ) \Delta L \eps^{-2}  \right)$ & \citet{zhang2021complexity} \\
 \cline{2-3}
 & $ \Omega \left( n + \sqrt{n \kappa_y} \Delta L \eps^{-2}  \right) $ & Theorem \ref{thm:average:nonconvex-strongly}; \citet{zhang2021complexity} \\
 \hline
\end{tabular}
\end{center}
\end{table}
\renewcommand{\arraystretch}{1}

\subsection{Related Work}

\paragraph{Lower bounds for finite-sum minimization problems}
There has been extensive study on this topic.
\citet{agarwal2015lower} established the lower bound $\Omega(n + \sqrt{n(\kappa-1) } \log (1/\eps) )$ when each component is $L$-smooth and their average is $\mu$-strongly convex by a resisting oracle construction, where $\kappa = L / \mu$ is the condition number. However, their lower bound only applies to deterministic algorithms. \citet{lan2017optimal}  obtained the lower bound $\Omega( (n + \sqrt{n \kappa}) \log (1/\eps) )$ for randomized incremental gradient methods, but their bound does not apply to multi-loop methods such as SVRG \citep{johnson2013accelerating} and SARAH \citep{nguyen2017sarah}.
\citet{woodworth2016tight} provided the lower bound $\Omega(n + \sqrt{n \kappa} \log(1/\eps) )$ for any randomized algorithms using gradient and proximal oracles.
Moreover, when the objective is only convex, their lower bound is $\Omega(n + \sqrt{n L /\eps} )$.
\citet{arjevani2016dimension} established a similar lower bound for the strongly convex case and their bound also applies to stochastic coordinate-descent methods.
\citet{hannah2018breaking} improved this bound to $\Omega \big( \frac{n \log (1/\eps)}{(1 + \log ( n / \kappa ) )_+ } \big)  $ when $\kappa = \fO(n)$.
\citet{zhou2019lower} proved lower bounds $\Omega( n + n^{3/4}\sqrt{\kappa} \log(1/\eps) )$ and $\Omega(n + n^{3/4} \sqrt{L/\eps})$ for the strongly convex and convex case respectively under the weaker condition that the class of component function is $L$-average smooth.

When the objective is nonconvex, \citet{fang2018spider}
proved the lower bound $\Omega( L \sqrt{n} / \eps^2 ) $ for 
$\eps = \fO (\sqrt{L} / n^{1/4}) $ under the average smooth condition.
\citet{li2021page} improved the bound to $\Omega(n + L \sqrt{n} / \eps^2)$ for an arbitrary $\eps$.
Under a more refined condition that objective is $\mu$-weakly convex (see Definition \ref{def:weakly}),
 \citet{zhou2019lower} established the lower bound to $\Omega ( 1/\eps^2 \min\{ n^{3/4} \sqrt{ L \mu}, \sqrt{n} L \}  )$ when $\eps$ is sufficiently small.
They also provided the lower bound $\Omega ( 1/\eps^2 \min\{ \sqrt{ n L \mu},  L \}  ) $ when each component is $L$-smooth.

\paragraph{Upper bounds for finite-sum minimax problems} 
For Problem~(\ref{prob:main}), if $\mu_x, \mu_y \ge 0$ and
each $f_i$ is $L$-smooth, the best known upper bound is $\fO\left(\left(n + \sqrt{n} (\kappa_x + \kappa_y) \right) \log(1/\eps) \right)$~\citep{carmon2019variance,luo2019stochastic}.
 Furthermore, if each $f_i$ has ${L}$-cocoercive gradient, which is a stronger assumption than smoothness, \citet{chavdarova2019reducing} provided an upper bound $\fO\left(\left(n + \kappa_x + \kappa_y \right) \log(1/\eps) \right)$.
 If $\{ f_i \}_{i=1}^n$ is $L$-average smooth, Accelerated SVRG \citep{palaniappan2016stochastic} attained the upper bound
$\tilde{\fO} \left( \left( n + \sqrt{n} (\kappa_x + \kappa_y) \right) \log (1 / \eps) \right)$ and \citet{alacaoglu2021stochastic} obtained the bound ${\fO} \left( \left( n + \sqrt{n} (\kappa_x + \kappa_y) \right) \log (1 / \eps) \right)$. Then \citet{luo2021near} improved this bound to $\tilde{\fO} ( \sqrt{  n (\sqrt{n} + \kappa_x) (\sqrt{n} + \kappa_y)  }   \log (1 / \eps) )$ by catalyst acceleration.
The same technique was also employed to derive lower bounds for the convex-strongly-concave case where $\mu_x = 0, \mu_y > 0$ \citep{yang2020catalyst, luo2021near}.

For the convex-concave case ($\mu_x = \mu_y = 0$), \citet{carmon2019variance} established the upper bound $\fO(n + \sqrt{n} L / \eps )$ under the smoothness assumption, while \citet{alacaoglu2021stochastic} developed the same upper bound under the average smoothness assumption. \citet{luo2021near} still used the catalyst acceleration and derived a similar bound.

In terms of the nonconvex-strongly-concave case ($\mu_x < 0, \mu_y > 0$),
\citet{luo2020stochastic} proposed an upper bound $\tilde\fO\left( n + \min\{  \sqrt{n} \kappa_y^2, \kappa_y^{2} + n \kappa_y \} \eps^{-2} \right)$, while \citet{zhang2021complexity} developed an upper bound $\tilde\fO\left( (n + n^{3/4} \sqrt{\kappa_y}) L \eps^{-2} \right)$. The latter is better when $n = \fO(\kappa^4)$. We emphasize that both results are under the average smoothness assumption.

\paragraph{Loopless methods} Variance-reduced methods for finite-sum minimization problems such as SVRG~\citep{johnson2013accelerating}, Katyusha~\citep{allen2017katyusha} and  SARAH~\citep{nguyen2017sarah} have a double-loop design where the full gradient needs to be calculated periodically. Recently, many researchers aim to study their loopless variants or devise new loopless methods such that whether to access the full gradient depends on a coin toss with a small head probability.
Equivalently speaking, the inner loop size obeys the geometric distribution with a small success probability.
Such a design facilities theoretical analysis without deteriorating the convergence rates.
For example, loopless SVRG (L-SVRG) was first proposed in 
\citet{hofmann2015variance}
and then further analyzed in 
\citet{kovalev2020don, qian2021svrg}
together with loopless Katyusha (L-Katyusha). Loopless SARAH (L2S) was developed in
\citet{li2020convergence}.
Other loopless methods include but are not limited to KatyushaX~\citep{allen2018katyushax}, PAGE~\citep{li2021page} and Anita~\citep{li2021anita}.
For finite-sum minimax problems, there are also many loopless methods~\citep{loizou2020stochastic, alacaoglu2021stochastic, beznosikov2022stochastic}.

\paragraph{The proximal oracle}
The proximal oracle provides more information than the gradient oracle and has been used in algorithm design~\citep{shalev2013stochastic, defazio2016simple, lan2017optimal, luo2019stochastic}.
Compared with catalyst acceleration, employing proximal oracles would neither increase the number of loops nor induce additional parameter tuning.
When each component function enjoys a simple form~\citep{zhang2017stochastic, du2017stochastic, lan2017optimal, carmon2019variance}, 
the proximal operator can be computed efficiently.
In terms of the power of proximal oracles, \citet{woodworth2016tight}
have shown that for smooth functions, the gradient oracle is sufficient for the optimal rate.
As a comparison,
for nonsmooth functions, having access to proximal oracles does reduce the complexity and \citet{woodworth2016tight} presented optimal
methods that improve over those only using gradient oracles.

\subsection{Organization}
The remainder of this paper is organized as follows. 
In Section~\ref{sec:preliminaries}, we introduce some necessary notation and definitions and give a concentration inequality for geometric distributions.
In Section~\ref{sec:alg}, we present and discuss the definition of PIFO algorithms.
In Section~\ref{sec:frame}, we define the optimization complexity and construct the hard instances for Problem~(\ref{prob:main}).
In Sections~\ref{sec:minimax} and \ref{sec:min}, we provide and analyze our lower bounds for finite-sum minimax and minimization problems respectively.
Finally, in Section~\ref{sec:conclusion}, we summarize our results and propose some future research directions.
\section{Preliminaries}\label{sec:preliminaries}
In this section, we present some necessary notation and definitions used in our paper and then give a concentration inequality about geometric distributions.

\paragraph{Notation}
We denote the set $\{1,2,\dots, n \}$ by $[n]$.
$a_+ \triangleq \max\{a, 0\}$ represent the positive part of a real number.
The projection operator is defined as  $\fP_{\fX} (\vx) \triangleq \argmin_{\vx' \in \fX} \norm{\vx' - \vx} $ where $\fX$ is a convex set and $\norm{\cdot}$ is the Euclidean norm.
We use $\vzero$ for all-zero vectors and $\ve_i$ for the unit vector with the $i$-th element equal to $1$ and others equal to $0$. Their dimensions will be specified by an additional subscript, if necessary, and otherwise are clear from the context.
We use $\geo(p)$ to denote the geometric distribution with success probability $p$, i.e., $Y \sim \geo(p)$ implies $\pr{Y = k} = (1 - p)^k p \text{ for } 0 < p \le 1, k \in \{0,1,2,\dots \}$.
Finally, we use the notation $\fO(\cdot), \Omega(\cdot), \Theta(\cdot)$ to hide absolute constants which do not depend on any problem parameter, and notation $\tilde{\fO}(\cdot)$ to hide absolute constants and log factors.

\begin{definition}\label{def:smooth}
    For a differentiable function $\varphi(\vx)$ from $\fX$ to $\BR$ and $L>0$,  $\varphi$ is said to be $L$-smooth if its gradient is $L$-Lipschitz continuous; that is, for any $\vx_1, \vx_2 \in \fX$, we have
    \begin{align*}
        \|\nabla \varphi(\vx_1) - \nabla \varphi(\vx_2)\|_2 \leq L \norm{\vx_1 - \vx_2}.
    \end{align*}
\end{definition}

\begin{definition}\label{def:average}
    For a class of differentiable functions $\{\varphi_i(\vx): \fX \to \BR \}_{i=1}^n$ and $L > 0$, $\{ \varphi_i \}_{i=1}^n$ is said to be $L$-average smooth if for any $\vx_1, \vx_2 \in \fX$, we have
    \begin{align*}
        \frac{1}{n} \sum_{i=1}^n \|\nabla \varphi_i(\vx_1) - \nabla \varphi_i(\vx_2)\|_2^2 \leq L^2 \norm{\vx_1 - \vx_2}^2.
    \end{align*}
\end{definition}
The assumption of average smoothness is widely used in many finite-sum optimizations~\citep{zhou2018stochastic, fang2018spider, zhou2019lower, alacaoglu2021stochastic}.

Now we discuss the relationship between smoothness and average smoothness.
For a class of differentiable
functions $\{ \varphi_i(\vx): \fX \rightarrow \BR \}_{i=1}^n$ and their average $\bar{\varphi} (\vx) = \frac{1}{n} \sum_{i=1}^n \varphi_i(\vx)$, we have the following result
\begin{align*}
    \varphi_i \text{ is } L \text{-smooth}, \forall i \Longrightarrow 
    \{ \varphi_i  \}_{i=1}^n \text{ is } L \text{-average smooth} \Longrightarrow
    \bar{\varphi} \text{ is } L \text{-smooth}.
\end{align*}
Moreover, suppose that $\varphi_i$ is $L_i$-smooth, $\bar{\varphi}$ is $L$-smooth and $\{ \varphi_i \}_{i=1}^n$ is $L'$-average smooth, we have $L \le L' \le \sqrt{ \frac{1}{n} \sum_{i=1}^n L_i^2 } $ and $L \le \frac{1}{n} \sum_{i=1}^n L_1  $.

However, $L$ and $L'$ can be much smaller than $L_i$. 
For example, if $\varphi_i(\vx) = \frac{1}{2} \left( \inner{\ve_i}{\vx} \right)^2$, then we have $L_i = 1$, $L = 1 / n$ and $L' = 1 / \sqrt{n}$. 
As a result, it is more restrictive to say that each $\varphi_i$ is $L$-smooth than to say that $\{ \varphi_i \}_{i=1}^n$ is $L$-average smooth.

\begin{definition}
    For a differentiable function $\varphi(\vx)$ from $\fX$ to $\BR$,  $\varphi$ is said to be convex if for any $\vx_1, \vx_2 \in \fX$, we have
    \begin{align*}
        \varphi(\vx_2) \ge \varphi(\vx_1) + \inner{\nabla \varphi(\vx_1)}{\vx_2 - \vx_1}.
    \end{align*}
\end{definition}

\begin{definition}\label{def:weakly}
    For a constant $\mu$, if the function
    $
        \hat{\varphi}(\vx) = \varphi(\vx) - \frac{\mu}{2} \norm{\vx}^2
    $
    is convex, then $\varphi$ is said to be $\mu$-strongly convex if $\mu > 0$ and $\varphi$ is said to be $\mu$-weakly convex if $\mu < 0$.
\end{definition}
One can check that if $\varphi$ is $L$-smooth, then it is $(-L)$-weakly-convex.

\begin{definition}
For a differentiable function $\varphi(\vx)$ from $\fX$ to $\BR$, we call $\hat{\vx}$ an $\eps$-stationary point of $\varphi$ if 
\[
\norm{ \nabla \varphi (\hat{\vx}) } < \eps.
\]
\end{definition}

\begin{definition}\label{def:convex-concave}
    For a differentiable function $f(\vx, \vy)$ from $\fX\times \fY$ to $\BR$,  
    $f$ is said to be convex-concave, if the function $f(\cdot, \vy)$ is convex for any $\vy \in \fY$ and the function $-f(\vx, \cdot)$ is convex for any $\vx \in \fX$.
    Furthermore, $f$ is said to be $(\mu_x, \mu_y)$-convex-concave, if the function $f(\vx, \vy) - \frac{\mu_x}{2} \norm{\vx}^2 + \frac{\mu_y}{2} \norm{y}^2$ is convex-concave.
\end{definition}

\begin{definition}
    We call a minimax optimization problem $\min_{\vx \in \fX} \max_{\vy \in \fY} f(\vx, \vy)$ satisfying the strong duality condition if
    \[
    \min_{\vx \in \fX} \max_{\vy \in \fY} f(\vx, \vy)
    = \max_{\vy \in \fY} \min_{\vx \in \fX}  f(\vx, \vy).
    \]
\end{definition}
By Sion's minimax theorem, if $\varphi(\vx, \vy)$ is convex-concave and either $\fX$ or $\fY$ is a compact set, then the strong duality condition holds. 

\begin{definition}
We call $({\vx}^*, {\vy}^* ) \in  \fX \times \fY$ the saddle point of $f(\vx, \vy)$ if 
\[
f(\vx^*, \vy) \le f(\vx^*, \vy^*) \le f(\vx, \vy^*)
\]
for all $(\vx, \vy) \in \fX \times \fY$.
\end{definition}

\begin{definition}
Suppose the strong duality of Problem (\ref{prob:main}) holds. We call $(\hat{\vx}, \hat{\vy}) \in \fX \times \fY$ an $\eps$-suboptimal solution to Problem (\ref{prob:main}) if
\[
\max_{\vy \in \fY} f(\hat{\vx}, \vy) - \min_{\vx \in \fX} f(\vx, \hat{\vy}) < \eps.
\]
\end{definition}

\subsection{A Concentration Inequality about Geometric Distributions}
In this subsection, we introduce a concentration inequality about geometric distributions.

\begin{lemma}\label{lem:geo}
    Let $\{Y_i\}_{i=1}^m$ be independent random variables,
    and $Y_i$ follows a geometric distribution with success probability $p_i$.
    Then for $m \ge 2$, we have
    \begin{align*}
        \pr{\sum_{i=1}^m Y_i > \frac{m^2}{4(\sum_{i=1}^m p_i)}} \ge \frac{1}{9}.
    \end{align*}
\end{lemma}
Lemma \ref{lem:geo} implies that at least with a constant probability, the sum of geometric random variables is larger than a constant number, which depends on the number of variables and their success probabilities.
Then we can obtain a lower bound of $\E \sum_{i=1}^m Y_i$, which is helpful to the construction in Section \ref{sec:frame}.
The proof is deferred to Appendix \ref{appendix:geo}.

\section{PIFO Algorithms}\label{sec:alg}
In this section, we present our definition of PIFO algorithms.
We first discuss previous definitions in Section \ref{sec:alg:previous} and our formal definition is given in Section \ref{sec:alg:our}.

\subsection{Discussion on Previous Definitions}\label{sec:alg:previous}
In this subsection,
we discuss the definitions of oracles and algorithms in previous work on the minimization problem $\min_{\vx \in \fX} f(\vx) = \frac{1}{n} \sum_{i=1}^n f_i(\vx)$.
With some abuse of notation, we do not distinguish the oracles for minimization problems from those for minimax problems.

\paragraph{IFO and PIFO}
The IFO
is defined as $h^{\mathrm{IFO}}_{f_i} (\vx) \triangleq [f_i(\vx), \nabla f_i (\vx)]$,
which
takes as input a point $\vx \in \fX$ and a component function $f_i$ and returns the function value and the gradient of $f_i$ at $\vx$. Many lower bounds for minimization optimization are based on this oracle, e.g., \citet{agarwal2015lower, lan2017optimal, zhou2019lower}.
They all consider \textit{linear-span randomized first-order algorithms}\footnote{The formal definition is given in Definition 3.3 in \citet{zhou2019lower}. 
Although the results of \citet{agarwal2015lower} do not rely on the linear span assumption, this assumption can be made without loss of generality, as shown in their Appendix A.}.
For these algorithms, the current point lies in the linear span of previous points and gradients returned by earlier IFO calls.

\citet{woodworth2016tight} considers the PIFO 
which is stronger than IFO and is defined as $h^{\mathrm{PIFO}}_{f_i} (\vx, \gamma) \triangleq [f_i(\vx), \nabla f_i(\vx), \prox_{f_i}^\gamma (\vx)] $ with the proximal operator $\prox_{f_i}^\gamma (\vx) \triangleq \argmin_{\vu} \big\{ f_i (\vu) + \frac{1}{2\gamma} \norm{\vx - \vu}^2 \big\}$.
When $f_i$ is convex, any $\gamma > 0$ is feasible.
Different from IFO, PIFO provides global
information about the function. 
To see this, letting $\gamma \rightarrow \infty$ 
yields
the exact minimizer of $f_i$.
Based on PIFO, \citet{woodworth2016tight} consider the class of \textit{any} randomized algorithms, a more general class than \textit{linear-span randomized first-order algorithms}.
We also
emphasize that
when $f_i$ is nonconvex, $\gamma$ should be sufficiently small such that $f_i(\vu) + \frac{1}{2 \gamma} \norm{\vx - \vu}^2 $ is a convex function of $\vu$. Otherwise, it can be pretty hard to calculate
$\prox_{f_i}^\gamma(\vx)$.
Specially, if $f$ is $(-\mu)$-weakly convex, we need to ensure $0 < \gamma < 1/\mu$.

\paragraph{Sampling of the component function}
Note that both IFO and PIFO depend on a specific component $f_i$. Different methods use different ways to choose the index $i$.
Some of them, e.g., SAGA \citep{defazio2014saga}, RPDG \citep{lan2017optimal},
pick $i$ randomly according to some distribution over $[n]$ and the full gradient is calculated only at the initial point.
However, much more methods need to calculate the full gradient periodically, either with a deterministic or random interval.
For multi-loop methods,
e.g., SVRG~\citep{johnson2013accelerating}, Katyusha~\citep{allen2017katyusha} and Spider~\citep{fang2018spider},
the interval is predetermined, while for loopless methods, e.g., KatyushaX~\citep{allen2018katyushax}, L2S~\citep{li2020convergence}, L-SVRG~
\citep{kovalev2020don},
the interval is a geometric random variable.


The lower bound of \citet{lan2017optimal} 
requires that the index $i_t$ at iteration $t$ is sampled from a predetermined distribution over $[n]$. Thus their bound does not apply to methods such as SVRG and L-SVRG.
\citet{woodworth2016tight, zhou2019lower} 
do not specify the way to choose $i_t$. As a result, their class of algorithms does include those multi-loop or loopless methods.

\citet{arjevani2016dimension} and \citet{hannah2018breaking} consider p-CLI algorithms equipped with the generalized first-order oracle, 
where the current point and the gradient can be left-multiplied by preconditioning matrices.
They do not specify the way to choose $i_t$, either. Thus their lower bounds apply to all the methods mentioned above.
Moreover, their framework can also be equipped with the steepest coordinate descent oracle to incorporate methods such as SDCA \citep{shalev2016sdca}.

\subsection{Our Definition}\label{sec:alg:our}
In this subsection,
we come back to the minimax problem (\ref{prob:main}) and 
formally introduce the definition of PIFO algorithms.

Recall that the PIFO has been defined in (\ref{eq:oracle}).
For convenience, we also define the First-order Oracle (FO) as
$h^{\mathrm{FO}}_{f} (\vx, \vy) \triangleq [f(\vx, \vy), \nabla_\vx f(\vx, \vy), -\nabla_\vy f(\vx, \vy)] $, which returns the full gradient information.
Since the feasible set of Problem~(\ref{prob:main}) is not necessarily the whole space, the algorithm should also access the projection operators $\fP_\fX$ and $\fP_\fY$.
Then we can define the PIFO algorithms we focus on in our paper.
\begin{definition}\label{def:alg}
Consider a randomized PIFO algorithm $\fA$ to solve Problem (\ref{prob:main}). Denote the point obtained by $\fA$ after step $t$ by $(\vx_{t}, \vy_{t}) $, which is generated by the following procedure.
\begin{enumerate}
    \item Initialize the set $\fH$ as $\{ (\vx_0, \vy_0) \}$, the distribution $\fD$ over $[n]$, a positive number $q \le c_0 / n 
    $ and set $t=1$.
    \item\label{item:pifo_query} Sample $i_t \sim \fD$ and query the oracle $h_{f_i}^{\mathrm{PIFO}}$. \label{alg:minimax:step2}
    \item Sample a Bernoulli random variable $a_t$ with expectation equal to $q$. If $a_t=1$, query the FO $h_{f}^{\mathrm{FO}}$ and add $ (\vx_{t-1}, \vy_{t-1}) $ to $\fH$. 
    \item Obtain $(\tilde{\vx}_{t}, \tilde{\vy}_{t})$ following the linear-span protocol
    \begin{align*}
        (\tilde{\vx}_{t}, \tilde{\vy}_{t} ) & \in \spn \big\{ (\vx_{0}, \vy_{0}), \dots, (\vx_{t-1}, \vy_{t-1}), \prox_{f_{i_j}}^{\gamma_j} (\vx_l, \vy_l) \text{ for } l < j \le t,\\
        & \qquad \qquad (\nabla_\vx f_{i_j} (\vx_{l}, \vy_{l} ), \vzero_{d_y} ),
        (\vzero_{d_x}, - \nabla_\vy f_{i_j} (\vx_{l}, \vy_{l} ) )
        \text{ for } l < j \le t, \\
        & \qquad \qquad 
        (\nabla_\vx f(\vu, \vv), \vzero_{d_y} ),
        (\vzero_{d_x}, - \nabla_\vy f(\vu, \vv) )  \text{ for } (\vu, \vv) \in \fH \big\}.
    \end{align*}
    \item Projection step: $\vx_{t} = \fP_{\fX} (\tilde{\vx}_{t}), \vy_{t} = \fP_{\fY} (\tilde{\vy}_{t} )$.
    \item 
    Output $(\vx_{t},
    \vy_{t})$,
    or
    set $t+1$ to $t$ and go back to step \ref{alg:minimax:step2}.
\end{enumerate}
Let $\mathscr{A}$ be the class of all such PIFO algorithms. 
A PIFO algorithm becomes an IFO algorithm if it queries the IFO at step \ref{item:pifo_query}. 
\end{definition}

\begin{remark}
We remark on some details in our definition of PIFO algorithms.
\begin{enumerate}
    \item The random vector sequence $\{ (i_t, a_t) \}_{t \ge 1}$ are mutually independent and each $i_t$ is also independent of $a_t$.
    \item $\fH$ is the set of points where FO is called and
    simultaneous PIFO queries \citep{fang2018spider, zhou2018stochastic, luo2020stochastic} are allowed.
    Previous PIFO or FO queries can be reused.
    At step $t$, the algorithm has access to  $(\nabla_\vx f_{i_t}(\vx_{0}, \vy_{0}), - \nabla_\vy f_{i_t}(\vx_{0}, \vy_{0}) ), \dots, (\nabla_\vx f_{i_t}(\vx_{t-1}, \vy_{t-1}), - \nabla_\vy f_{i_t}(\vx_{t-1}, \vy_{t-1}) ) $ with shared $i_t$ as well as
    gradient information obtained at previous steps,
    i.e., 
    $(\nabla_\vx f_{i_j} (\vx_l, \vy_l), - \nabla_\vy f_{i_j} (\vx_l, \vy_l) )$ for $l < j < t$. 
    \item When $f_i$ is not convex-concave, $\gamma$ should be chosen such that $f_i(\vu, \vv) + \frac{1}{2 \gamma} \norm{\vx - \vu}^2 - \frac{1}{2 \gamma} \norm{\vy - \vv}^2 $ is convex-concave w.r.t. to $(\vu, \vv)$.
    \item Without loss of generality, we assume that the PIFO algorithm $\fA$ starts from $(\vx_0, \vy_0) = (\vzero_{d_x}, \vzero_{d_y})$ to simplify our analysis. Otherwise, we can take $\{{\tilde f}_i(\vx, \vy) = f_i(\vx + \vx_0, \vy + \vy_0)\}_{i=1}^n$ into consideration.
    \item Let $p_i = \pr[Z \sim \fD]{Z = i}$ for $i \in [n]$. 
    The distribution 
    $\fD$ 
    can be the uniform distribution or
    based on the smoothness of the component functions, e.g., 
    $ p_i
    \propto L_i$ \citep{xiao2014proximal} or $ p_i
    \propto L_i^2$ \citep{allen2018katyushax} for $\, i \in [n]$, 
    where $L_i$ is the smoothness parameter of $f_i$.
    We can assume that $p_1 \le p_2 \le \cdots \le p_n$ by rearranging the component functions $\{f_i\}_{i=1}^n$.
          Suppose that $p_{s_1} \le p_{s_2} \le \cdots \le p_{s_n}$ where $\{s_i\}_{i=1}^n$ is a permutation of $[n]$. We can consider $\{\hat{f}_i\}_{i=1}^n$ and categorical distribution $\fD'$ such that the algorithm draws $\hat{f}_{i} \triangleq f_{s_i}$ with probability $p_{s_i}$ instead.
\end{enumerate}
\end{remark}
Recall that by setting $\fY$ as a singleton, we can obtain the definition of IFO and PIFO algorithms for finite-sum minimization problems.

We emphasize that only the proximal operator of the individual component function $f_i$ is allowed.
The algorithm is not accessible to the proximal operator of the averaged function $f$.
In practice, each $f_i$ usually depends on a single sample and enjoys a simple form~\citep{zhang2017stochastic, du2017stochastic,lan2017optimal, carmon2019variance}.
Then $\prox_{f_i}^{\gamma} (\vx, \vy)$ is easy to calculate.
However, computing $\prox_{f}^{\gamma} (\vx, \vy)$ is as hard as
solving the original problem~(\ref{prob:main}).
To see this, just let $\gamma \rightarrow \infty$.

\paragraph{Methods for minimization problems}
Clearly, methods such as SAGA~\citep{defazio2014saga} and PointSAGA~\citep{defazio2016simple} belong to PIFO algorithms, since these methods only calculate the full gradient at the first iteration.
Other methods such as SVRG~\citep{johnson2013accelerating} and Katyusha~\citep{allen2017katyusha} have two loops and
the full gradient needs to be calculated periodically at the beginning of the outer loop.
Although these two-loop methods do not satisfy our definition, their loopless variants do.
These loopless variants only have one loop and whether to compute the full gradient depends on a coin toss with a small head probability, i.e., $q$ in Definition \ref{def:alg}.
\citet{kovalev2020don} have shown that L-SVRG and L-Katyusha enjoy the same theoretical properties as the original methods.
With a constant $q$, these loopless methods can also be viewed as two-loop methods with a random inner-loop size that obeys the geometric distribution with success probability $q$.
Other loopless methods that satisfy our definition include KatyushaX~\citep{allen2018katyushax}, L2S~\citep{li2020convergence}, PAGE~\citep{li2021page}, Anita~\citep{li2021anita} and so on. For these methods, the order of $q$ is usually $\Theta(1/n)$. And it suffices to set $c_0 = 2$.

Now we consider catalyst accelerated methods. It looks like these methods do not satisfy our definition, since they have two loops and the full gradient needs to be calculated at the beginning of the outer loop. 
Nevertheless, we can slightly change them without affecting the convergence rate.
Firstly, we can replace the algorithm used to solve the inner-loop subproblem, e.g., SVRG, with its loopless variant.
Secondly, the complexity of the inner-loop is at least of the order $\Omega(n)$ (all the components need to be sampled at least once). 
Now we remove the full gradient step at the beginning of the outer loop and do not update the current point until the FO is called.
In expectation, we need $\Theta(1/q)$ more steps.
Thus, if we choose $q = \Theta (1/n)$, such a change makes no difference to the order of the complexity.

\paragraph{Methods for minimax problems}

One can check SAGA~\citep{palaniappan2016stochastic} and PointSAGA~\citep{defazio2016simple} are PIFO algorithms. Although SVRG~\citep{palaniappan2016stochastic} does not satisfy our definition, we believe a loopless variant of it can share the same convergence properties. 
Existing loopless methods that belong to PIFO algorithms include L-SVRHG~\citep{loizou2020stochastic} and L-SVRE\footnote{The method was renamed by \citet{luo2021near} and we adopt the new name.}~\citep{alacaoglu2021stochastic}.
Moreover, similar to the analysis above, the catalyst accelerated methods in \citet{luo2021near,zhang2021complexity} also satisfy our definition. For these methods, the order of $q$ is still $\Theta(1/n)$ and we can set $c_0=2$.

Finally, we emphasize that all the methods analyzed above except  PointSAGA~\citep{defazio2014saga, defazio2016simple} are also IFO algorithms.
From the results in Table~\ref{table-minimax-average} and the analysis in Section~\ref{sec:min},
we find that IFO algorithms are powerful enough for smooth functions.

\section{Framework of Construction}\label{sec:frame}
In this section, we introduce the framework of our construction to prove the lower bound for Problem~(\ref{prob:main}).
In Section~\ref{sec:frame:def},
we give the definition of the optimization complexity.
In Section~\ref{sec:frame:instance}, we construct the hard instances used to prove the lower bound and present some fundamental lemmas. Now we first highlight the key idea of our construction.

\paragraph{Key idea}
To construct the hard instance,
we partition the tridiagonal matrix in \citet{nesterov2013introductory} into $n$ groups and each component function is defined in terms of only one group.
Then the hard instance satisfies a variant of \textit{zero-chain} property: starting from the origin, only when a specific component is drawn, can we increase the nonzero elements of the current point by at most $2$.
And the number of PIFO calls required to draw this component obeys the geometric distribution.
Once we prove that we cannot obtain any $\eps$-suboptimal solution or $\eps$-stationary point unless we span all the dimensions, the complexity can be lower bounded by the concentration inequality of geometric distributions, i.e., Lemma~\ref{lem:geo}.
As a comparison, previous span-based constructions~\citep{lan2017optimal, zhou2019lower} partition the variable and the number of nonzero elements of the current point can increase no matter which component is drawn.
A more detailed analysis is deferred to Section~\ref{sec:min:instance}.

\subsection{Optimization Complexity}\label{sec:frame:def}
Before presenting the definition of the optimization complexity,
we first introduce the function class we consider.
Define
the primal function as $\phi_f(\vx) = \max_{\vy \in \fY} f(\vx, \vy)$ and the dual function as $\psi_f(\vy) = \min_{\vx \in \fX} f(\vx, \vy)$.

\paragraph{Function class} We develop lower bounds for PIFO algorithms that find a suboptimal solution or near stationary point of Problem~(\ref{prob:main}) in the following sets.

{ 
\begin{align*}
    {\fF}_{\mathrm{CC}}(R_x, R_y, L, \mu_x, \mu_y) = \bigg\{f(\vx, \vy) = \frac{1}{n} \sum_{i=1}^n f_i(\vx, \vy) ~\Big\vert~ f\colon \fX \times \fY \to \BR, \, \diam (\fX) \le 2 R_x,, \\
    \diam (\fY) \le 2 R_y, \, \{f_i\}_{i=1}^n \text{ is } L\text{-average smooth}, f \text{ is } (\mu_x, \mu_y) \text{-convex-concave } \bigg\}.
    \\
    {\fF}_{\mathrm{NCC}}(\Delta, L, \mu_x, \mu_y) = \bigg\{f(\vx, \vy) = \frac{1}{n} \sum_{i=1}^n f_i(\vx, \vy) ~\Big\vert~ f\colon \fX \times \fY \to \BR, \,  \phi(\vzero) - \inf_{\vx \in \fX} \phi(\vx) \le \Delta,\\
    \{f_i\}_{i=1}^n \text{ is } L\text{-average smooth}, f \text{ is } (-\mu_x, \mu_y)\text{-convex-concave } \bigg\}. 
\end{align*} }
We remark that for the second class, $\mu_x$ measures how nonconvex the function is. A natural upper bound of $\mu_x$ is $L$.
Moreover, we do not specify the dimensions of the feasible set.
That is to say, the two classes include  functions defined on $\fX \times \fY \subseteq \BR^{d_x} \times \BR^{d_y}$ with any positive integers $d_x$ and $d_y$. 

\paragraph{Optimization complexity}

Then we formally define the optimization complexity.

\begin{definition}\label{defn:complexity}
    For a function $f$, a PIFO algorithm $\fA$ and a tolerance $\eps > 0$, the number of queries to PIFO needed by $\fA$ to find an  $\eps$-suboptimal solution to Problem~(\ref{prob:main}) or an $\eps$-stationary point of $\phi_f(\vx)$ is defined as
    \begin{align*}
        T(\fA, f, \eps) = \begin{cases}
         \inf \left\{ T \in \BN ~\vert~ \E \phi_f(\vx_{\fA, T}) -  \E \psi_f(\vy_{\fA, T}) < \eps \right\}, 
        & \text{if } f \in \fF_{\mathrm{CC}}(R_x, R_y, L, \mu_x, \mu_y), \\
        \inf \left\{ T \in \BN ~\vert~ \E \norm{\nabla \phi_f(\vx_{\fA, T})} < \eps \right\}, 
        & \text{if } f \in \fF_{\mathrm{NCC}}(\Delta, L, \mu_x, \mu_y),
        \end{cases}
    \end{align*}
    where $(\vx_{\fA, T}, \vy_{\fA, T})$ is the point obtained by the algorithm $\fA$ at time-step $T-1$.
    The optimization complexity with respect to 
    the two function classes 
    is defined as \footnote{Our definition follows from \citet{carmon2017lower}.}
    \begin{align*}
        {\mathfrak{m}}^{\mathrm{CC}}(\eps, R_x, R_y, L, \mu_x, \mu_y) & \triangleq \inf_{\fA \in \mathscr{A}} \sup_{f \in 
        {\fF}_{\mathrm{CC}}(R_x, R_y, L, \mu_x, \mu_y)}  T(\fA, f, \eps). \\
        {\mathfrak{m}}^{\mathrm{NCC}}(\eps, \Delta, L, \mu_x, \mu_y) & \triangleq \inf_{\fA \in \mathscr{A}} \sup_{f \in 
        {\fF}_{\mathrm{NCC}}(\Delta, L, \mu_x, \mu_y)}  T(\fA, f, \eps).
    \end{align*}

\end{definition}
When $f$ is convex-concave, the functions we consider have a bounded feasible set and $L$-average smooth components.
By Sion's minimax theorem, the strong duality condition holds.
Then the primal-dual gap is a natural measurement of the optimality\footnote{When $f$ is strongly-convex-strongly-concave, the boundness of the feasible set is not necessary and we can also use the prima-dual gap at the initial point as the parameter to define the function class.
}.
Specially, if $f$ is strongly-convex-strongly-concave,
the saddle point $(\vx^*, \vy^*) $ is unique and the distance to the saddle point is also a measurement of the optimality.
And we have $\frac{\mu_x}{2} \norm{\vx - \vx^*}^2 + \frac{\mu_y}{2} \norm{\vy - \vy^*}^2 \le \phi_f (\vx) - \psi_f (\vy) \le \frac{L}{2} \norm{\vx - \vx^*}^2 + \frac{L}{2} \norm{\vy - \vy^*}^2$.
The results in Section~\ref{sec:minimax} show that the optimal methods have linear convergence rates in this case. Thus, the complexities w.r.t. the two measurements are equivalent up to log factors.
As for the nonconvex-strongly-concave case,
we aim to find the stationary point of the primal function and use the norm of the gradient of the primal function as the measurement.

Note that we use the number of PIFO calls to measure the complexity. 
We claim that the infrequent FO calls do not influence the order of this complexity. At each step, the FO is called with probability $q = \fO \left( \frac{1}{n} \right)$. Since the computation cost of each FO call is no larger than that of $n$ PIFO calls,  the total cost of PO calls is no larger than the order of the number of PIFO calls in expectation. Thus our definition of complexity is reasonable, due to that we usually 
ignore the influence of constants.

\subsection{
The Hard Instances
}\label{sec:frame:instance}
In this subsection, we construct the (unscaled) hard instances used to prove the lower bound.
The constructions for convex-concave case and the nonconvex-strongly-concave case are slightly different and presented in Sections \ref{sec:frame:instance:cc} and \ref{sec:frame:instance:ncc} respectively.
However, they are both based on the following class of matrices, which is also used in the proof of lower bounds in deterministic minimax optimization \citep{ouyang2018lower, zhang2019lower}:
\begin{align}\label{eq:mat:B}
    \mB(m, \omega, \zeta) =
    \begin{bmatrix}
        \omega  &       &       &       &       \\
        1       & -1    &       &       &       \\
                & 1     & -1    &       &       \\
                &       & \ddots&\ddots &       \\
                &       &       & 1     & -1    \\
                &       &       &       & \zeta
    \end{bmatrix}
    \in \BR^{(m+1) \times m}.
\end{align}
In fact, $\mB(m, \omega, \zeta)^\top \mB(m, \omega, \zeta)$ is the widely-used tridiagonal matrix in the analysis of lower bounds for convex optimization~\citep{nesterov2013introductory, lan2017optimal, zhou2019lower}.

For convenience, we denote the $l$-th row of the matrix $\mB(m, \omega, \zeta)$ by $\vb_{l-1}(m, \omega, \zeta)^{\top}$.
To construct a hard instance for the finite-sum optimization problem, we partition the row vectors of $\mB(m, \omega, \zeta)$ according to the index sets $\fL_i = \big\{ l: 0 \le l \le m, l \equiv i - 1~(\bmod~n) \big\}$.
The $i$-th component is constructed in terms of $\{ \vb_{l} (m, \omega, \zeta): l \in \fL_i \} $.
This way of partition is different from 
those used in 
\citet{lan2017optimal} and \citet{zhou2019lower} (a detailed comparison is deferred to Section~\ref{sec:min:instance}).
We find that the $\vb_{l} (m, \omega, \zeta)$ have at most two nonzero elements and the vectors whose indices lie in the same index sets are mutually orthogonal, as long as $n \ge 2$.

\subsubsection{Convex-Concave Case}\label{sec:frame:instance:cc}
The hard instance for the convex-concave case is constructed as
\begin{align}\label{prob:tilde:r}
    \min_{\vx \in \fX} \max_{\vy \in \fY} \, r^{\mathrm{CC}}(\vx, \vy; m, \zeta, {\vc}^{\mathrm{CC}})
    \triangleq \frac{1}{n} \sum_{i=1}^n r^{\mathrm{CC}}_i(\vx, \vy; m, \zeta, {\vc}^{\mathrm{CC}}),
\end{align}
where ${\vc}^{\mathrm{CC}} = (c^{\mathrm{CC}}_1, c^{\mathrm{CC}}_2)$, $\fX = \{\vx \in \BR^m: \norm{\vx} \le R_x\}, \fY = \{\vy \in \BR^m: \norm{\vy} \le R_y\}$ and
\begin{align*}
     & r^{\mathrm{CC}}_i(\vx, \vy; m, \zeta, {\vc}^{\mathrm{CC}}) \\
     & =
    \begin{cases}
        n \sum\limits_{l \in \fL_i} \vy^\top \ve_{l} \vb_{l}(m, 0, \zeta)^{\top} \vx
        + \frac{c^{\mathrm{CC}}_1}{2} \norm{\vx}^2
        - \frac{c^{\mathrm{CC}}_2}{2} \norm{\vy}^2
        - n \inner{\ve_1}{\vx},
        & \text{ for } i = 1, \\
        n \sum\limits_{l \in \fL_i} \vy^\top \ve_{l} \vb_{l}(m, 0, \zeta)^{\top} \vx
        + \frac{c^{\mathrm{CC}}_1}{2} \norm{\vx}^2
        - \frac{c^{\mathrm{CC}}_2}{2} \norm{\vy}^2,
        & \text{ for } i = 2, 3, \dots, n.
    \end{cases}
\end{align*}
Note that
$\vb_0(m, 0, \zeta) = \vzero$,
which implies that this hard instance is based on the last $m$ rows of $\mB(m, \omega, \zeta)$.
Then we can determine the smoothness and strong convexity coefficients of $r^{\mathrm{CC}}_i$ as follows.

\begin{proposition}\label{prop:convex-concave:base}
For $c^{\mathrm{CC}}_1, c^{\mathrm{CC}}_2 \ge 0$ and $0 \le \zeta \le \sqrt{2}$, we have that $r^{\mathrm{CC}}_i$ is $L$-smooth and $(c^{\mathrm{CC}}_1, c^{\mathrm{CC}}_2)$-convex-concave, and $\{ r^{\mathrm{CC}}_i \}_{i=1}^n$ is $L'$-average smooth,
where 
\[
L = \sqrt{4 n^2 + 2 \max \{ c^{\mathrm{CC}}_1, c^{\mathrm{CC}}_2 \}^2 }
\;\text{ and }
\; L' = \sqrt{8n + 2 \max \{ c^{\mathrm{CC}}_1, c^{\mathrm{CC}}_2 \}^2 }.
\]
\end{proposition}
We find if $\max\{c_1^\mathrm{CC}, c_2^{\mathrm{CC}} \} = \fO(\sqrt{n})$, then $L / L' = \Theta(\sqrt{n})$.

Define the subspaces $\{\fF_k\}_{k=0}^m$ as 
\begin{align}\label{eq:subspace}
\fF_k = \begin{cases}
\spn \{ \ve_1, \ve_{2}, \dots, \ve_{k}\}, & \text{for } 1 \le k \le m, \\
 \{\vzero\}, & \text{for } k=0.
\end{cases}
\end{align}
Now we show that the hard instance satisfies a variant of the \textit{zero-chain} property \citep{carmon2017lower}.

\begin{lemma}\label{lem:convex-concave:jump}
    Suppose that $n \ge 2$ and $\fF_{-1} = \fF_0$.
    Then for $(\vx, \vy) \in \fF_{k} \times \fF_{k-1}$ and $0 \le k < m$, we have that
    \begin{align*}
        \begin{pmatrix}
            \nabla_\vx r^{\mathrm{CC}}_i(\vx, \vy) \\
            - \nabla_\vy r^{\mathrm{CC}}_i(\vx, \vy)
        \end{pmatrix},
        \prox_{r^{\mathrm{CC}}_i}^{\gamma}(\vx, \vy) \in 
        \begin{cases}
            \fF_{k+1} \times \fF_{k}, & \text{ if } i \equiv k + 1\, (\bmod ~n), \\
            \fF_{k} \times \fF_{k-1}, & \text{ otherwise},
        \end{cases}
    \end{align*}
    where we omit the parameters of $r^{\mathrm{CC}}_i$ to simplify the presentation.
\end{lemma}
If the current point is $(\vx, \vy)$, the information brought by the PIFO call at $(\vx, \vy)$ will not increase 
the nonzero elements of $(\vx, \vy)$
unless a specific component function is drawn.
Moreover, if such a specific component is drawn, the increase 
is at most $2$.
This variant of \textit{zero-chain} property is also different from the conventional \textit{zero-chain} property in finite-sum minimization problems \citep{lan2017optimal, zhou2019lower}, where regardless of which component is drawn, the nonzero elements of the current point can increase. 
Such a difference comes from different ways of partitioning and ensures that our construction requires a lower dimension (see the analysis in Section \ref{sec:min:results}).
The proofs of Proposition \ref{prop:convex-concave:base} and Lemma \ref{lem:convex-concave:jump} are given in Appendix \ref{appendix:frame:cc}.

When we apply a PIFO algorithm $\fA$ to solve Problem~(\ref{prob:tilde:r}), Lemma \ref{lem:convex-concave:jump} implies that $\vx_t = \vy_t = \vzero$ will hold until algorithm $\fA$ draws the component $f_1$ or calls the FO. 
Then, for any $t < T_1 = \min_t \{t: i_t = 1 \mbox{ or } a_t = 1 \}$, we have $\vx_t, \vy_t \in \fF_0$ while $\vx_{T_1} \in \fF_1$ and $\vy_{T_1} \in \fF_0$. The value of $T_1$ can be regarded as the smallest integer such that $\vx_{T_1} \in \fF_1 \setminus \fF_0$ could hold. 
Similarly, for $T_1 \le t < T_2 = \min_t \{t > T_1: i_t = 2 \mbox{ or } a_t = 1 \}$ it holds that $\vx_{t} \in \fF_1$ and $\vy_{t} \in \fF_0$ while we can ensure that $\vx_{T_2} \in \fF_2$ and $\vy_{T_2} \in \fF_1$. Figure \ref{fig:jump} illustrates this optimization process. 

\setlength{\tabcolsep}{2pt}
\begin{figure}[ht]
    \centering
    \begin{tikzpicture}[node distance=70pt]
        \node[draw, rounded corners]                        (start)   { \footnotesize \begin{tabular}{cc}
        $\vx_{0}\in \fF_0$\\
        $\vy_{0}\in\fF_{0}$
        \end{tabular}};
        \node[draw, rounded corners, right=of start]        (middle)  {\footnotesize\begin{tabular}{cc}
        $\vx_{T_1}\in \fF_{1}$\\
        $\vy_{T_1}\in \fF_{0}$
        \end{tabular}};
        \node[draw, rounded corners, right=of middle]        (end)  {\footnotesize\begin{tabular}{cc}
        $\vx_{T_2}\in \fF_{2}$\\
        $\vy_{T_2}\in \fF_{1}$
        \end{tabular}};
        \node[draw, rounded corners, right=of end]        (t3)  {\footnotesize\begin{tabular}{cc}
        $\vx_{T_3}\in \fF_{3}$\\
        $\vy_{T_3}\in \fF_{2}$
        \end{tabular}};
        \coordinate[right=of t3, xshift=-30pt] (d1);
    
        \draw[->] (start) -> node{\begin{tiny}\begin{tabular}{cc}
        $\fA$ draws $r^{\mathrm{CC}}_1$ or \\ calls $h_f^{\mathrm{FO}}$ at step $T_1$
        \end{tabular}\end{tiny}} (middle);
        \draw[->] (middle) -> node{\begin{tiny}\begin{tabular}{cc}
        $\fA$ draws $r^{\mathrm{CC}}_2$ or \\ calls $h_f^{\mathrm{FO}}$ at step $T_2$
        \end{tabular}\end{tiny}} (end);
        \draw[->] (end) -> node{\begin{tiny}\begin{tabular}{cc}
        $\fA$ draws $r^{\mathrm{CC}}_3$ or \\ calls $h_f^{\mathrm{FO}}$ at step $T_3$
        \end{tabular}\end{tiny}} (t3);
        \draw[->] (t3) -> node[anchor=south]{$\dots$}(d1);
    \end{tikzpicture}\\[0.4cm]
    \caption{An illustration of the process of solving the Problem (\ref{prob:tilde:r}) with a PIFO algorithm $\fA$.}
    \label{fig:jump}
\end{figure}
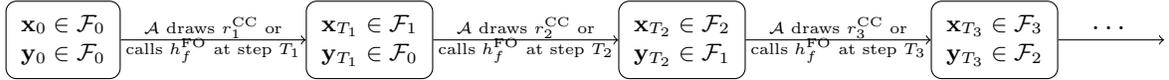
We can define $T_k$ to be the smallest integer such that $\vx_{T_k} \in \fF_k \setminus \fF_{k-1}$ and $\vy_{T_k} \in \fF_{k-1} \setminus \fF_{k-2}$ could hold.
The following corollary demonstrates that we can connect $T_k$
to geometrically distributed random variables.

\begin{corollary}
\label{coro:convex-concave:stopping-time}
Assume we employ a PIFO algorithm $\fA$ to solve Problem (\ref{prob:tilde:r}). 
Let
\begin{align}
\label{def:convex-concave:stopping-time}
    T_0 = 0,~\text{ and } 
    ~T_k = \min_t \{t: t > T_{k-1}, i_t \equiv k~(\bmod ~n) \mbox{ or } a_t = 1 \}~\text{ for } k \ge 1.
\end{align}
Then we have
\begin{align*}
    (\vx_t, \vy_t) \in \fF_{k-1} \times \fF_{k-2}, ~~~ \text{ for } t < T_k, k \ge 1.
\end{align*}
Moreover, the random variables $\{Y_k\}_{k \ge 1}$ such that $Y_k \triangleq T_k - T_{k-1}$ are mutually independent and $Y_k$ follows a geometric distribution with success probability $p_{k'} + q - p_{k'
} q$ where $k' \equiv k~(\bmod~n)$ and $l \in [n]$. 
\end{corollary}

The basic idea of our analysis is that we guarantee that the $\eps$-suboptimal solution 
of Problem (\ref{prob:tilde:r}) does not lie in $\fF_k  \times \fF_k $ for $k < m$
and assure that the PIFO algorithm extends the space $\spn\{ (\vx_0, \vy_0), (\vx_1, \vy_1),\dots, (\vx_t, \vy_t) \}$ slowly with $t$ increasing. 
By Corollary \ref{coro:convex-concave:stopping-time},
we know that 
$\spn\{ (\vx_0, \vy_0), (\vx_1, \vy_1), \dots, (\vx_{T_k-1}, \vy_{T_k-1}) \} \subseteq \fF_{k-1} \times \fF_{k-1}$.   
Hence, $T_k$ is the quantity that measures how  $\spn\{ (\vx_0, \vy_0), (\vx_1, \vy_1),\dots, (\vx_t, \vy_t) \}$ expands. 
Note that $T_k$ can be written as the sum of geometrically distributed random variables.
Recalling Lemma \ref{lem:geo}, 
we can obtain how many PIFO calls we need.

\begin{lemma}\label{lem:minimiax:base}
If $M$ satisfies $1 \le M < m$,
\begin{align}\label{eq:lem:base:condition}
\min_{\substack{
     \vx \in \fX \cap \fF_M  \\
     \vy \in \fY \cap \fF_M
}} \left( \max_{\vv \in \fY} r^{\mathrm{CC}}(\vx, \vv) - \min_{\vu \in \fX} r^{\mathrm{CC}}(\vu, \vy)\right) \ge 9\eps
\end{align}
and $N = \frac{n(M+1)}{4(1 + c_0)}$, then we have
\begin{align*}
    \min_{t \le N} \E \left(\max_{\vv \in \fY} r^{\mathrm{CC}}(\vx_t, \vv) - \min_{\vu \in \fX} r^{\mathrm{CC}}(\vu, \vy_t)\right)  \ge \eps.
\end{align*}
\end{lemma}
Note that rescaling will not influence the \textit{zero-chain} property. Thus Lemma \ref{lem:minimiax:base} still holds for any rescaled version of $r^{\mathrm{CC}}$. 
It remains to pick up the parameters carefully, obtain a condition of the form (\ref{eq:lem:base:condition}) and then estimate the order of $N$.
These steps depend on the specific problem and are deferred to Sections~\ref{sec:minimax:scsc} to \ref{sec:minimax:cc} and Appendices~\ref{appendix:minimax:scsc} to \ref{appendix:minimax:cc}.

The proofs of Corollary \ref{coro:convex-concave:stopping-time} and Lemma \ref{lem:minimiax:base} are given in Appendix \ref{appendix:frame:other}.

\subsubsection{Nonconvex-Strongly-Concave Case}\label{sec:frame:instance:ncc}

For the nonconvex-strongly-concave case, the hard instance is constructed as
\begin{align}\label{prob:hat:r}
    \min_{\vx \in \BR^m} \max_{\vy \in \BR^m} \, r^{\mathrm{NCC}}(\vx, \vy; m, \omega, {\vc}^{\mathrm{NCC}})
    \triangleq \frac{1}{n} \sum_{i=1}^n r^{\mathrm{NCC}}_i(\vx, \vy; m, \omega, {\vc}^{\mathrm{NCC}})
\end{align}
where ${\vc}^{\mathrm{NCC}} = (c^{\mathrm{NCC}}_1, c^{\mathrm{NCC}}_2, c^{\mathrm{NCC}}_3)$ and
\begin{align*}
     & r^{\mathrm{NCC}}_i(\vx, \vy; m, \omega, {\vc}^{\mathrm{NCC}}) \\
     & =
    \begin{cases}
        n \sum\limits_{l \in \fL_i} \vy^\top \ve_{l+1} \vb_{l}(m, \omega, 0)^{\top} \vx
        - \frac{c^{\mathrm{NCC}}_1}{2} \norm{\vy}^2  
        +  c^{\mathrm{NCC}}_2 \sum\limits_{i=1}^{m-1} \Gamma (c^{\mathrm{NCC}}_3  x_i)
        - n \inner{\ve_1}{\vy},
        \,
         \text{ for } i = 1, \\
        n \sum\limits_{l \in \fL_i} \vy^\top \ve_{l+1} \vb_{l}(m, \omega, 0)^{\top} \vx
        - \frac{c^{\mathrm{NCC}}_1}{2} \norm{\vy}^2  
        + c^{\mathrm{NCC}}_2 \sum\limits_{i=1}^{m-1} \Gamma (c^{\mathrm{NCC}}_3  x_i),
        \; \, \quad
        \text{ for } i = 2, 3, \dots, n.
    \end{cases}
\end{align*}
The nonconvex function $\Gamma: \BR \rightarrow \BR$ is
\begin{align*}
    \Gamma(x) \triangleq 120 \int_1^{x} \frac{t^2(t-1)}{1 + t^2} d t,
\end{align*}
which was introduced by \citet{carmon1711lower}.
Since $\vb_m(m, \omega, 0) = \vzero_m$, the vector $\ve_{m+1}$ will not appear in the definition of $r^{\mathrm{NCC}}$. Thus $r^{\mathrm{NCC}}$ is well-defined and only depends on the first $m$ rows of $\mB(m, \omega, \zeta)$.
We can determine the smoothness and strong convexity coefficients of $r^{\mathrm{NCC}}_i$ as follows.

\begin{proposition}\label{prop:nonconvex-strongly:base}
For $c^{\mathrm{NCC}}_1 \ge 0$, $c^{\mathrm{NCC}}_2, c^{\mathrm{NCC}}_3 > 0$ and $0 \le \omega \le \sqrt{2}$, we have that $r^{\mathrm{NCC}}_i$ is $L$-smooth and $\left( -45(\sqrt{3}-1) c^{\mathrm{NCC}}_2 (c^{\mathrm{NCC}}_3)^2, c^{\mathrm{NCC}}_1 \right)$-convex-concave, and $\{ r^{\mathrm{NCC}}_i \}_{i=1}^n$ is $L'$-average smooth, where
\[
L =  \sqrt{4n^2 + 2 (c^{\mathrm{NCC}}_1)^2} + 180 c^{\mathrm{NCC}}_2 (c^{\mathrm{NCC}}_3)^2 \;
\text{ and }\;
L' = 2\sqrt{4n + (c^{\mathrm{NCC}}_1)^2 + 16200 (c^{\mathrm{NCC}}_2)^2 (c^{\mathrm{NCC}}_3)^4}.
\]
\end{proposition}
We find if $\max\{c_1^\mathrm{NCC}, c_2^{\mathrm{NCC}} (c_3^\mathrm{NCC})^2 \} = \fO(\sqrt{n})$, then $L / L' = \Theta(\sqrt{n})$.

The next lemma shows that
the $r_i^{\mathrm{NCC}}$ share the similar \textit{zero-chain} property as Lemma~\ref{lem:convex-concave:jump}.
\begin{lemma}\label{lem:nonconvex-strongly:jump}
Suppose that $n \ge 2$, $c^{\mathrm{NCC}}_2, c^{\mathrm{NCC}}_3 > 0$ and $\gamma <  \frac{\sqrt{2} + 1 }{60 c^{\mathrm{NCC}}_2 (c^{\mathrm{NCC}}_3)^2}$. 
If $(\vx, \vy) \in \fF_k \times \fF_k$ and $0 \le k < m-1$, we have that
\begin{align*}
    \begin{pmatrix}
        \nabla_x r^{\mathrm{NCC}}_i(\vx, \vy) \\
        - \nabla_y r^{\mathrm{NCC}}_i(\vx, \vy) 
    \end{pmatrix},
    \prox_{r^{\mathrm{NCC}}_i}^{\gamma}(\vx, \vy) \in \begin{cases}
        \fF_{k+1} \times \fF_{k+1}, & \text{ if } i \equiv k + 1\, (\bmod ~n), \\
        \fF_{k} \times \fF_{k}, & \text{ otherwise},
    \end{cases}
\end{align*}
where we omit the parameters of $r^{\mathrm{NCC}}_i$ to simplify the presentation. 
\end{lemma}
The proofs of Proposition \ref{prop:nonconvex-strongly:base} and Lemma \ref{lem:nonconvex-strongly:jump} are given in Appendix~\ref{appendix:frame:ncc}.

It is worth emphasizing that the assumption on $\gamma$ naturally holds. 
Recall that the choice of $\gamma$ should satisfy that $r_i(\vu, \vy) + \frac{1}{2 \gamma} \norm{\vx - \vu}^2 - \frac{1}{2 \gamma} \norm{\vy - \vv}^2$ is  convex-concave in $(\vu, \vv)$.
Proposition \ref{prop:nonconvex-strongly:base} implies that we must have $\gamma \le \frac{1}{ 45 (\sqrt{3} - 1) c_2^{\mathrm{NCC}} (c_3^\mathrm{NCC} )^2 } \le \frac{\sqrt{2} + 1}{ 60 c_2^{\mathrm{NCC}} (c_3^\mathrm{NCC} )^2 }$.

When we apply a PIFO algorithm to solve Problem~(\ref{prob:hat:r}), the optimization process is similar to the process related to Problem~(\ref{prob:tilde:r}). We demonstrate the optimization process in Figure \ref{fig:nonconvex-strongly:jump} and present a formal statement in Corollary \ref{coro:nonconvex-strongly:stopping-time}.

\begin{figure}[ht]
    \centering
    \begin{tikzpicture}[node distance=70pt]
        \node[draw, rounded corners]                        (start)   { \footnotesize \begin{tabular}{cc}
        $\vx_{0}\in \fF_0$\\
        $\vy_{0}\in\fF_{0}$
        \end{tabular}};
        \node[draw, rounded corners, right=of start]        (middle)  { \footnotesize \begin{tabular}{cc}
        $\vx_{T_1}\in \fF_{1}$\\
        $\vy_{T_1}\in \fF_{1}$
        \end{tabular}};
        \node[draw, rounded corners, right=of middle]        (end)  { \footnotesize \begin{tabular}{cc}
        $\vx_{T_2}\in \fF_{2}$\\
        $\vy_{T_2}\in \fF_{2}$
        \end{tabular}};
        \node[draw, rounded corners, right=of end]        (t3)  {\footnotesize\begin{tabular}{cc}
        $\vx_{T_3}\in \fF_{3}$\\
        $\vy_{T_3}\in \fF_{3}$
        \end{tabular}};
        \coordinate[right=of t3, xshift=-30pt] (d1);
    
        \draw[->] (start) -> node{\begin{tiny}\begin{tabular}{cc}
        $\fA$ draws $r^{\mathrm{NCC}}_1$ or \\ calls $h_f^{\mathrm{FO}}$ at step $T_1$
        \end{tabular}\end{tiny}} (middle);
        \draw[->] (middle) -> node{\begin{tiny}\begin{tabular}{cc}
        $\fA$ draws $r^{\mathrm{NCC}}_2$ or \\ calls $h_f^{\mathrm{FO}}$ at step $T_2$
        \end{tabular}\end{tiny}} (end);
        \draw[->] (end) -> node{\begin{tiny}\begin{tabular}{cc}
        $\fA$ draws $r^{\mathrm{NCC}}_3$ or \\ calls $h_f^{\mathrm{FO}}$ at step $T_3$
        \end{tabular}\end{tiny}} (t3);
        \draw[->] (t3) -> node[anchor=south]{$\dots$}(d1);
    \end{tikzpicture}\\[0.4cm]
    \caption{An illustration of the process of solving the Problem (\ref{prob:hat:r}) with a PIFO algorithm $\fA$.}
    \label{fig:nonconvex-strongly:jump}
\end{figure}
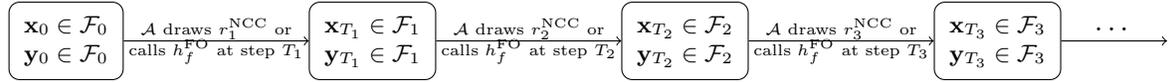

\begin{corollary}
\label{coro:nonconvex-strongly:stopping-time}
Assume we employ a PIFO algorithm $\fA$ to solve Problem (\ref{prob:hat:r}). 
Let
\begin{align*}
    T_0 = 0,~\text{ and } 
    ~T_k = \min_t \{t: t > T_{k-1}, i_t \equiv k~(\bmod ~n) \mbox{ or } a_t = 1 \}~\text{ for } k \ge 1.
\end{align*}
Then we have
\begin{align*}
    (\vx_t, \vy_t) \in \fF_{k-1} \times \fF_{k-1}, ~~~ \text{ for } t < T_k, k \ge 1.
\end{align*}
Moreover, the random variables $\{Y_k\}_{k \ge 1}$ such that $Y_k \triangleq T_k - T_{k-1}$ are mutual independent and $Y_k$ follows a geometric distribution with success probability $p_{k'} + q - p_{k'} q$ where $k' \equiv k~(\bmod~n)$ and $l \in [n]$. 
\end{corollary}
The proof of Corollary \ref{coro:nonconvex-strongly:stopping-time} is similar to that of Corollary \ref{coro:convex-concave:stopping-time}.
Furthermore, the prime-dual gap in Lemma \ref{lem:minimiax:base} can be replaced with the gradient norm of the primal function in the nonconvex-strongly-concave case. 
\begin{lemma}\label{lem:minimax:ncsc:base}
Let $\phi_{r^{\mathrm{NCC}}}(\vx) \triangleq \max_{\vy \in \BR^m} r^{\mathrm{NCC}}(\vx, \vy)$. 
If $M$ satisfies $1 \le M < m$ and 
\begin{align}\label{eq:lem:ncsc:base:condition}
\min_{\vx \in \fF_{M}} \norm{ \nabla  \phi_{r^{\mathrm{NCC}}}(\vx) } \ge 9\eps
\end{align}
and $N = \frac{n(M+1)}{4(1 + c_0)}$, then we have
\begin{align*}
    \min_{t \le N} \E \norm{ \nabla \phi_{r^{\mathrm{NCC}}}(\vx_t) } \ge \eps.
\end{align*}
\end{lemma}
Lemma \ref{lem:minimax:ncsc:base} also holds for any rescaled version of $r^{\mathrm{NCC}}$. 
It remains to pick up the parameters carefully, obtain a condition of the form (\ref{eq:lem:ncsc:base:condition}) and then estimate the order of $N$.
The details are deferred to Section~\ref{sec:minimax:ncsc} and Appendix~\ref{appendix:minimax:ncsc}.
\section{Lower Complexity Bounds for the Minimax Problems}
\label{sec:minimax}
In this section, 
we focus on the minimax problem~(\ref{prob:main}), which is restated as follows.
\begin{align*}\label{prob:minimax}
    \min_{\vx \in \fX} \max_{\vy \in \fY} f(\vx, \vy) = \frac{1}{n}\sum_{i=1}^n f_i(\vx,\vy).
\end{align*}
We assume that the function class $\{ f_i(\vx, \vy)\}_{i=1}^n$ is $L$-average smooth, and the feasible sets $\fX$ and $\fY$ are closed and convex.
In addition, $f(\vx,\vy)$ is convex in $\vx$ and concave in $\vy$ or  $f(\vx,\vy)$ is nonconvex in $\vx$ and strongly-concave in $\vy$.
The lower bound results are shown in Section \ref{sec:minimax:mainresults}.
The detailed constructions for different cases are shown in Sections \ref{sec:minimax:scsc} to \ref{sec:minimax:ncsc}.
Finally, in Section \ref{sec:minimax:smooth},
we consider the more constrained case where each $f_i$ is $L$-smooth and briefly introduce the results.

\subsection{Main Results}\label{sec:minimax:mainresults}
Recall that the comparison of the upper and lower bounds is already shown in Table~\ref{table-minimax-average}.
In this subsection, we present the formal statements of our lower bounds and give some interpretation.
We emphasize that the methods in \citet{luo2021near, zhang2021complexity} are just IFO algorithms from the analysis in Section~\ref{sec:alg:our}, which implies PIFO oracles are not much more powerful than IFO oracles.

We start with the case where the objective function $f$ is $\mu_x$-strongly-convex in $\vx$ and $\mu_y$-strongly-concave in $\vy$. Define the condition numbers $\kappa_x \triangleq L / \mu_x $ and $\kappa_y \triangleq L / \mu_y $. Without loss of generality, we assume $\mu_x \le \mu_y$. According to the relationship between $\kappa_x, \kappa_y$ and $n$, we can classify the problem into three cases: (a) $f$ is extremely ill-conditioned w.r.t. both $\vx$ and $\vy$, i.e., $\kappa_x, \kappa_y = \Omega (\sqrt{n})$; (b) $f$ is only extremely ill-conditioned w.r.t. $\vx$, i.e., $\kappa_x = \Omega(\sqrt{n}), \kappa_y = \fO(\sqrt{n})$;
(c) $f$ is relatively well-conditioned w.r.t. both $\vx$ and $\vy$, i.e., $\kappa_x, \kappa_y = \fO (\sqrt{n})$.
For the three cases, we can prove different lower bounds as follows.
\begin{theorem}\label{thm:average:strongly-strongly}
Let $n \ge 4$ be a positive integer and 
$L, \mu_x, \mu_y, R_x, R_y, \eps$ be positive parameters.
Assume additionally that
$\kappa_x \ge \kappa_y \ge 2$
and $\eps \le  \min \left\{ \frac{ n \mu_x R_x^2}{800 \kappa_x \kappa_y}, \frac{\mu_x R_x^2}{720}, \frac{\mu_y R_y^2}{800} \right\}$. Then we have
    \begin{align*}
        {\mathfrak{m}}^{\mathrm{CC}}(\eps, R_x, R_y, L, \mu_x, \mu_y) =
        \begin{cases}
            \Omega\left(\left(n {+} \sqrt{ \kappa_x \kappa_y n } \right)\log\left(1/\eps \right)\right), &\text{ for } \kappa_x, \kappa_y = \Omega(\sqrt{n}), \\
            \Omega\left(\left(n {+} n^{3/4} \sqrt{\kappa_x }\right)\log\left(1/\eps \right)\right), &\text{ for } \kappa_x = \Omega(\sqrt{n}), \kappa_y = \fO(\sqrt{n}), \\
            \Omega\left( n \right), &\text{ for } \kappa_x, \kappa_y = \fO(\sqrt{n}).
        \end{cases}
    \end{align*}
\end{theorem}
We mainly focus on the first two cases where at least one condition number is of the order $\Omega(\sqrt{n})$. Then the lower bound can be summarized as $\Omega ( \sqrt{n (\sqrt{n} + \kappa_x) (\sqrt{n} + \kappa_y) }  \log (1 / \eps) ) $, as shown in Table \ref{table-minimax-average}.

Some works focus on the balanced case $\kappa_x = \kappa_y$.
For example,
the upper bound of Accelerated SVRG/SAGA \citep{palaniappan2016stochastic} is $\fO\left(\left(n +  \frac{\sqrt{n} L}{\min\{\mu_x, \mu_y\}}\right) \log(1/\eps) \right)$. 
L-SVRE~\citep{alacaoglu2021stochastic}
also achieves the same upper bound\footnote{The setting in Section 4.3 of \citet{alacaoglu2021stochastic} is slightly different from ours here. However, the proof of their result can be adapted to strongly-convex-strongly-concave cases.}.
At least for the balanced case, their upper bounds nearly match our lower bound.
However, for the unbalanced case, there still exists a gap.
\citet{luo2021near}
focus on the unbalanced case.
They employ the catalyst technique to accelate L-SVRE and propose the method AL-SVRE, which achieves the upper bound $\tilde{\fO} ( \sqrt{ n (\sqrt{n} + \kappa_x) (\sqrt{n} + \kappa_y) } \log ({1}/{\eps} )  )$.
This bound nearly matches our lower bound for the unbalanced case up to log factors.

Then we consider the lower bound when the objective function is not strongly-convex in $\vx$, i.e., $\mu_x=0$. In this case, only the condition number w.r.t. $\vy$ is well-defined.
According to the relationship between $\kappa_y$ and $\sqrt{n}$, we can also split the problem into two cases: (a) $f$ is extremely ill-conditioned w.r.t. $\vy$, i.e., $\kappa_y = \Omega(\sqrt{n})$;
(b) $f$ is relatively well-conditioned w.r.t. $\vy$, i.e., $\kappa_y = \fO (\sqrt{n})$. We can prove the lower bounds as follows.
\begin{theorem}\label{thm:average:convex-strongly}
Let $n \ge 4$ be a positive integer and 
$L, \mu_y, R_x, R_y, \eps$ be positive parameters. Assume additionally that
$
\kappa_y
\ge 2$ and $\eps \le \min \left\{ \frac{L R_x^2}{4}, \frac{\mu_y R_y^2}{36} \right\}$.
Then we have
\begin{align*}
    {\mathfrak{m}}^{\mathrm{CC}}(\eps, R_x, R_y, L, 0, \mu_y)= \!
    \begin{cases}
    \Omega\Big(n {+} R_x n^{3/4}\! \sqrt{ \frac{L}{\eps} } {+} R_x \sqrt{ \frac{n L \kappa_y }{\eps} } {+} n^{3/4}\! \sqrt{ \kappa_y } \log ( \frac{1}{\eps} ) \Big), \  \text{for } \kappa_y {=} \Omega(\sqrt{n}), \\
    \Omega\Big(n {+} R_x n^{3/4}\! \sqrt{ \frac{L}{\eps} }  {+} R_x \sqrt{ \frac{n L \kappa_y}{\eps} } \Big), 
    \qquad \qquad \qquad \quad \ 
    \text{for } \kappa_y {=} \fO(\sqrt{n}).
    \end{cases}
    \end{align*}
\end{theorem}
For both cases, the leading term w.r.t. $\eps$ is of the order $\Omega(\sqrt{1 / \eps})$ and
the only difference between the two bounds
is the term $\Omega (n^{3/4} \sqrt{\kappa_y} \log (\frac{1}{\eps}) ) $, which is usually much smaller than the $\Omega(\sqrt{1 / \eps})$ term,
especially when $\eps$ is small. The upper bound of AL-SVRE~\citep{luo2021near} for this case is $\fO \Big(  \big( n + R_x n^{3/4} \sqrt{ \frac{L}{\eps} } + R_x \sqrt{ \frac{n L \kappa_y}{\eps} } + n^{3/4} \sqrt{\kappa_y} \big) \log (\frac{1}{\eps}) \Big)$, which nearly matches our lower bound up to log factors.

For the general convex-concave case where $\mu_x = \mu_y = 0$, we have the following lower bound.
\begin{theorem}\label{thm:average:convex-concave}
Let $n \ge 2$ be a positive integer and 
$L, R_x, R_y, \eps$ be positive parameters. Assume additionally that $\eps \le \frac{L}{4} \min\{ R_x^2, R_y^2 \}$. Then we have
    \begin{align*}
        {\mathfrak{m}}^{\mathrm{CC}}(\eps, R_x, R_y, L, 0, 0)=
        \Omega\left(n {+} \frac{ \sqrt{n} L R_x R_y }{\eps} {+} (R_x + R_y) n^{3/4} \sqrt{ \frac{L}{\eps} } \right).
    \end{align*}
\end{theorem}
The leading term w.r.t. $\eps$ is of the order $\Omega(1/\eps)$.
If $\eps = \fO \left( \frac{ L R_x^2 R_y^2 }{ \sqrt{n} (R_x + R_y)^2 } \right)$, our lower bound is $\Omega \left( n + \frac{\sqrt{n} L R_x R_y }{\eps} \right) $, which matches the upper bound $\fO \left( n + \frac{ \sqrt{n} L (R_x^2 + R_y^2) }{\eps} \right) $ of \citet{alacaoglu2021stochastic} in terms of $n$, $L$ and $\eps$. The upper bound of AL-SVRE~\citep{luo2021near} for this case is $\fO \Big( \big( n + \frac{\sqrt{n} L R_x R_y}{\eps} + (R_x + R_y) n^{3/4} \sqrt{\frac{L}{\eps}} \big)   \log (\frac{1}{\eps}) \Big) $, which nearly matches our lower bound up to log factors.

Finally, we give the lower bound when the objective function is nonconvex in $\vx$ but strongly-concave in $\vy$.

\begin{theorem}\label{thm:average:nonconvex-strongly}
Let $n \ge 2$ be a positive integer and 
$L, \mu_x, \mu_y, R_x, R_y, \eps$ be positive parameters. Assume additionally that $\eps^2 \le \frac{\Delta L^2 \alpha}{435456 n \mu_y}$, where $\alpha = \min \left\{1, \frac{128(\sqrt{3} + 1)n\mu_x \mu_y}{45 L^2}, \frac{32n\mu_y}{135 L} \right\}$. Then we have
    \begin{align*}
        {\mathfrak{m}}^{\mathrm{NCC}}(\eps, \Delta, L, \mu_x, \mu_y)=
        \Omega \left( n + \frac{\Delta L^2 \sqrt{\alpha}}{\mu_y \eps^2}\right).
    \end{align*}
For $\kappa_y = L / \mu_y \ge 32n / 135$, we have
\[
\Omega \left( n + \frac{\Delta L^2 \sqrt{\alpha}}{\mu_y \eps^2}\right)
= \Omega \left( n + \frac{\Delta L \sqrt{n}
}{\eps^2} \min \left\{ \sqrt{\kappa_y}, \sqrt{ \frac{\mu_x}{\mu_y} } \right\} \right).
\]
\end{theorem}
We mainly focus on the ill-conditioned setting $\kappa_y = \Omega(n)$, where the lower bound has a more concise expression.
Recall that $\mu_x$ measures the nonconvexity of function. When $L$ is fixed, we must have $\mu_x \le L$. If we are uninterested in the dependence of the lower bound on $\mu_x$, then we can consider the largest function class $\fF_{\mathrm{NCC}} (\Delta, L, L, \mu_y) $. which corresponds to the complexity
$  {\mathfrak{m}}^{\mathrm{NCC}}(\eps, \Delta, L, L, \mu_y)=
\Omega \left( n + \frac{\Delta L \sqrt{ n \kappa_y }}{ \eps^2}\right)$\footnote{A concurrent work by \citet{zhang2021complexity} obtains a similar lower bound. }, as shown in Table~\ref{table-minimax-average}.

As for the upper bound, \citet{luo2020stochastic} propose the method SREDA and establish the upper bound  
$\fO \left(n \log (\kappa_y / \eps) +  L \kappa_y^2 \sqrt{n} \eps^{-2} \right) $ for $n \ge \kappa_y^2 $ and $ \fO \left( (\kappa_y^2 + \kappa_y n) L \eps^{-2} \right)$ for $n < \kappa_y^2 $.
\citet{zhang2021complexity} propose Catalyst-SVRG/SAGA and obtain the upper bound $\tilde{\fO} \left( (n + n^{3/4} \sqrt{\kappa_y} ) \Delta L \eps^{-2} \right) $.
When $n \le \kappa^4$, the upper bound of \citet{zhang2021complexity} is better; otherwise, the upper bound of \citet{luo2020stochastic} is better.
Since we focus on the ill-conditioned setting, the upper and lower bounds nearly match in terms of $\kappa_y$. And there is still a $n^{1/4}$ gap in terms of $n$.
\subsection{Construction for the Strongly-Convex-Strongly-Concave Case}\label{sec:minimax:scsc}
In this subsection, 
we give the exact forms of the hard instance 
when the objective function is strongly-convex in $\vx$ and strongly-concave in $\vy$.
We still assume $\mu_x \le \mu_y$.
Then we have $\kappa_y \le \kappa_x$.
This means that the max part has a smaller condition number and is easier to solve.
According to the magnitude of $\kappa_x$ and $\kappa_y$,
the construction can be divided into three cases.

\paragraph{Case 1: $\kappa_x, \kappa_y = \Omega(\sqrt{n})$.}

When both condition numbers are
no smaller than $\Theta (\sqrt{n})$, 
the analysis depends on the following construction.

\begin{definition}\label{defn:scsc}
For fixed $L, \mu_x, \mu_y, R_x, R_y$ and $n$ such that $\mu_x \le \mu_y$, 
$\kappa_x \ge \kappa_y \ge 2$
we define $f_{\mathrm{SCSC}, i} : \BR^m \times \BR^m \rightarrow \BR$ as follows
\begin{align*}
    f_{\mathrm{SCSC}, i} (\vx, \vy) = \lambda\, r^{\mathrm{CC}}_i\left(\vx/\beta, \vy/\beta; m, \sqrt{\frac{2}{\alpha + 1}}, {\vc}^{\mathrm{SCSC}} \right), \text{ for } 1 \le i \le n,
\end{align*}
where 
{ \small
\begin{align*}
    \alpha & = 
    \sqrt{ \frac{ \left( \kappa_y - 2 / \kappa_y \right) \kappa_x  }{ 2n } + 1}, \quad
    {\vc}^{\mathrm{SCSC}}  = \left( 
    \frac{2 \kappa_y }{ \kappa_x } \sqrt{ \frac{ 2n }{ \kappa_y^2 - 2 } },\,
    2 \sqrt { \frac{ 2 n }{ \kappa_y^2 - 2 } } 
    \right), \\
    \beta & = \min \left\{ 
    2 R_x \sqrt{ \frac{ 2 \alpha n }{ \kappa_x^2 (1 - 2 / \kappa_y^2)  }}
    ,\, 
    \frac{ 4 R_x }{ \alpha + 1 } \sqrt{ \frac{ \alpha n }{ \kappa_x^2 (1 - 2 / \kappa_y^2) } }
    ,\,
    \frac{ \sqrt{ 2 \alpha} R_y  }{ \alpha - 1 }
    \right\} \text{  and  }
    \lambda  = 
    \frac{ \beta^2 }{ 2 } \sqrt{ \frac{ L^2 - 2 \mu_y^2 }{2n} }.
\end{align*} }
Consider the minimax problem 
\begin{align}
    \min_{\vx \in \fX}\max_{\vy \in \fY} f_{\mathrm{SCSC}}(\vx, \vy) \triangleq \frac{1}{n}\sum_{i=1}^n f_{\mathrm{SCSC},i}(\vx, \vy). \label{prob:scsc}\\[-0.8cm]\nonumber
\end{align}
where $\fX = \{ \vx \in \BR^m : \norm{\vx} \le R_x  \}$ and
$\fY = \{ \vy \in \BR^m : \norm{\vy} \le R_y  \}$.
Define $\phi_{ \mathrm{SCSC} } (\vx) = \max_{\vy \in \fY} f_{\mathrm{SCSC}} (\vx, \vy)$ and
$\psi_{ \mathrm{SCSC} } (\vy) = \min_{\vx \in \fX} f_{\mathrm{SCSC}} (\vx, \vy) $.
\end{definition}
One can check that $f_{\mathrm{SCSC}}$ belongs to $\fF_{\mathrm{CC}} (R_x, R_y, L, \mu_x, \mu_y)$ and satisfies a condition of the form (\ref{eq:lem:base:condition}) (please see Proposition \ref{prop:strongly-strongly} in Appendix \ref{appendix:minimax:scsc}).
Then we can prove the lower bound of the complexity for finding 
$\eps$-suboptimal point
of Problem~\ref{prob:scsc}) by PIFO algorithms.

\begin{theorem}\label{thm:scsc:example}
Consider the minimax problem~\ref{prob:scsc}) and $\eps>0$.
Let 
$\alpha = \sqrt{ \frac{ \left( \kappa_y - 2 / \kappa_y \right) \kappa_x  }{ 2n } + 1}$.
Suppose that 
{ \small
\begin{align*}
    & n \ge 2, \, \kappa_x \ge \kappa_y \ge \sqrt{2n + 2},\ 
    \eps \le \frac{1}{800} \min \left\{ \frac{ n \mu_x R_x^2}{ \kappa_x \kappa_y}, \mu_y R_y^2 \right\} , \\
    & \text{and } 
    m = \left\lfloor \frac{\alpha}{4}\log \left(\frac{ \max  \left\{ \mu_x R_x^2, \mu_y R_y^2 \right\} }{9\eps}\right) \right\rfloor + 1.
\end{align*} }
In order to find $(\hat{\vx}, \hat{\vy}) \in \fX \times \fY$ such that $ \E \phi_{\mathrm{SCSC}} (\hat{\vx}) - \E \psi_{\mathrm{SCSC}} (\hat{\vy}) < \eps$, PIFO algorithm $\fA$ needs at least $N$ queries, where
\begin{align*}
    N =
    \Omega\left(\left(n + \sqrt{n \kappa_x \kappa_y}\right)\log\left( \frac{1}{\eps} \right)\right).
\end{align*}
\end{theorem}
The proof of Theorem \ref{thm:scsc:example} is deferred to Appendix \ref{appendix:minimax:scsc}.

\paragraph{Case 2: $\kappa_x = \Omega(\sqrt{n})$, $\kappa_y = \fO(\sqrt{n})$.}
When only $\kappa_y$ is 
no smaller than $\Theta(\sqrt{n})$,
the lower bound is characterized by the following theorem.

\begin{theorem}\label{thm:scsc:example:2}
For any $L, \mu_x, \mu_y, n, R_x, R_y, \eps$ such that $n \ge 4 $,
{ \small
\begin{align*}
    & n \ge 4, \, \kappa_x  \ge \sqrt{2n + 2} \ge \kappa_y \ge 2,\ 
    \eps \le \frac{1}{720} \mu_x R_x^2 , \ 
    \tilde{L} = \sqrt{ n (L^2 - \mu_x^2) / 2 - \mu_x^2 },
    \\ 
    & \text{and } 
    m = \left\lfloor \frac{1}{4} \left( \sqrt{ \frac{ 2( \tilde{L} / \mu_x -1) }{n} + 1 } \right) \log \left(\frac{\mu_x R_x^2  }{9\eps}\right) \right\rfloor + 1,
\end{align*} }
there exist n functions $\{ f_i : \BR^m \times \BR^m \rightarrow \BR \}_{i=1}^n$ such that 
$f = \frac{1}{n} \sum_{i=1}^n f_i \in \fF_{\mathrm{CC}} (R_x, R_y, L, \mu_x, \mu_y)$.
Let $\fX = \{ \vx \in \BR^m : \norm{\vx} \le R_x  \}$ and
$\fY = \{ \vy \in \BR^m : \norm{\vy} \le R_y  \}$. 
In order to find $(\hat{\vx}, \hat{\vy}) \in \fX \times \fY$ such that $ \E \max_{\vy \in \fY} f (\hat{\vx}, \vy) - \E \min_{\vx \in \fX} f (\vx ,\hat{\vy}) < \eps$, PIFO algorithm $\fA$ needs at least $N$ queries, where
    $N =
    \Omega\left(\left(n + n^{3/4} \sqrt{ \kappa_x}\right)\log\left( \frac{1}{\eps} \right)\right).$
\end{theorem}
We find that $\kappa_y$ does not appear in the lower bound.
In fact, since $\kappa_y$ is relatively small,
the max part is easier to solve than the min part and the min part becomes the main obstacle.
To construct the hard instance, it suffices to consider the separable function of the form $f(\vx, \vy) = f_x(\vx) - f_y(\vy) $ where $f_x$ is the hard instance used for finite-sum minimization problems and $f_y (\vy) = \frac{\mu_y}{2} \norm{\vy}^2$.
For the details, see Appendix \ref{appendix:minimax:scsc}.

\paragraph{Case 3: $\kappa_x, \kappa_y = \fO(\sqrt{n})$.}

When both the condition numbers are relatively small, the lower bound is $\Omega (n)$, which means that the number of component functions becomes the main obstacle.
\begin{lemma}\label{lem:scsc:example:3}
For any $L, \mu_x, \mu_y, n, R_x, R_y, \eps$ such that $n \ge 2$, $L \ge \mu_x$, $L \ge \mu_y$ and $\eps \le \frac{1}{4} L R_x^2$,
there exist n functions $\{ f_i : \BR \times \BR \rightarrow \BR \}_{i=1}^n$ such that 
$f = \frac{1}{n} \sum_{i=1}^n f_i \in \fF_{\mathrm{CC}} (R_x, R_y, L, \mu_x, \mu_y)$.
Let $\fX = \{ x \in \BR : |x| \le R_x  \}$ and
$\fY = \{ y \in \BR : |y| \le R_y  \}$. 
In order to find $(\hat{x}, \hat{y}) \in \fX \times \fY$ such that $ \E \max_{y \in \fY} f (\hat{x}, y) - \E \min_{x \in \fX} f (x ,\hat{y}) < \eps$, PIFO algorithm $\fA$ needs at least $N = \Omega(n)$ queries.
\end{lemma}
This bound is trivial in some sense, since
we usually need to compute the full gradient at least once, whose complexity is of the order $\Omega(n)$.
The proof is also deferred to Appendix~\ref{appendix:minimax:scsc}.
Combining Theorems \ref{thm:scsc:example}, \ref{thm:scsc:example:2} and Lemma \ref{lem:scsc:example:3}, we can obtain Theorem \ref{thm:average:strongly-strongly}.
\subsection{Construction for the Convex-Strongly-Concave Case}\label{sec:minimax:csc}
In this subsection, we construct the hard instance when $f$ is convex in $\vx$ and strongly-concave in $\vy$.
The condition number $\kappa_y$ is still well-defined.
Our analysis is based on the following functions.
\begin{definition}\label{defn:csc}
For fixed $L, \mu_y, n, R_x, R_y$ such that $
\kappa_y \ge 2$, 
we define $f_{\mathrm{CSC}, i} : \BR^m \times \BR^m \rightarrow \BR$ as follows
\begin{align*}
    f_{\mathrm{CSC}, i} (\vx, \vy) = \lambda \, r^{\mathrm{CC}}_i \left( \vx / \beta, \vy / \beta;  m, 1, {\vc}^{\mathrm{CSC}} \right),
\end{align*}
where
{ \small
\begin{align*}
    {\vc}^{\mathrm{CSC}}  = \left( 0, 
    2 \sqrt{ \frac{2n}{ 
    \kappa_y^2
    - 2 } }
    \right), \ 
    \beta  = \min \left\{ 
    \frac{R_x \sqrt{ \frac{
    \kappa_y^2
    - 2}{2n} } }{2 (m+1)^{3/2} } 
    , \frac{R_y}{\sqrt{m}} \right\} \  \text{and} \ \,
    \lambda  = 
    \frac{\beta^2}{2} \sqrt{ \frac{L^2 - 2 \mu_y^2}{2n} }
    .
\end{align*} }
Consider the minimax problem
\begin{equation}
    \min_{\vx \in \fX}\max_{\vy \in \fY} f_{\mathrm{CSC}}(\vx, \vy) \triangleq \frac{1}{n}\sum_{i=1}^n f_{\mathrm{CSC},i}(\vx, \vy), \label{prob:csc}
\end{equation}
where $\fX = \{ \vx \in \BR^m : \norm{\vx} \le R_x  \}$ and
$\fY = \{ \vy \in \BR^m : \norm{\vy} \le R_y  \}$.
Define $\phi_{ \mathrm{CSC} } (\vx) = \max_{\vy \in \fY} f_{\mathrm{CSC}} (\vx, \vy)$ and
$\psi_{ \mathrm{CSC} } (\vy) = \min_{\vx \in \fX} f_{\mathrm{CSC}} (\vx, \vy) $.
\end{definition}
One can check that $f_{\mathrm{CSC}}$ belongs to $\fF_{\mathrm{CC}} (R_x, R_y, L, 0, \mu_y)$ and satisfies a condition of the form (\ref{eq:lem:base:condition}) (please see Proposition \ref{prop:convex-strongly} in Appendix \ref{appendix:minimax:csc}).
Then we can prove the lower bound of the complexity for finding 
$\eps$-suboptimal point
of Problem (\ref{prob:csc}) by PIFO algorithms.

\begin{theorem}\label{thm:csc:example}
Consider the minimax problem (\ref{prob:csc}) and $\eps>0$.
Suppose that
{ \small
\begin{align*}
    n \ge 2, \, 
    \kappa_y
    \ge 2, \,
    \eps \le \min \left\{ \frac{L^2 R_x^2}{5184\, n \mu_y}, \frac{\mu_y R_y^2}{36} \right\} \text{ and } \ m = \floor{\frac{R_x}{6} \sqrt{\frac{L^2 - 2 \mu_y^2}{2 n \mu_y \eps}}} - 2.
\end{align*} }
In order to find $(\hat{\vx}, \hat{\vy}) \in \fX \times \fY$ such that $ \E \phi_{\mathrm{CSC}} (\hat{\vx}) - \E \psi_{\mathrm{CSC}} (\hat{\vy}) < \eps$, PIFO algorithm $\fA$ needs at least $N$ queries, where
    $N =
     \Omega\left( n + 
    R_x \sqrt{ 
    {n L \kappa_y}/{ \eps}}
     \right).$
\end{theorem}
When $\kappa_y$ is small, the second term of $N$
is also small. In fact, when $ \kappa_y = \fO(\sqrt{n})$, we can provide a better lower bound as follows.
\begin{theorem}\label{thm:csc:example:2}
For any $L, \mu_y, n, R_x, R_y, \eps$ such that $n \ge 2$, $L \ge \mu_y$, 
$\eps \le \frac{ \sqrt{2} R_x^2 L}{768 \sqrt{n}}$ 
and 
$m = \floor{
\frac{ \sqrt[4]{18} }{12} R_x n^{-1/4} \sqrt{ \frac{L}{\eps} }
} - 1 $,
there exist n functions $\{ f_i : \BR^m \times \BR^m \rightarrow \BR \}_{i=1}^n$ such that 
$f = \frac{1}{n} \sum_{i=1}^n f_i \in \fF_{\mathrm{CC}} (R_x, R_y, L, 0, \mu_y)$
Let $\fX = \{ \vx \in \BR^m : \norm{\vx} \le R_x  \}$ and
$\fY = \{ \vy \in \BR^m : \norm{\vy} \le R_y  \}$. 
In order to find $(\hat{\vx}, \hat{\vy}) \in \fX \times \fY$ such that $ \E \max_{\vy \in \fY} f (\hat{\vx}, \vy) - \E \min_{\vx \in \fX} f (\vx ,\hat{\vy}) < \eps$, PIFO algorithm $\fA$ needs at least $N = \Omega\left(n + R_x n^{3/4} \sqrt{ 
{L}/{\eps} }\right)$ queries.
\end{theorem}
The construction of Theorem \ref{thm:csc:example:2} is similar to that of Theorem \ref{thm:scsc:example:2}.
We still consider the separable function $f(\vx, \vy) = f_x(\vx) - f_y(\vy)$ where $f_x$ is the hard instance used for finite-sum minimization problems and $f_y(\vy) = \frac{\mu_y}{2}\norm{\vy}^2$.
The proofs of Theorems \ref{thm:csc:example} and \ref{thm:csc:example:2} are deferred to Appendix \ref{appendix:minimax:csc}.

Now we give the proof of Theorem \ref{thm:average:convex-strongly}.
\begin{proof}[Proof of Theorem \ref{thm:average:convex-strongly}]
By Lemma \ref{lem:scsc:example:3}, we have the lower bound $\Omega(n)$ if $\eps \le LR_x^2/4$.
Note that if $\eps \ge \frac{L^2 R_x^2}{5184 n \mu_y}$, $\Omega(n) = \Omega\left( n + 
R_x \sqrt{ \frac{ n L \kappa_y }{ \eps } }
\right)$. 
And if $\eps \ge \frac{ \sqrt{2} R_x^2 L}{768 \sqrt{n} }$, $\Omega(n) = \Omega\left(n + R_x n^{3/4} \sqrt{ \frac{L}{\eps} }\right)$.
Then 
for $\eps \le \min \left\{ \frac{L R_x^2}{4}, \frac{\mu_y R_y^2}{36} \right\}$, we have $\mathfrak{m}_{}^{\mathrm{CC}}(\eps, R_x, R_y, L, 0, \mu_y) = \Omega\left(n {+} R_x \sqrt{ \frac{n L}{\eps} }  {+} \frac{R_x L}{ \sqrt{\mu_y \eps} }  \right).$
It remains to add the term $\Omega ( n^{3/4} \sqrt{\kappa_y} \log (\frac{1}{\eps}) ) $ for $\kappa_y = \Omega(\sqrt{n})$.

Now we construct $\{H_{\mathrm{CSC}, i}\}_{i=1}^n, H_{\mathrm{CSC}}: \BR^m \times \BR^m \rightarrow \BR$ as follows.
\begin{align*}
    H_{\mathrm{CSC}, i} (\vx, \vy) & = \frac{L}{2} \norm{\vx}^2 - g_{\text{SC}, i} (\vy) , \\
    H_{\mathrm{CSC}} (\vx, \vy) & =  \frac{1}{n} \sum_{i=1}^n H_{\mathrm{CSC}, i} (\vx, \vy) = \frac{L}{2} \norm{\vx}^2 - g_{\text{SC}} (\vy),
\end{align*}
where $g_{\text{SC}}(\vy)$ is $\mu_y$-convex and $\{ g_{\text{SC}, i} (\vy) \}_{i=1}^n$ is $L$-average smooth.
It is easy to check $H_{\mathrm{CSC}} \in \fF_{\mathrm{CC}} (R_x, R_y, L, 0, \mu_y)$,
\[
\min_{\vx \in \fX} H_{\mathrm{CSC}}(\vx, \vy) = - g_{\text{SC}}(\vy) 
\quad \text{and} \quad
\max_{\vy \in \fY} H_{\mathrm{CSC}}(\vx, \vy) = \frac{1}{2} \norm{\vx}^2  - \min_{\vy \in \fY} g_{\text{SC}}(\vy).
\]
It follows that for any $(\hat{\vx}, \hat{\vy}) \in \fX \times \fY$, we have
\[
\max_{\vy \in \fY} H_{\mathrm{CSC}}(\hat{\vx}, \vy) - \min_{\vx \in \fX} H_{\mathrm{CSC}}(\vx, \hat{\vy})
\ge g_{\text{SC}}(\hat{\vy}) -  \min_{\vy \in \fY} g_{\text{SC}} (\vy).
\]
By Theorem \ref{thm:average:strongly}, for $\eps \le LR_y^2 / 4$ and $
\kappa_y
= \Omega(\sqrt{n} )$, we have $\mathfrak{m}_{\eps}^{\mathrm{CC}}(R_x, R_y, L, 0, \mu_y) = n^{3/4} \sqrt{ 
\kappa_y
} \log \left( \frac{1}{\eps} \right)$. This completes the proof.
\end{proof}
\subsection{Construction for the Convex-Concave Case}\label{sec:minimax:cc}
For the general convex-concave case, the hard instance is constructed as follows.
\begin{definition}\label{defn:cc}
For fixed $L, n, R_x, R_y$ such that $n \ge 2$,
we define $f_{\mathrm{CC}, i} : \BR^m \times \BR^m \rightarrow \BR$ as follows
\begin{align*}
    f_{\mathrm{CC}, i} (\vx, \vy) = \lambda\, {r}^{\mathrm{CC}}_i \left( \vx / \beta, \vy / \beta;  m, 1, \vzero \right).
\end{align*}
where $\lambda = \frac{L R_y^2}{ m \sqrt{8n} }$ and $\beta = \frac{R_y}{\sqrt{m}}$.
Consider the minimax problem
\begin{equation}
    \min_{\vx \in \fX}\max_{\vy \in \fY} f_{\mathrm{CC}}(\vx, \vy) \triangleq \frac{1}{n}\sum_{i=1}^n f_{\mathrm{CC},i}(\vx, \vy), \label{prob:cc}
\end{equation}
where $\fX = \{ \vx \in \BR^m : \norm{\vx} \le R_x  \}$ and
$\fY = \{ \vy \in \BR^m : \norm{\vy} \le R_y  \}$.
Define $\phi_{ \mathrm{CC} } (\vx) = \max_{\vy \in \fY} f_{\mathrm{CC}} (\vx, \vy)$ and
$\psi_{ \mathrm{CC} } (\vy) = \min_{\vx \in \fX} f_{\mathrm{CC}} (\vx, \vy) $.
\end{definition}
One can check that $f_{\mathrm{CC}}$ belongs to $\fF_{\mathrm{CC}} (R_x, R_y, L, 0, 0)$ and satisfies a condition of the form~(\ref{eq:lem:base:condition}) (please see Proposition~\ref{prop:convex-concave} in Appendix~\ref{appendix:minimax:cc}).
Then, we can obtain a PIFO lower bound complexity for the general finite-sum convex-concave minimax problem.

\begin{theorem}\label{thm:cc:example}
Consider minimax problem (\ref{prob:cc}) and $\eps>0$. Suppose that
{ \small
\begin{align*}
    n \ge 2, \,
   \eps \le \frac{L R_x R_y}{72\sqrt{n} }, \text{ and } \ m = \floor{\frac{L R_x R_y}{ 18 \eps \sqrt{n}  }} - 1.
\end{align*} }
In order to find $(\hat{\vx}, \hat{\vy}) \in \fX \times \fY$ such that $ \E \phi_{\mathrm{CC}} (\hat{\vx}) - \E \psi_{\mathrm{CC}} (\hat{\vy}) < \eps$, PIFO algorithm $\fA$ needs at least $N = \Omega\left( n + 
{ \sqrt{n} L R_x R_y}/{\eps} \right)$ queries.
\end{theorem}
Note that Theorem \ref{thm:csc:example} requires the condition $\eps\leq\fO(L/\sqrt{n})$ to obtain the desired lower bound. 
For large $\eps$, we can apply the following lemma.
\begin{lemma}\label{lem:cc:example:2} 
For any positive $L, n, R_x, R_y, \eps$ such that $n \ge 2$ and
    $\eps \le \frac{1}{4} L R_x R_y$
there exist n functions $\{ f_i : \BR \times \BR \rightarrow \BR \}_{i=1}^n$ 
such that 
 $f = \frac{1}{n} \sum_{i=1}^n f_i \in \fF_{\mathrm{CC}} (R_x, R_y, L, 0, 0)$.
Let $\fX = \{ x \in \BR : |x| \le R_x  \}$ and
$\fY = \{ y \in \BR : |y| \le R_y  \}$. 
In order to find $(\hat{x}, \hat{y}) \in \fX \times \fY$ such that
$ \E \max_{y \in \fY} f (\hat{x}, y) - \E \min_{x \in \fX} f (x ,\hat{y}) < \eps$, PIFO algorithm $\fA$ needs at least $N = \Omega(n)$ queries.
\end{lemma}
This Lemma is similar to Lemma \ref{lem:scsc:example:3}.
The proofs of Theorem \ref{thm:cc:example} and Lemma \ref{lem:cc:example:2} are deferred to Appendix \ref{appendix:minimax:cc}.

Now we can give the proof of Theorem \ref{thm:average:convex-concave}.
\begin{proof}[Proof of Theorem \ref{thm:average:convex-concave}]
Note that for $\eps \ge \frac{L R_x R_y}{ 72 \sqrt{n} }$, we have $\Omega \left( n + \frac{ \sqrt{n} L R_x R_y}{\eps} \right) = \Omega(n)$.
Combining Theorem \ref{thm:cc:example} and Lemma \ref{thm:csc:example:2},
we obtain the lower bound $\Omega \left( n + \frac{ \sqrt{n} L R_x R_y}{\eps} \right) $ for $\eps \le L R_x R_y / 4$.
On the other hand,
$G_{\mathrm{CSC}} $ defined in the proof of Theorem \ref{thm:csc:example:2} and $H_{\text{SCSC}}$ defined in the proof of Lemma \ref{lem:scsc:example:3} are also convex-concave and $\eps \ge \frac{ \sqrt{2} R_x^2 L }{ 768 \sqrt{n} }$ implies $\Omega \left( n + R_x n^{3/4} \sqrt{ \frac{L}{\eps} } \right) = \Omega(n) $. Thus, 
we have the lower bound $\Omega \left( n + R_x n^{3/4} \sqrt{ \frac{L}{\eps} } \right)$ for $\eps \le L R_x^2 / 4$.
It is also worth noting that if $f(\vx, \vy)$ is convex in $\vx$ and concave in $\vy$, then $-f(\vx, \vy)$ is convex in $\vy$ and concave in $\vx$. This implies the symmetry of $\vx$ and $\vy$. Thus, we can also obtain the lower bound $\Omega \left( n + R_y n^{3/4} \sqrt{ \frac{L}{\eps} } \right)$ for $\eps \le L R_y^2 / 4$.
In summary, for $\eps \le \frac{LR_x R_y}{4}$, the lower bound is $\Omega \left( n + \frac{L R_x R_y}{\eps} + (R_x + R_y) n^{3/4} \sqrt{ \frac{L}{\eps} } \right)$.
\end{proof}
\subsection{Construction for the Nonconvex-Strongly-Concave Case}\label{sec:minimax:ncsc}

In this subsection, we consider the finite-sum minimax problem where the objective function is strongly-concave in $\vy$ but nonconvex in $\vx$.
The analysis is based on the following construction.

\begin{definition}\label{defn:ncsc}
For fixed $L, \mu_x, \mu_y, \Delta, n$, we define $f_{\mathrm{NCSC}, i}: \BR^{m+1} \times \BR^{m+1}  \rightarrow \BR$ as follows
\begin{align*}
    f_{\mathrm{NCSC}, i}(\vx, \vy) = \lambda\, {r}^{\mathrm{NCC}}_i\left(\vx / \beta, \vy / \beta; 
    m + 1,
    \sqrt[4]{\alpha}, {\vc}^{\mathrm{NCSC}} \right), \text{ for } 1 \le i \le n,
\end{align*}
where
{ \small
\begin{align*}
    \alpha & = \min \left\{1, \frac{32n \mu_y}{135 L}, \frac{128 (\sqrt{3} + 1) n \mu_x \mu_y}{45 L^2} \right\}, \;
    {\vc}^{\mathrm{NCSC}} = \left( \frac{16 \sqrt{n} \mu_y}{L}, \frac{\sqrt{\alpha} L }{16 \sqrt{n} \mu_y }, \sqrt[4]{\alpha}  \right) ,\; \\
    \lambda & = \frac{5308416 n^{3/2} \mu_y^2 \eps^2}{L^3 \alpha},\;
    \beta = 4 \sqrt{\lambda \sqrt{n} / L}
    \; \mbox{ and } \; 
    m = \floor{\frac{\Delta L^2 \sqrt{\alpha}}{3483648 n \eps^2 \mu_y}}.
\end{align*} }
Define $\phi_{\mathrm{NCSC}} (\vx) = \max_{\vy \in \BR^{m+1}} f_{\mathrm{NCSC}} (\vx, \vy)$. Consider the minimax problem
\begin{equation}
    \min_{\vx \in \BR^{m+1}}\max_{\vy \in \BR^{m+1}} f_{\mathrm{NCSC}}(\vx, \vy) \triangleq \frac{1}{n}\sum_{i=1}^n f_{\mathrm{NCSC},i}(\vx, \vy). \label{prob:ncsc}
\end{equation}
\end{definition}
One can check that $f_{\mathrm{NCSC}}$ belongs to $\fF_{\mathrm{NCC}} (\Delta, L, \mu_x, \mu_y)$ and satisfies a condition of the form (\ref{eq:lem:ncsc:base:condition}) (please see Proposition \ref{prop:nonconvex-strongly} in Appendix \ref{appendix:minimax:ncsc}).
With Proposition \ref{prop:nonconvex-strongly}, we
can give the proof of Theorem \ref{thm:average:nonconvex-strongly}.
\begin{proof}[Proof of Theorem \ref{thm:average:nonconvex-strongly}]
Combining Lemma \ref{lem:minimax:ncsc:base} and the third property of Proposition \ref{prop:nonconvex-strongly}, for $N = \frac{ n m}{4(1+c_0)}$, we have 
$
\min_{t \le N} \E \norm{ \nabla \phi_{\mathrm{NCSC} } (\vx_t) } \ge \eps.
$
Thus, in order to find $(\hat{\vx}, \hat{\vy})$ such that $\E \norm{ \nabla \phi_{\mathrm{NCSC} } (\hat{\vx}) } < \eps$, $\fA$ needs at least $N$ PIFO queries, where
$
    N = \frac{nm}{4(1+c_0)} = \Omega \left( \frac{\Delta L^2 \sqrt{\alpha} }{ \eps^2 \mu_y} \right).
    $
Since $\eps^2 \le \frac{\Delta L^2 \alpha}{6767296 n \mu_y}$ and $\alpha \le 1$, we have $\Omega \left( \frac{\Delta L^2 \sqrt{\alpha} }{ \eps^2 \mu_y} \right) = \Omega \left( n + \frac{\Delta L^2 \sqrt{\alpha} }{  \eps^2 \mu_y} \right)$.
\end{proof}
\subsection{Smooth Cases}\label{sec:minimax:smooth}
In this subsection, we focus on the more constrained function classes where each component $f_i$ is $L$-smooth.
The results are summarized in Table~\ref{table-minimax}.
We defer the definitions of the function class and optimization complexity and the formal statements of our lower bounds to Appendix~\ref{appendix:minimax:smooth}.
\renewcommand{\arraystretch}{1.8}
\begin{table}[h]
\caption{
Upper and lower bounds with the assumption that $f_i$ is $L$-smooth and $f$ is $(\mu_x, \mu_y)$-convex-concave.
The condition numbers are defined as $\kappa_x = L / \mu_x$ and $\kappa_y = L / \mu_y$ when $\mu_x, \mu_y > 0$. The definitions of $R_x, R_y$ and $\Delta$ are given in Table~\ref{table-minimax-average}.}
\label{table-minimax}
\footnotesize
\setlength{\tabcolsep}{4pt}
\begin{center}
 \begin{tabular}{|c|c|c|} 
 \hline
 Cases & Upper or Lower Bounds & References \\ [0.1cm] 
  \hline
 \multirow{2}{*}{$\mu_x > 0, \mu_y > 0$} & $\tilde{\fO}\left(\left(n +  \frac{\sqrt{n} L}{\min\{\mu_x, \mu_y\}}\right) \log(1/\eps) \right)$ & \makecell{\citet{carmon2019variance}; \\ \citet{luo2019stochastic} } \\
  \cline{2-3}
 & $\Omega\left(\sqrt{\left(n + \kappa_x \right)\left(n + \kappa_y \right)} \log(1/\eps)\right)$ & Theorem \ref{thm:strongly-strongly} \\[0.15cm]
 \cline{1-3}
 $\mu_x = 0, \mu_y > 0$ & $\Omega\left(n + R_x \sqrt{\frac{n L}{\eps}} + 
 R_x \sqrt{ \frac{L \kappa_y}{\eps} }  
 + \sqrt{ n \kappa_y } \log \left( \frac{1}{\eps} \right) \right)$ & Theorem \ref{thm:convex-strongly}  \\[0.15cm]
 \cline{1-3}
 \multirow{2}{*}{ $\mu_x = 0, \mu_y = 0$ } & $\tilde{\fO} \left(n + 
 \frac{\sqrt{n} L (R_x^2 + R_y^2)}{\eps} \right) 
 $ &  \citet{carmon2019variance} \\
 \cline{2-3}
 & $\Omega\left(n + \frac{L R_x R_y}{\eps} + (R_x + R_y) \sqrt{\frac{n L}{\eps}}\right)$ & Theorem \ref{thm:convex-concave}  \\[0.15cm]
 \hline
 \makecell{ $\mu_x < 0, \mu_y > 0$ \\ $\kappa_y = \Omega(\sqrt{n})$ } & $\Omega \left( n + \frac{\Delta L \sqrt{\kappa_y} }{\eps^2} 
 \right)$ & Theorem \ref{thm:nonconvex-strongly} \\[0.15cm]
 \hline
\end{tabular}
\end{center}
\end{table}
\renewcommand{\arraystretch}{1}

In Table~\ref{table-minimax}, we only present the upper bounds of some methods designed for the smoothness case\footnote{Although the method in \citet{carmon2019variance} has two loops and does not satisfy our definition, we list it here for a better comparison.}.
Methods designed for the average smoothness functions also apply here and thus the upper bounds in Table~\ref{table-minimax-average} are still valid.
However, there exists some gap in all cases.

Compared to the lower bounds in Table~\ref{table-minimax-average}, the lower bounds in Table~\ref{table-minimax} have the same dependence on $L, \kappa_x, \kappa_y, \eps$, but with a weaker dependence on $n$.
Specially. if we replace $L$, $\kappa_x$ and $\kappa_y$ in Table~\ref{table-minimax} by $\sqrt{n} L$, $\sqrt{n} \kappa_x$ and $\sqrt{n} \kappa_y$ respectively\footnote{For the nonconvex-strongly-concave case, we just need to replace $L$ by $\sqrt{n}L$.}, we can obtain the lower bounds in Table~\ref{table-minimax-average}.
This is due to the way of partitioning the matrix $\mB(m, \omega, \zeta)$ in Section~\ref{sec:frame:instance}.
Intuitively, we partition the Hessian matrix of the coupling term between $\vx$ and $\vy$ and each component only gets a low-rank part.
Propositions~\ref{prop:convex-concave:base} and \ref{prop:nonconvex-strongly:base}
have shown the $\sqrt{n}$ gap between the smoothness and average smoothness parameters as long as the non-coupling term is not too large.

\paragraph{Convex-concave cases}
We speculate that when $f$ is convex-concave, the lower bounds in Table~\ref{table-minimax} are the best ones our framework can obtain, because the corresponding lower bounds under the average smoothness assumption have been nearly matched by existing upper bounds.
To further improve the lower bounds, one may have to resort to new constructions.

As for the upper bounds, we notice that most work only uses the average smoothness condition.
We guess that the smoothness property of each component function needs to be better employed,
because the upper and lower bounds for convex minimization problems under the two smoothness conditions nearly match (see Tables~\ref{table-min} and \ref{table-min-average}),

\paragraph{Nonconvex-strongly-concave case}
When $f$ is nonconvex-strongly-concave, there exists a gap between the upper and lower bounds under both smoothness and average smoothness assumptions.
Since the nonconvexity poses more difficulty to the problem,
it remains an open problem whether the upper bounds, the lower bounds, or both can be further tightened. 
\section{Lower Complexity Bounds for the Minimization Problems}\label{sec:min}
In this section, we focus on 
the minimization problem
\begin{align}\label{prob:min}
    \min_{\vx \in \fX} f(\vx) = \frac{1}{n} \sum_{i=1}^n f_i(\vx),
\end{align}
where each individual component $f_i(\vx)$ is $L$-smooth or the function class $\{ f_i(\vx)\}_{i=1}^n$ is $L$-average smooth, the feasible set $\fX$ is closed and convex such that $\fX \subseteq \BR^{d}$.
We show that
we can obtain similar lower bounds as those in \citet{woodworth2016tight, hannah2018breaking, zhou2019lower}.

Recall that Problem (\ref{prob:main}) becomes Problem (\ref{prob:min}) if we set $\fY$ as a singleton.
Then the definitions of function classes and optimization complexity come directly from their counterparts in Sections~\ref{sec:frame:def}.
The details are deferred to Appendix~\ref{sec:min:setup:appendix}.

In Section~\ref{sec:min:instance}, we construct the hard instances for Problem~(\ref{prob:min}).
In Section~\ref{sec:min:results}, we summarize our results and compare them with previous work.

\subsection{The Hard Instances
}\label{sec:min:instance}
In this subsection, we present the construction of hard instances for Problem (\ref{prob:min}) and compare our construction with some related work.

The construction is also based on the class of matrices $\mB(m, \omega, \zeta)$ define in Equation~(\ref{eq:mat:B}).
We still use $\vb_{l-1}(m,\omega,\zeta)^\top$ to denote the $l$-th row of $\mB(m, \omega, \zeta)$
and defined the index sets
 $\fL_1,\dots,\fL_n$ as $\fL_i = \big\{ l: 0 \le l \le m, l \equiv i - 1~(\bmod~n) \big\}$.
Then the hard instance is constructed as
\begin{align}\label{prob:r}
    \min_{\vx \in \fX} r(\vx; m, \omega, \zeta, \vc) \triangleq \frac{1}{n} \sum_{i=1}^n r_i(\vx; m, \omega, \zeta, \vc),
\end{align}
where 
$\vc = (c_1, c_2, c_3)$, $\fX = \{\vx \in \BR^m: \norm{\vx} \le R_x \}$ or $\BR^m$, and
\begin{align*}
     & r_i(\vx; m, \omega, \zeta, \vc) \\
     & =
    \begin{cases}
        \frac{n}{2} \sum\limits_{l \in \fL_i} \norm{\vb_{l}(m, \omega, \zeta)^{\top}\vx}^2
        + \frac{c_1}{2} \norm{\vx}^2 + c_2 \sum\limits_{i=1}^{m-1} \Gamma (x_i) - c_3 n \inner{\ve_1}{\vx}, 
        \quad \text{ for } i = 1, \\
        \frac{n}{2} \sum\limits_{l \in \fL_i} \norm{\vb_{l}(m, \omega, \zeta)^{\top}\vx}^2
        + \frac{c_1}{2} \norm{\vx}^2 + c_2 \sum\limits_{i=1}^{m-1} \Gamma (x_i) , 
        \qquad \quad \, \text{ for } i = 2, 3, \dots, n.
    \end{cases}
\end{align*}
The nonconvex function $\Gamma: \BR \rightarrow \BR$ is
$
    \Gamma(x) \triangleq 120 \int_1^{x} \frac{t^2(t-1)}{1 + t^2} d t. $
We can determine the smoothness and strong convexity parameters 
of $r_i$ 
similar to Propositions \ref{prop:convex-concave:base} and \ref{prop:nonconvex-strongly:base}. 
The details are deferred to Proposition~\ref{prop:base} in Appendix~\ref{sec:min:instance:appendix}.

One can check that $r(\vx; m, \omega, \zeta, \vc) = \frac{1}{2} \vx^\top \mA(m, \omega, \zeta)\, \vx + \frac{c_1}{2} \norm{\vx}^2 + c_2 \sum_{i=1}^{m-1} \Gamma(\vx_i) - c_3 \inner{\ve_1}{\vx} $, where
{ \small
\begin{align*}
    \mA(m, \omega, \zeta) \triangleq \mB(m, \omega, \zeta)^{\top} \mB(m, \omega, \zeta) =
    \begin{bmatrix}
        \omega^2 + 1 & -1     &        &    &    \\
        -1           & 2      & -1     &    &    \\
                     & \ddots & \ddots &\ddots&    \\
                     &        & -1     & 2  & -1 \\
                     &        &        & -1 & \zeta^2+1
    \end{bmatrix}.
\end{align*} }
The matrix $\mA(m, \omega, \zeta)$ is widely-used in the analysis of lower bounds for convex optimization \citep{nesterov2013introductory,agarwal2015lower,lan2017optimal,carmon2017lower,zhou2019lower}.

Now we compare our construction with \citet{lan2017optimal} and \citet{zhou2019lower}.
In our construction, we partition the row vectors of $\mB (m, \omega, \zeta)$ into $n$ parts and each component function is defined in terms of only one part. All the component functions share the same $\vx$.
However, in \citet{lan2017optimal}, different component functions
share the same form except that they are based on different subvectors of the high-dimensional $\vx$.
Intuitively speaking,
we partition the Hessian matrix while \citet{lan2017optimal} partition the variable.
The construction of \citet{zhou2019lower} is more complex than \citet{lan2017optimal} but the basic idea is the same.

Recall the subspaces $\{\fF_k\}_{k=0}^m$ 
defined in (\ref{eq:subspace}).
The next lemma shows that the hard instance also satisfies a variant of the \textit{zero-chain} property.
\begin{lemma}\label{lem:jump}
Suppose that $n \ge 2$, $c_1 \ge 0$ and  $\vx \in \fF_k$, $0 \le k < m$. 
If 
(i) (convex case) $c_2=0$ and $\omega=0$, or
(ii) (nonconvex case) $c_1 = 0$, $c_2 > 0$, $\zeta=0$ and $\gamma < \frac{\sqrt{2} + 1}{60 c_2}$, we have
\begin{align*}
        \nabla r_i(\vx),\, \prox_{r_{i}}^{\gamma} (\vx)\in 
        \begin{cases}
        \fF_{k+1}, & \text{ if } i \equiv k + 1 (\bmod ~n), \\
        \fF_{k}, & \text{ otherwise}.
        \end{cases}
    \end{align*}
We omit the parameters of ${r}_i$ to simplify the presentation.
\end{lemma}
The proof of Lemma \ref{lem:jump} are given in Appendix~\ref{appendix:min:proof:base}.

We emphasize that the assumption on $\gamma$ naturally holds. 
Recall that the choice of $\gamma$ should satisfy that $r_i(\vu) + \frac{1}{2 \gamma} \norm{\vx - \vu}^2$ is a convex function of $\vu$ for a fixed $\vx$. Proposition~\ref{prop:base} implies that we must have $\gamma \le \frac{1}{ 45 (\sqrt{3} - 1) c_2 } \le \frac{\sqrt{2} + 1}{ 60 c_2 }$.

In short, if $\vx \in \fF_k$, 
then there exists only one $i\in\{1,\dots,n\}$ such that 
$h_{f_i}^{\mathrm{PIFO}}$
could provide additional information in $\fF_{k+1}$.
This property is
the main difference between
the constructions in \citet{lan2017optimal, zhou2019lower} and ours.
In \citet{lan2017optimal, zhou2019lower}, no matter which component is drawn, the number of the nonzero elements of the current point can increase. 
Such a difference results from the different ways of partitioning.
As a consequence, their hard instances need to be constructed in a space with a higher dimension than ours. 
Moreover, our construction also works for PIFO oracles while the constructions of \citet{lan2017optimal} and \citet{zhou2019lower} only apply to IFO oracles.

With Lemma \ref{lem:jump}, we can obtain how many PIFO calls we need as what we did in Section~\ref{sec:frame:instance}. The details are deferred to Appendix~\ref{sec:min:instance:appendix}.

\subsection{Results}\label{sec:min:results}
In this subsection, we present our lower bounds in Tables~\ref{table-min} and \ref{table-min-average}, and compare them with previous upper and lower bounds.
It is worth emphasizing that we are not trying to list all the upper bounds, just to provide a few algorithms that could match our lower bounds.
The formal statements of our lower bounds are deferred to Appendix~\ref{sec:min:result:appendix}.

\renewcommand{\arraystretch}{1.8}
\begin{table}[ht]
\caption{The upper and lower bounds with the assumption that $f_i$ is $L$-smooth and $f$ is $\mu$-strongly convex, convex or $\mu$-weakly convex. 
$\kappa = L / \mu$ for $\mu > 0$.
The definitions of $R$, $\Delta$ and  optimization complexity are given in Appendix~\ref{sec:min:setup:appendix}.
}
\label{table-min}
\footnotesize
\setlength{\tabcolsep}{4pt}
\begin{center}
 \begin{tabular}{|c|c|c|} 
 \hline
 Cases & Upper or Lower Bounds & References \\
  \hline
 \multirow{4}{*}{$\mu > 0$} & $\fO\left(\left(n {+} \sqrt{\kappa n}\right)\log\left(1/\eps \right)\right)$  & \citet{defazio2016simple, li2021anita} \\
  \cline{2-3}
  & $\fO \left( n +
  \frac{ n  \log (1 / \eps) }{ 1 + (\log(n/\kappa))_{+} }
    \right),\ \  \kappa = \fO (n)$  & \citet{hannah2018breaking} \\
  \cline{2-3}
 &  \multirow{2}{*}{ 
 $
 \left\{
 \begin{aligned}
    & \Omega\left(\left(n {+} \sqrt{\kappa n}\right)\log\left(1/\eps \right)\right), &
    \kappa = \Omega(n), \\
    & 
    { 
    \Omega\left( n + 
    \frac{ n  \log (1 / \eps) }{ 1 + (\log(n/\kappa))_{+} }
    \right) 
    }, &
    \kappa = \fO(n).
\end{aligned} \right.
 $ 
 } & \multirow{2}{*}{ 
 \citet{hannah2018breaking};
 Theorem \ref{thm:strongly} 
 } \\
 & & \\
 \cline{1-3}
 \multirow{2}{*}{ $\mu = 0$ } & $\tilde{\fO} \left( n + R \sqrt{nL / \eps}   \right) $ &  \citet{li2021anita} \\
 \cline{2-3}
 & $\Omega\left(n {+} R\sqrt{nL/\eps} \right)$ & \citet{woodworth2016tight}; Theorem \ref{thm:convex}  \\
 \hline
 \multirow{2}{*}{$\mu < 0$} & \multirow{1}{*}{ $ \tilde{\fO} \left( n + \frac{\Delta}{\eps^2} \min\{ \sqrt{n}L, n |\mu| + \sqrt{n |\mu| L}\} \right)$ } & \multirow{1}{*}{ 
 \citet{lan2019accelerated, li2020convergence}
 } \\
 \cline{2-3}
 & $\Omega \left( n + \frac{\Delta}{\eps^2} \min\{L, \sqrt{n |\mu| L}\} \right)$ & \citet{zhou2019lower}; Theorem \ref{thm:nonconvex} \\
 \hline
\end{tabular}
\end{center}
\end{table}
\renewcommand{\arraystretch}{1}

\paragraph{Smooth cases}
Table~\ref{table-min} shows the upper and lower bounds when each $f_i$ is $L$-smooth\footnote{ The lower bound of \citet{hannah2018breaking} for $\kappa = \Omega(n)$ uses the lower bound in \citet{woodworth2016tight}.  }.
For the strongly-convex and convex cases,
the upper bounds and lower bounds nearly match up to log factors, while for the nonconvex case, there is still a $\sqrt{n}$ gap.
Specially, when $\kappa = \Omega(n)$, the lower bound is $\Omega(n + \Delta \sqrt{n |\mu| L} / \eps^2 )$ and has been achieved by \citet{lan2017optimal} up to log factors. When $\kappa = \fO(n)$, the lower bound is $\Omega(n + \Delta / \eps^2 )$, while the upper bound by \citet{li2020convergence} is $\Omega(n + \sqrt{n} \Delta / \eps^2 )$.
From the analysis in Section~\ref{sec:alg:our},
the algorithms in \citet{defazio2016simple, hannah2018breaking, li2021anita, lan2019accelerated, li2020convergence} all belong to PIFO algorithms.
In fact, except the one in \citet{defazio2016simple}, others are also IFO algorithms.

As for the lower bounds,
\citet{hannah2018breaking} consider the class of
p-CLI oblivious algorithms introduced in \citet{arjevani2016dimension}.
For these algorithms, we can left-multiply the gradient by a preconditioning matrix. 
Thus, the linear-span assumption can be violated.
However, proximal operators are still not taken into account.
\citet{woodworth2016tight} prove the lower bounds for arbitrary randomized algorithms with access to PIFO oracles.
Although 
smaller than that in \citet{woodworth2016tight}, 
our class of algorithms is large enough to include many near-optimal algorithms.
Moreover, our construction is simpler than \citet{woodworth2016tight}.
As a result, such a construction can not only provide more intuition about the optimization process, but also requires fewer dimensions to construct the hard instances.
Specially, for the convex case, our construction only requires the dimension to be $\fO \left( 1 + R \sqrt{L/(n\eps)} \right)$ (see Appendix~\ref{sec:min:convex}), which is much smaller than 
$\fO \left(\frac{L^2R^4}{\eps^2}\log\left(\frac{nLR^2}{\eps}\right)\right)$ in
\citet{woodworth2016tight}.

\citet{zhou2019lower} only consider the class of IFO algorithms, which is
only a subset of PIFO algorithms.
Moreover, our construction still requires fewer dimensions.
For the nonconvex case,
our construction only requires the dimension to be $\fO \left(1 + \frac{\Delta}{\eps^2} \min \{L/n, \sqrt{\mu L /n}\} \right)$ (see Appendix \ref{sec:min:nonconvex}), which is much smaller than 
$\fO \left( \frac{\Delta}{\eps^2} \min \{L, \sqrt{n \mu L}\} \right) $ in \citet{zhou2019lower}.

\renewcommand{\arraystretch}{1.8}
\begin{table}[ht]
\caption{
The upper and lower bounds with the assumption that $\{f_i\}_{i=1}^n$ is $L$-average smooth and $f$ is $\mu$-strongly convex, convex or $\mu$-weakly convex. $\kappa = L / \mu$ for $\mu > 0$.
The definitions of $R$, $\Delta$ and  optimization complexity are given in Appendix~\ref{sec:min:setup:appendix}.
}
\label{table-min-average}
\footnotesize
\setlength{\tabcolsep}{4pt}
\begin{center}
 \begin{tabular}{|c|c|c|} 
 \hline
 Cases & Upper or Lower Bounds & References \\
  \hline
 \multirow{2}{*}{
 \shortstack{ $\mu > 0$,\\ $\kappa = \Omega (\sqrt{n}) $ } } & $\fO\left(\left(n {+} n^{3/4} \sqrt{\kappa }\right)\log\left(1/\eps \right)\right)$  & \citet{allen2018katyushax} \\
  \cline{2-3}
 &  
 { $
 \left.
 \begin{aligned}
    & \Omega\left(\left(n {+} n^{3/4} \sqrt{\kappa }\right)\log\left(1/\eps \right)\right)
\end{aligned} \right.
 $ } & 
 { \citet{zhou2019lower}; Theorem \ref{thm:average:strongly}  } \\
 \cline{1-3}
 \multirow{2}{*}{ $\mu = 0$ } & $\fO \left( n \log (1/\eps) + R n^{3/4} \sqrt{L / \eps}   \right) $ &  \citet{allen2018katyushax} \\
 \cline{2-3}
 & $\Omega\left(n {+} R n^{3/4} \sqrt{L/\eps} \right)$ & \citet{zhou2019lower}; Theorem \ref{thm:average:convex}  \\
 \hline
 \multirow{2}{*}{$\mu < 0$} &  
 { $\tilde{\fO} \left( n + \frac{\Delta}{\eps^2} \min\{ \sqrt{n} L, n^{3/4} \sqrt{ |\mu| L}\} \right)$ } & 
 \citet{allen2017natasha, li2021page}
 \\
 \cline{2-3}
 & $\Omega \left( n + \frac{\Delta}{\eps^2} \min\{ \sqrt{n} L, n^{3/4} \sqrt{ |\mu| L}\} \right)$ & \citet{zhou2019lower}; Theorem \ref{thm:average:nonconvex} \\
 \hline
\end{tabular}
\end{center}
\end{table}
\renewcommand{\arraystretch}{1}

\paragraph{Average smooth cases}
For the average smooth cases, the upper and lower bounds nearly match up to log factors for all three cases.
Specially, for the nonconvex case, when $\kappa = \Omega(\sqrt{n})$, the lower bound is $\Omega(n + \Delta n^{3/4} \sqrt{ |\mu| L } / \eps^2 )$ and has been achieved by  repeatedSVRG in \citet{agarwal2017finding, carmon2018accelerated, allen2017natasha}\footnote{This method was implicitly proposed in \citet{agarwal2017finding, carmon2018accelerated} and formally named as repeatedSVRG in \citet{allen2017natasha}. } up to log factors. When $\kappa = \fO (\sqrt{n})$, the lower bound is $\Omega(n + \Delta L \sqrt{n} / \eps^2)$ and has been achieved by \citet{li2021page}.
One can check that the algorithms in \citet{allen2018katyushax, li2021page} are both IFO algorithms.
The method repeatedSVRG in   \citet{allen2017natasha}
can also be modified into 
IFO algorithms\footnote{
Similar to the analysis for catalyst accelerated methods in Section~\ref{sec:alg:our}.
}.
As for the lower bounds, our results have the same orders as those in \citet{zhou2019lower} and can apply to PIFO algorithms. And our constructions also require fewer dimensions than \citet{zhou2019lower}. The details are deferred to Appendix~\ref{sec:min:result:appendix}.

\paragraph{IFO and PIFO algorithms}
From the above analysis, we find that PIFO oracles are no more powerful than IFO oracles in terms of the complexity for smooth functions.
The PIFO lower bounds have been nearly matched by many IFO algorithms.
This is consistent with the observation in \citet{woodworth2016tight}.
From the results in Table~\ref{table-minimax-average},
this phenomenon also appears in finite-sum minimax problems under the average smoothness assumption.
As a comparison, \citet{woodworth2016tight} shows that
for Lipschitz but nonsmooth functions, having access to proximal oracles does reduce the complexity.

\section{Concluding Remarks}\label{sec:conclusion}

In this paper, focusing on finite-sum minimax and minimization optimization problems, 
we have given a new definition of PIFO algorithms, which have access to proximal and gradient oracles for each component function and can obtain the full gradient infrequently.
This class of PIFO algorithms are large enough to include many near-optimal methods.
We have developed a novel approach to constructing the hard instance.
Instead of partitioning the variable~\citep{lan2017optimal, zhou2019lower}, we partition the classical tridiagonal matrix in \citet{nesterov2013introductory}  into $n$ groups.
Such a construction is friendly to the analysis of both IFO and PIFO algorithms, providing some intuition of the optimization process and requiring fewer dimensions than those in \citet{woodworth2016tight, zhou2019lower}.

Based on our approach, we have established the lower bounds for finite-sum minimax problems when $f$ is convex-concave or nonconvex-strongly-concave and $\{f_i\}_i^n$ is $L$-average smooth. Most of the lower bounds are nearly matched by existing upper bounds up to log factors.
For minimization problems, we have derived similar lower bounds as in \citet{woodworth2016tight, hannah2018breaking, zhou2019lower}.
The comparison of upper and lower bounds shows that for smooth functions, the proximal oracles are not much more powerful than gradient oracles.

Finally, we propose several future research directions.
\begin{itemize}
    \item When $f$ is nonconvex-strongly-concave or each $f_i$ is $L$-smooth, there still exists some gap between the upper and lower bounds. It remains open to design faster algorithms or tighten the lower bound to close the gap. 
    \item It would be interesting to apply our construction framework to prove the lower bounds for nonconvex-concave cases.
    \item The definition of PIFO algorithms can be further extended to include more methods.
    For example, the distribution $\fD$ over $[n]$ and the expectation $q$ of the Bernoulli random variable need not be stationary over time. Sampling without replacement and methods that break the linear-span protocol are also worth considering.
\end{itemize}

\vskip 0.2in
\bibliographystyle{plainnat}
\bibliography{reference}

\newpage
\appendix
\section{Results of the Sum of Geometric Distributions}\label{appendix:geo}
In this section, we present the approach to proving Lemma \ref{lem:geo}.
We can view the probability $\pr{\sum_{i=1}^m Y_i > j}$ as a function of $m$ variables $p_1, p_2, \dots, p_m$:
\begin{align}\label{def:fmj}
    f_{m, j}(p_1, p_2, \dots, p_m) \triangleq \pr{\sum_{i=1}^m Y_i > j}.
\end{align}
We first provide the following useful result about the function $f_{m, j}$.
    \begin{lemma}\label{lem:geo:aux}
        For $m \ge 2$ and $j \ge 1$, we have that
        \begin{align*}
            f_{m, j}(p_1, p_2, \dots, p_m)
            \ge f_{m, j}\left( \frac{\sum_{i=1}^m p_i}{m}, \dots, \frac{\sum_{i=1}^m p_i}{m} \right).
        \end{align*}
    \end{lemma}
This lemma implies that with the sum of the $p_i$ unchanged, the uniform case (all the $p_i$ are equal) is the least heavy-tailed.
Since we aim to give a lower bound,
it suffices to only focus on the uniform case.
The proof of Lemma \ref{lem:geo:aux} is given in Appendix \ref{appendix:geo:aux}.
With Lemma \ref{lem:geo:aux} in hand, we give the proof of Lemma \ref{lem:geo}.

\begin{proof} [Proof of Lemma \ref{lem:geo}].
    Let $p = \frac{\sum_{i=1}^m p_i}{m}$ and $\left\{Z_i \sim \geo\left( p \right)\right\}_{i=1}^m$ be independent geometric random variables.
    Then we have
    \begin{align*}
        \pr{\sum_{i=1}^m Y_i > \frac{m^2}{4(\sum_{i=1}^m p_i)}} > \pr{\sum_{i=1}^m Z_i > \frac{m}{4 p}}.
    \end{align*}
    Denote $\sum_{i=1}^m Z_i$ by $\tau$. It is easily checked that $
        \E [\tau] = \frac{m}{p} \; \mbox{ and } \; \Var (\tau) = \frac{m(1-p)}{p^2}$.
    Hence, we have
    \begin{align*}
          \quad \pr{\tau > \frac{1}{4} \E \tau} & = \pr{\tau - \E \tau > -\frac{3}{4} \E \tau} \\
         & = 1 - \pr{\tau - \E \tau \le -\frac{3}{4} \E \tau}
        \ge 1 - \pr{|\tau - \E \tau| \ge \frac{3}{4} \E \tau}                                 \\
         & \ge 1 - \frac{16\Var(\tau)}{9(\E \tau)^2}
        = 1 - \frac{16m(1-p)}{9m^2} \ge 1 - \frac{16}{9m} \ge \frac{1}{9},
    \end{align*}
which completes the proof.
\end{proof}
\subsection{Proof of Lemma \ref{appendix:geo:aux}}\label{appendix:geo:aux}
Before giving the proof of Lemma \ref{lem:geo:aux}, 
we first present some results about $f_{2, j}$.
which is defined in Equation (\ref{def:fmj}).
\begin{lemma}\label{lem:two-geo}
    The following properties hold for the function $f_{2, j}$.
    \begin{enumerate}
        \item For $j \ge 1$, $p_1, p_2 \in (0, 1]$, it holds that
              \begin{align*}
                  f_{2, j}(p_1, p_2) = \begin{cases}
                      j p_1(1 - p_1)^{j-1} + (1 - p_1)^j, ~~~                  & \text{if  } p_1 = p_2, \\
                      \frac{p_2 (1 - p_1)^j - p_1 (1 - p_2)^j}{p_2 - p_1}, ~~~ & \text{otherwise}.
                  \end{cases}
              \end{align*}
        \item For $j \ge 2, p_1 \neq p_2$, we have
              \[f_{2, j}(p_1, p_2) > f_{2, j} \left( \frac{p_1 + p_2}{2}, \frac{p_1 + p_2}{2} \right).\]
    \end{enumerate}

\end{lemma}
\begin{proof}
    \textit{1.} Let $Y_1 \sim \geo(p_1), Y_2 \sim \geo(p_2)$ be two independent random variables. Then
    \begin{align*}
        \pr{Y_1 + Y_2 > j}
         & = \sum_{l=1}^{j} \pr{Y_1 = l} \pr{Y_2 > j - l}
        + \pr{Y_1 > j}                                                                 \\
         & = \sum_{l=1}^{j} (1-p_1)^{l-1} p_1 (1-p_2)^{j-l} + (1-p_1)^j                \\
         & = p_1 (1-p_2)^{j-1} \sum_{l=1}^{j} \left( \frac{1-p_1}{1-p_2} \right)^{l-1}
        + (1-p_1)^j.
    \end{align*}
    If $p_1 = p_2$, Then $\pr{Y_1 + Y_2 > j} = j p_1(1 - p_1)^{j-1} + (1 - p_1)^j$; if $p_1 < p_2$, we have
    \begin{align*}
        \pr{Y_1 + Y_2 > j}
         & = p_1 \frac{(1 - p_1)^j - (1 - p_2)^j}{p_2 - p_1} + (1-p_1)^j 
         = \frac{p_2 (1 - p_1)^j - p_1 (1 - p_2)^j}{p_2 - p_1}.
    \end{align*}

    \textit{2.} Now we suppose that $p_1 + p_2 = c$ and $p_1 < p_2$. Consider
    \begin{align*}
        h(p_1) \triangleq f_{2, j}(p_1, c - p_1) = \frac{(c - p_1)(1 - p_1)^j - p_1(1 + p_1 - c)^j}{c - 2p_1},
    \end{align*}
    where $p_1 \in (0, c/2)$.
    It is clear that
    \[
    h(c/2) \triangleq \lim_{p_1 \rightarrow c/2} h(p_1) = f_{2, j}\left(c/2, \, c/2\right).
    \]
    If $h'(p_1) < 0$ for $p_1 \in (0, c/2)$, then there holds $h(p_1) > h(c/2)$, i.e.,
    \[f_{2, j}(p_1, p_2) > f_{2, j} \left( \frac{p_1 + p_2}{2}, \frac{p_1 + p_2}{2} \right).\]

    Note that
    \begin{align*}
        h'(p_1) & = \frac{-(1 - p_1)^j - j(c - p_1)(1 - p_1)^{j-1}
        - (1 + p_1 - c)^j - jp_1(1 + p_1 - c)^{j-1}}{c - 2p_1}                      \\
                & \quad + 2\frac{(c - p_1)(1 - p_1)^j - p_1(1 + p_1 - c)^j}{(c - 2p_1)^2} \\
                & = \frac{[c(1 - p_1) - j (c - p_1)(c - 2p_1)](1 - p_1)^{j-1}
        - [c(1 + p_1 - c) + jp_1(c - 2p_1)](1 + p_1 - c)^{j-1}}{(c - 2p_1)^2}.
    \end{align*}
    Hence $h'(p_1) < 0$ is equivalent to
    \begin{align}\label{lem:two-geo:eq}
        \frac{c(1 - p_1) - j (c - p_1)(c - 2p_1)}{c(1 + p_1 - c) + jp_1(c - 2p_1)}
        < \left( \frac{1 + p_1 - c}{1 - p_1} \right)^{j-1}.
    \end{align}
    Observe that
    \begin{align*}
        \quad \frac{c(1 - p_1) - j (c - p_1)(c - 2p_1)}{c(1 + p_1 - c) + jp_1(c - 2p_1)}
        = 1 - \frac{(j-1)c(c - 2p_1)}{c(1 + p_1 - c) + jp_1(c - 2p_1)}
        = 1 - \frac{j-1}{\frac{1 + p_1 - c}{c - 2p_1} + j\frac{p_1}{c}}.
    \end{align*}
    Letting $x = \frac{1 + p_1 - c}{c - 2p_1}$, inequality (\ref{lem:two-geo:eq}) can be written as
        $ 1 - \frac{j-1}{x + j p_1 / c} < \left(\frac{x}{x+1}\right)^{j-1} $.
    Note that
    \begin{align*}
         & \quad (x+1)^j - j/2 (x+1)^{j-1} \\
         & = x^j + \sum_{l=0}^{j-1} \left[ \binom{j}{l} - \frac{j}{2} \binom{j-1}{l} \right] x^l                 \\
         & = x^j + \sum_{l=0}^{j-1} \left[ \left(\frac{j}{j - l} - \frac{j}{2} \right)\binom{j-1}{l} \right] x^l \\
         & \le x^j + j/2 x^{j-1} = x^{j-1} (x + j/2).
    \end{align*}
    That is
        $ (x + 1)^{j - 1}(x + j/2) - (j-1)(x+1)^{j-1} \le x^{j-1}(x + j/2) $.
    Consequently, we have
    \begin{align*}
        \left(\frac{x}{x+1}\right)^{j-1} \ge 1 - \frac{j-1}{x + j/2} > 1 - \frac{j-1}{x + j p_1 / c},
    \end{align*}
    which is the result we desired.
\end{proof}
Now we can give the proof of Lemma \ref{lem:geo:aux}.
\begin{proof}[Lemma \ref{lem:geo:aux}]
    We first prove the continuity of the function $f_{m, j}$. Actually, we can prove that
    \begin{align}\label{lem:geo:aux:continuous}
        |f_{m, j}(p_1, p_2, \dots, p_m) - f_{m, j}(p'_1, p_2, \dots, p_m)| \le j |p_1 - p'_1|.
    \end{align}
    Recall that
        $ f_{m, j}(p_1, p_2, \dots, p_m) \triangleq \pr{\sum_{i=1}^m Y_i > j}, $
    where $\{Y_i \sim \geo(p_i)\}_{i=1}^m$ are independent geometric random variables.
    Let $Y'_1 \sim \geo(p'_1)$ be independent of the $Y_i 
    $, then by mean value theorem for $1 \le l \le j - 1$, there holds
    \begin{align*}
        \left|\pr{Y_1 > l} - \pr{Y'_1 > l}\right|
         & = \left| (1 - p_1)^l - (1 - p_1')^l \right|                      \\
         & = \left| l(1 - \xi)^{l - 1} \right| \left| p_1 - p'_1 \right|    \\
         & \le l \left| p_1 - p'_1 \right| \le j \left| p_1 - p'_1 \right|,
    \end{align*}
    where $\xi$ lies on the interval $[p_1, p'_1]$.
    Consequently, with $Z \triangleq \sum_{i=2}^m Y_i$, we conclude that
    \begin{align*}
         & \quad \left| f_{m, j}(p_1, p_2, \dots, p_m) - f_{m, j}(p'_1, p_2, \dots, p_m) \right| \\
         & = \left|\pr{Y_1 + Z > j} - \pr{Y'_1 + Z > j} \right|                                    \\
         & = \bigg|\sum_{l = 1}^{j-1} \pr{Z = l} \pr{Y_1 > j - l} + \pr{Z > j - 1}
        - \sum_{l = 1}^{j-1} \pr{Z = l} \pr{Y'_1 > j - l} + \pr{Z > j - 1}\bigg|                   \\
         & \le \sum_{l = 1}^{j-1} \pr{Z = l} \bigg| \pr{Y_1 > j - l} - \pr{Y'_1 > j - l} \bigg|    \\
         & \le j |p_1 - p'_1| \sum_{l = 1}^{j-1} \pr{Z = l}                                        \\
         & = j |p_1 - p'_1| \pr{1 \le Z \le j - 1} \le j |p_1 - p'_1|,
    \end{align*}
    where we have used $\pr{Y_1 > 0} = 1$ in the second equality.

    Following from Equation (\ref{lem:geo:aux:continuous}) and the symmetry of the function $f_{m, j}$, we know that
    \begin{align*}
        |f_{m, j}(p_1, p_2, \dots, p_m) - f_{m, j}(p'_1, p'_2, \dots, p'_m)| \le j \sum_{i=1}^m |p_i - p'_i|,
    \end{align*}
    which implies that $f_{m, j}$ is a continuous function.

    Furthermore, following the way we obtain the Equation (\ref{lem:geo:aux:continuous}) and the fact that
    \begin{align*}
        | (1 - p_1)^{l} - 1 | \le l p_1, ~~ l = 1, 2, \cdots, j - 1,
    \end{align*}
    we have
        $|f_{m, j}(p_1, p_2, \dots, p_m) - 1| \le j p_1.$
    Moreover, by symmetry of the function $f_{m, j}$, it holds that
    \begin{align}\label{lem:geo:aux:limit}
        1 - f_{m, j}(p_1, p_2, \dots, p_m) \le j \min\{p_1, p_2, \dots, p_m\}.
    \end{align}

    For $1 \le j \le m-1$, we have $f_{m, j} (p_1, p_2, \dots, p_m) \equiv 1$ and the desired result is apparent. Then Lemma \ref{lem:two-geo} implies the desired result holds for $m = 2$.

    For $m \ge 3$, $j \ge m$ and $c \in (0, m)$, our goal is to find the minimal value of $f_{m, j} (p_1, p_2, \dots, p_m)$ with the domain
    \begin{align*}
        \fB = \left\{ (p_1, p_2, \dots, p_m) \bigg\vert \sum_{i=1}^m p_m = c,~ p_i \in (0, 1] \text{ for } i \in [m] \right\}.
    \end{align*}
    For $j \ge m$, note that
    \begin{align*}
         & \quad f_{m, j} (c/m, c/m, \dots, c/m) = \pr{\sum_{i=1}^m Z_i > j} \le \pr{\sum_{i=1}^m Z_i > m} \\
         & = 1 - \pr{\sum_{i=1}^m Z_i \le m} = 1 - \pr{Z_1 = 1, Z_2 = 1, \cdots, Z_m = 1}                   \\
         & = 1 - \left(\frac{c}{m}\right)^m < 1,
    \end{align*}
    where $\{Z_i \sim \geo(c/m)\}_{i=1}^m$ are independent random variables, and we have used that $\pr{Z_i \ge 1} = 1$ for $i \in [m]$.

    By Equation (\ref{lem:geo:aux:limit}), if there is an index $i$ satisfies $p_i < \delta \triangleq \frac{1 - f_{m, j} (c/m, c/m, \dots, c/m)}{j} > 0$, then we have
    \begin{align*}
        f_{m, j}(p_1, p_2, \dots, p_m) \ge 1 - j p_i > f_{m, j} (c/m, c/m, \dots, c/m).
    \end{align*}
    Therefore, we just need to find the minimal value of $f_{m, j} (p_1, p_2, \dots, p_m)$ with the domain
    \begin{align*}
        \fB' = \left\{ (p_1, p_2, \dots, p_m) \bigg\vert \sum_{i=1}^m p_m = c,~ p_i \in [\delta, 1] \text{ for } i \in [m] \right\},
    \end{align*}
    which is a compact set.
    Hence, by continuity of $f_{m, j}$, we know that there exists $(q_1, q_2, \dots, q_m) \in \fB'$ such that
    \begin{align*}
        \min_{(p_1, p_2, \dots, p_m) \in \fB'} f_{m, j} (p_1, p_2, \dots, p_m) = f_{m, j} (q_1, q_2, \dots, q_m).
    \end{align*}
    Suppose that there are indexes $k, l \in [m]$ such that $q_k < q_l$. By symmetry of the function $f_{m, j}$, we assume that $q_1 < q_2$. \\[0.3cm]
    Let $\{X'_1, X'_2\} \cup \{X_i\}_{i=1}^m$ be independent geometric random variables and $X'_1, X'_2 \sim \geo\left( \frac{q_1 + q_2}{2} \right)$, $X_i \sim \geo(q_i)$ for $i \in [m]$.
    Denoting $Z' = \sum_{i=3}^m X_i$, we have
    \begin{align*}
         & \quad f_{m, j} (q_1, q_2, \dots, q_m)                                         \\
         & = \pr{X_1 + X_2 + Z' > j}                                                      \\
         & = \sum_{l = 1}^{j-1} \pr{Z' = l} \pr{X_1 + X_2 > j - l} + \pr{Z' > j - 1}      \\
         & \ge \sum_{l = 1}^{j-1} \pr{Z' = l} \pr{X'_1 + X'_2 > j - l} + \pr{Z' > j - 1}  \\
         & =  \pr{X'_1 + X'_2 + Z' > j}                                                   \\
         & = f_{m, j} \left(\frac{q_1 + q_2}{2}, \frac{q_1 + q_2}{2}, \dots, q_m\right),
    \end{align*}
    where the inequality is according to Lemma \ref{lem:two-geo}. \\[0.3cm]
    However, for $l = m - 2$, it holds that $\pr{Z' = m - 2} = 1 - \prod_{i=2}^m q_i > 0 $ and $\pr{X_1 + X_2 > j - m + 2} > \pr{X'_1 + X'_2 > j - m + 2}$ by Lemma \ref{lem:two-geo}, which implies that
    \begin{align*}
        f_{m, j} (q_1, q_2, \dots, q_m)
        > f_{m, j} \left(\frac{q_1 + q_2}{2}, \frac{q_1 + q_2}{2}, \dots, q_m\right).
    \end{align*}
    Note that $\frac{q_1 + q_2}{2} + \frac{q_1 + q_2}{2} + \sum_{i=2}^m q_i = c$ and $\frac{q_1 + q_2}{2} \in [\delta, 1]$. Hence we have
    \begin{align*}
        \left(\frac{q_1 + q_2}{2}, \frac{q_1 + q_2}{2}, \dots, q_m\right) \in \fB',
    \end{align*}
    which contradicts the fact that $(q_1, q_2, \dots, q_m)$ is the optimal point in $\fB'$.

    Therefore, we can conclude that
    \begin{align*}
        f_{m, j} (p_1, p_2, \dots, p_m) \ge f_{m, j} \left( \frac{\sum_{i=1}^m p_i}{m}, \frac{\sum_{i=1}^m p_i}{m}, \dots, \frac{\sum_{i=1}^m p_i}{m} \right).
    \end{align*}
This completes the proof.
\end{proof}

\section{Technical Lemmas}
In this section, we present some technical lemmas.
\begin{lemma}\label{lem:scale}
    Suppose $f(\vx, \vy)$ is $(\mu_x, \mu_y)$-convex-concave and $L$-smooth, then the function $\hat{f}(\vx, \vy) = \lambda f(\vx/\beta, \vy/\beta)$ is $\left( \frac{\lambda \mu_x}{\beta^2}, \frac{\lambda \mu_y}{\beta^2} \right)$-convex-concave and $\frac{\lambda L}{\beta^2}$-smooth. 
    Moreover, if $\{f_i(\vx, \vy)\}_{i=1}^n$ is $L'$-average smooth, then the function class $\{\hat{f}_i(\vx, \vy) \triangleq \lambda f_i(\vx/\beta, \vy/\beta)\}_{i=1}^n$ is $\frac{\lambda L'}{\beta^2}$-average smooth.
\end{lemma}

\begin{lemma}\label{lem:proj}
    Suppose that $\fX = \{\vx \in \BR^d: \norm{\vx} \le R_x \}$, then we have
    \begin{align*}
        \fP_{\fX}(\vx) = \begin{cases}
            \vx, ~~~&\text{ if } \vx \in \fX, \\
            \frac{R_X}{\norm{\vx}} \vx, ~~~&\text{ otherwise. }
        \end{cases}
    \end{align*}
\end{lemma}
\begin{remark}
    By Lemma \ref{lem:proj}, vectors $\fP_{\fX}(\vx)$ and $\vx$ are always collinear.
\end{remark}

\begin{proposition}[Lemmas 2,3,4, \cite{carmon1711lower}]\label{prop:nonconvex:prop:base}
Let $G_{\text{NC}}: \BR^{m+1} \rightarrow \BR$ be
\begin{align*}
    G_{\text{NC}}(\vx; \omega, m+1) = \frac{1}{2} \norm{\mB(m+1, {\omega}, 0) \vx}^2 - {\omega^2} \inner{\ve_1}{\vx} + \omega^4 \sum_{i=1}^{m} \Gamma (x_i).
\end{align*}
For any $0 < \omega \le 1$, 
it holds that
\begin{enumerate}
    \item $\Gamma(x)$ is $180$-smooth and  $[-45 (\sqrt{3} - 1)]$-weakly convex.
    \item $G_{\text{NC}}(\vzero_{m+1}; \omega, m+1) - \min_{\vx \in \BR^{m+1}} G_{\text{NC}}(\vx; \omega, m+1) \le {\omega^2} / 2 + 10 \omega^4 m$.
    \item For any $\vx \in \BR^{m+1}$ such that $x_{m} = x_{m+1} = 0$, $G_{\text{NC}}(\vx; \omega, m)$ is $(4 + 180 \omega^4)$-smooth and $[-45 (\sqrt{3} - 1) \omega^4]$-weakly convex and
    \[\norm{\nabla G_{\text{NC}}(\vx; \omega, m)} \ge \omega^{3}/4.\]
\end{enumerate}
\end{proposition}

\begin{lemma}\label{lem:solution:z}
Suppose that $0 < \lambda_2 < (2 + 2\sqrt{2}) \lambda_1$, 
then $z = 0$ is the only real solution to the equation 
\begin{align}\label{eq:solution:z}
\lambda_1 z + \lambda_2 \frac{z^2(z-1)}{1+ z^2} = 0.    
\end{align}
\end{lemma}
\begin{proof}
Since $0 < \lambda_2 < (2 + 2\sqrt{2}) \lambda_1$, we have 
\[\lambda_2^2 - 4 \lambda_1 (\lambda_1 + \lambda_2) < 0,\]
and consequently, for any $z$, $(\lambda_1 + \lambda_2)z^2 - \lambda_2 z + \lambda_1 > 0$.

On the other hand, we can rewrite Equation (\ref{eq:solution:z}) as 
\begin{align*}
    z \big((\lambda_1 + \lambda_2)z^2 - \lambda_2 z + \lambda_1\big) = 0.
\end{align*}
Clearly, $z = 0$ is the only real solution to Equation (\ref{eq:solution:z}).
\end{proof}

\begin{lemma}\label{lem:solution:z1z2}
Suppose that $0 < \lambda_2 < (2 + 2\sqrt{2}) \lambda_1$ and $\lambda_3 > 0$, 
then $z_1 = z_2 = 0$ is the only real solution to the equation 
\begin{align}\label{eq:solution:z1z2}
\begin{cases}
\lambda_1 z_1 + \lambda_3 (z_1 - z_2) + \lambda_2 \frac{z_1^2(z_1-1)}{1+ z_1^2} = 0.  \\
\lambda_1 z_2 + \lambda_3 (z_2 - z_1) + \lambda_2 \frac{z_2^2(z_2-1)}{1+ z_2^2} = 0.
\end{cases}
\end{align}
\end{lemma}

\begin{proof}
If $z_1 = 0$, then $z_2 = 0$. So let assume that $z_1 z_2 \neq 0$.
Rewrite the first equation of Equations (\ref{eq:solution:z1z2}) as
\begin{align*}
    \frac{\lambda_1 + \lambda_3}{\lambda_3} + \frac{\lambda_2}{\lambda_3} \frac{z_1(z_1-1)}{1+ z_1^2} = \frac{z_2}{z_1}
\end{align*}
Note that 
\begin{align*}
    \frac{1 - \sqrt{2}}{2} \le \frac{z(z-1)}{1+z^2}.
\end{align*}
Thus, we have
\begin{align*}
    \frac{\lambda_1 + \lambda_3}{\lambda_3} + \frac{\lambda_2}{\lambda_3} \frac{1 - \sqrt{2}}{2} 
    \le \frac{z_2}{z_1}.
\end{align*}
Similarly, it also holds
\begin{align*}
    \frac{\lambda_1 + \lambda_3}{\lambda_3} + \frac{\lambda_2}{\lambda_3} \frac{1 - \sqrt{2}}{2} 
    \le \frac{z_1}{z_2}.
\end{align*}

By $0 < \lambda_2 < (2 + 2\sqrt{2}) \lambda_1$, we know that $\lambda_1 + \frac{1 - \sqrt{2}}{2}\lambda_2 > 0$. 
Thus
\begin{align*}
    \frac{\lambda_1 + \lambda_3}{\lambda_3} + \frac{\lambda_2}{\lambda_3} \frac{1 - \sqrt{2}}{2} > 1.
\end{align*}
Since $z_1/z_2 > 1$ and $z_2 / z_1 > 1$ can not hold at the same time, so we get a contradiction.
\end{proof}

\begin{lemma}\label{lem:appendix:gap}
    Define the function
    \begin{align}\label{eq:J_func}
        J_{k, \beta}(y_1, y_2, \dots, y_k) \triangleq y_k^2 + \sum_{i=2}^{k}(y_i - y_{i - 1})^2 + (y_1 - \beta)^2.
    \end{align}
    Then we have $\min J_{k, \beta}(y_1, \dots, y_k) = \frac{\beta^2}{k+1}$.
\end{lemma}
\begin{proof}
    Letting the gradient of $J_{k, \beta}$ equal to zero, we get
    \begin{align*}
        2y_k - y_{k - 1} = 0, ~~ 2y_1 - y_2 - \beta = 0, \text{~and~}
        y_{i+1} - 2y_{i} + y_{i-1} = 0, ~ \text{for~} i = 2, 3, \dots, k-1.
    \end{align*} 
    That is, 
    \begin{align}
        y_i = \frac{k - i + 1}{k + 1} \beta \text{~for~} i = 1, 2, \dots, k. \label{eq:yJ}
    \end{align} 
    Thus by substituting Equation (\ref{eq:yJ}) into the expression of $J_{k, \beta}(y_1, y_2, \dots, y_k)$, we achieve the desired result.
\end{proof}
\section{Proofs for Section \ref{sec:frame}}\label{appendix:frame}
In this section, we present some omitted proofs in Section \ref{sec:frame}.


\subsection{Proofs of Proposition \ref{prop:convex-concave:base} and Lemma \ref{lem:convex-concave:jump}}\label{appendix:frame:cc}
Let $\widetilde{\mB}(m, \zeta)$ denote the last $m$ rows of $\mB(m, 0, \zeta)$ and $\tilde{\vb}_l (m, \zeta) = \vb_{l}(m, 0, \zeta) $ for $0 \le l \le m$. Note that $\tilde{\vb}_0 (m, \zeta) = \vzero$.
For simplicity, we omit the parameters of $\widetilde{\mB}$, $\tilde{\vb}_l$ and $\tilde{r}_i$. Then we have $\widetilde{\mB} = (\tilde{\vb}_1, \tilde{\vb}_2, \dots, \tilde{\vb}_m)^\top$.

Recall that
$$\fL_i = \{ l: 0 \le l \le m, l \equiv i - 1 (\bmod n) \},\, i = 1, 2, \dots, n.$$ 
For $1 \le i \le n$, let $\widetilde{\mB}_i$ be the submatrix of $\widetilde{\mB}$ whose rows are $\big\{ \tilde{\vb}_l^\top \big\}_{l \in \fL_i}$.
Note that $\widetilde{\mB} = \sum_{l=1}^m \ve_l \tilde{\vb}_l^\top$ and $\widetilde{\mB}_i = \sum_{l \in \fL_i} \ve_l \tilde{\vb}_l^\top$.
Then $\tilde{r}_i$ can be written as
\begin{align*}
\tilde{r}_i(\vx, \vy) &= n \inner{\vy}{\widetilde{\mB}_i \vx} + \frac{\tilde{c}_1}{2} \norm{\vx}^2 - \frac{\tilde{c}_2}{2} \norm{\vy}^2 - n \inner{\ve_1}{\vx} \bone_{\{i=1\}}.
\end{align*}
\begin{proof}[Proposition \ref{prop:convex-concave:base}]
Firstly, it is clear that $\tilde{r}_i$ is $(\tilde{c}_1, \tilde{c}_2)$-convex-concave.

Next, note that for $l_1, l_2 \in \fL_i$ and $l_1 \neq l_2$, we have $|l_1 - l_2| \ge n \ge 2$, thus $\tilde{\vb}_{l_1}^{\top} \tilde{\vb}_{l_2} = 0$.
Since $\zeta \le 2$, $\tilde{\vb}_l^{\top} \tilde{\vb}_l \le 2$,
it follows that
\begin{align*}
    \norm{\sum_{l \in \fL_i} \tilde{\vb}_l \ve_l^{\top} \vy}^2 
    &= \sum_{l \in \fL_i} \vy^{\top} \ve_l \tilde{\vb}_l^{\top} \tilde{\vb}_l \ve_l^{\top} \vy 
    \le 2 \sum_{l \in \fL_i} \left( \ve_l^{\top} \vy \right)^2 
    \le 2 \norm{\vy}^2, \\
    \norm{\sum_{l \in \fL_i} \ve_l \tilde{\vb}_l^{\top} \vx}^2 
    &= \sum_{l \in \fL_i} \left( \tilde{\vb}_l^{\top} \vx \right)^2 
    \le \sum_{l \in \fL_i \backslash \{ m \}} 2 \left( x_l^2 + x_{l+1}^2\right) + \zeta^2 x_m^2 \bone_{ \{ m \in \fL_i \} }
    \le 2 \norm{\vx}^2.
\end{align*}
Note that
\begin{align*}
    \nabla_\vx \tilde{r}_i (\vx, \vy)
    = &\, n \widetilde{\mB}_i^\top \vy + \tilde{c}_1 \vx - n \ve_1 \bone_{ \{ i=1 \} }, \\
    \nabla_\vy \tilde{r}_i (\vx, \vy)
    = &\, n \widetilde{\mB}_i \vx - \tilde{c}_2 \vy.
\end{align*}
With $\vu = \vx_1 - \vx_2$ and $\vv = \vy_1 - \vy_2$, we have
\begin{align*}
    &\quad \, \norm{\nabla \tilde{r}_i(\vx_1, \vy_1) - \nabla \tilde{r}_i(\vx_2, \vy_2)}^2 \\
    &= \norm{\nabla_{\vx} \tilde{r}_i(\vx_1, \vy_1) - \nabla_{\vx} \tilde{r}_i(\vx_2, \vy_2)}^2 + \norm{\nabla_{\vy} \tilde{r}_i(\vx_1, \vy_1) - \nabla_{\vy} \tilde{r}_i(\vx_1, \vy_2)}^2 \\
    &= \norm{\tilde{c}_1 \vu + n \sum_{l \in \fL_i} \tilde{\vb}_l \ve_l^{\top}\vv}^2 + \norm{\tilde{c}_2 \vv - n \sum_{l \in \fL_i} \ve_l \tilde{\vb}_l^{\top} \vu }^2 \\
    &\le 2 \left( \tilde{c}_1^2 \norm{\vu}^2 + \tilde{c}_2^2 \norm{\vv}^2 \right) + 2 n^2 \norm{\sum_{l \in \fL_i} \tilde{\vb}_l \ve_l^{\top}\vv}^2 + 2 n^2 \norm{\sum_{l \in \fL_i} \ve_l \tilde{\vb}_l^{\top} \vu}^2 \\
    &\le 2 \left( \tilde{c}_1^2 \norm{\vu}^2 + \tilde{c}_2^2 \norm{\vv}^2 \right) + 4 n^2 \sum_{l \in \fL_i} \left( \ve_l^{\top}\vv \right)^2 + 2 n^2 \sum_{l \in \fL_i} \left( \tilde{\vb}_l^{\top} \vu \right)^2 \\
    &\le \left( 2 \max\{\tilde{c}_1, \tilde{c}_2\}^2 + 4 n^2 \right) \left( \norm{\vu}^2 + \norm{\vv}^2 \right),
\end{align*}
where the first inequality follows from $(a + b)^2 \le 2(a^2 + b^2)$.
In addition,
\begin{align*}
    &\quad \, \frac{1}{n} \sum_{i=1}^n \norm{\nabla \tilde{r}_i(\vx_1, \vy_1) - \nabla \tilde{r}_i(\vx_2, \vy_2)}^2 \\
    & \le 2 \left( \tilde{c}_1^2 \norm{\vu}^2 + \tilde{c}_2^2 \norm{\vv}^2 \right) + 4 n \sum_{l=1}^{m} \left( \ve_l^{\top}\vv \right)^2 + 2 n \sum_{l=1}^m \left( \tilde{\vb}_l^{\top} \vu \right)^2 \\
    & \le 2 \left( \tilde{c}_1^2 \norm{\vu}^2 + \tilde{c}_2^2 \norm{\vv}^2 \right) + 4 n \norm{\vv}^2 + 8n \norm{\vu}^2 \\
    & \le \left( 2 \max\{\tilde{c}_1, \tilde{c}_2\}^2 + 8 n \right) \left(\norm{\vu}^2 + \norm{\vv}^2 \right).
\end{align*}
Thus, $\tilde{r}_i$ is $\sqrt{4 n^2 + 2 \max \{ \tilde{c}_1, \tilde{c}_2 \}^2 }$-smooth, and $\{ \tilde{r}_i \}_{i=1}^n$ is $\sqrt{8n + 2 \max \{ \tilde{c}_1, \tilde{c}_2 \}^2 }$-average smooth.
\end{proof}

\begin{proof}[Proof of Lemma \ref{lem:convex-concave:jump}]
Note that 
\begin{align*}
    \ve_l \tilde{\vb}_l^\top \vx =
    \begin{cases}
        (x_{l} - x_{l+1}) \ve_l, & 1 \le l < m, \\
        \zeta x_m \ve_m, & l=m,
    \end{cases}
    \mbox{  and  }
    \tilde{\vb}_l \ve_l^\top \vy =
    \begin{cases}
        y_l (\ve_{l} - \ve_{l+1}), & 1 \le l < m, \\
        \zeta y_m \ve_m, & l=m.
    \end{cases}
\end{align*}
For $\vx, \vy \in \fF_k$ with $1 \le k < m$, we have
\begin{align}
\label{proof:convex-concave:ebx}
    \ve_l \tilde{\vb}_l^\top \vx \in
    \begin{cases}
        \fF_k, & l = k, \\
        \fF_{k-1}, & l \neq k.
    \end{cases}
\mbox{  and  }
    \tilde{\vb}_l \ve_l^\top \vy \in
    \begin{cases}
        \fF_{k+1}, & l = k, \\
        \fF_{k}, & l \neq k.
    \end{cases}
\end{align}
Recall that
\begin{align*}
    \nabla_\vx \tilde{r}_i (\vx, \vy)
    = &\, n  \sum_{l \in \fL_i} \tilde{\vb}_l \ve_l^\top \vy  + \tilde{c}_1 \vx - n \ve_1 \bone_{ \{ i=1 \} }, \\
    \nabla_\vy \tilde{r}_i (\vx, \vy)
    = &\, n \sum_{l \in \fL_i} \ve_l \tilde{\vb}_l^\top \vx - \tilde{c}_2 \vy.
\end{align*}
By Inclusions (\ref{proof:convex-concave:ebx}),
we have the following results.
\begin{enumerate}
    \item Suppose that $\vx, \vy \in \fF_0$. 
    It holds that 
    $ \nabla_\vx \tilde{r}_1(\vx, \vy) = n \ve_1 \in \fF_1 $, 
    $ \nabla_\vx \tilde{r}_j(\vx, \vy) = \vzero $ for $j \ge 2$ 
    and $ \nabla_\vy \tilde{r}_j(\vx, \vy) = \vzero $ for any $j$.
    
    \item Suppose that $\vx \in \fF_1$ and $\vy \in \fF_0$ and $1 \in \fL_i$. 
    It holds that $\nabla_\vx \tilde{r}_j(\vx, \vy) = \tilde{c}_1 \vx + n \ve_1 \bone_{ \{i=1\}} \in \fF_1 $ for any $j$,
    $\nabla_\vy \tilde{r}_i (\vx, \vy) \in \fF_1$ 
    and $\nabla_\vy \tilde{r}_j (\vx, \vy) = \vzero$ for $j \neq i$.
    
    
    \item Suppose that $\vx \in \fF_{k+1}$, $\vy \in \fF_k$, $1 \le k < m$ and $k+1 \in \fL_i$. 
    It holds that $\nabla_\vx \tilde{r}_j(\vx, \vy) \in \fF_{k+1} $ for any $j$, 
    $\nabla_\vy \tilde{r}_i (\vx, \vy) \in \fF_{k+1}$ 
    and $\nabla_\vy \tilde{r}_j (\vx, \vy) \in \fF_k $ for $j \neq i$.
\end{enumerate}

\vskip 5pt

Now we turn to consider $(\vu_i, \vv_i) = \prox_{\tilde{r}_i}^\gamma (\vx, \vy)$. We have
\begin{align*}
    \nabla_\vx \tilde{r}_i (\vu_i, \vv_i) + \frac{1}{\gamma} (\vu_i - \vx) & = \vzero, \\
    \nabla_\vy \tilde{r}_i (\vu_i, \vv_i) - \frac{1}{\gamma} (\vv_i - \vy) & = \vzero,
\end{align*}
that is
\begin{align*}
    \begin{bmatrix}
    \left( \tilde{c}_1 + \frac{1}{\gamma }\right) \mI_m & n \widetilde{\mB}_i^\top \\
    - n \widetilde{\mB}_i & \left( \tilde{c}_2 + \frac{1}{\gamma} \right) \mI_m
    \end{bmatrix}
    \begin{bmatrix}
    \vu_i \\ \vv_i
    \end{bmatrix}
    =
    \begin{bmatrix}
    \tilde{\vx}_i \\ \tilde{\vy}
    \end{bmatrix},
\end{align*}
where $\tilde{\vx}_i = \vx / \gamma + n \ve_1 \bone_{\{i=1\}}$ and $ \tilde{\vy} = \vy / \gamma$.
Recall that for $l_1, l_2 \in \fL_i$ and $l_1 \neq l_2$, $\tilde{\vb}_{l_1}^\top \tilde{\vb}_{l_2} = 0 $. 
It follows that
\begin{align*}
    \widetilde{\mB}_i \widetilde{\mB}_i^\top
    = \left( \sum_{l \in \fL_i} \ve_l \tilde{\vb}_l^\top \right) \left( \sum_{l \in \fL_i}  \tilde{\vb}_l \ve_l^\top \right)
    = \sum_{l \in \fL_i} \ve_l \tilde{\vb}_l^\top \tilde{\vb}_l \ve_l^\top,
\end{align*}
which is a diagonal matrix.
Assuming that 
\begin{align*}
    \mD_i \triangleq
    \left( \tilde{c}_2 + \frac{1}{\gamma} \right) \mI_m + 
    \frac{n^2}{\tilde{c}_1 + 1 / \gamma} \widetilde{\mB}_i \widetilde{\mB}_i^\top
    = \diag \left( d_{i,1}, d_{i,2}, \dots, d_{j,m} \right),
\end{align*}
we have
\begin{align}
    \begin{bmatrix}
    \vu_i \\ \vv_i
    \end{bmatrix}
    & = \begin{bmatrix}
    \left( \tilde{c}_1 + \frac{1}{\gamma }\right) \mI_m & n \widetilde{\mB}_i^\top \\
    - n \widetilde{\mB}_i & \left( \tilde{c}_2 + \frac{1}{\gamma} \right) \mI_m
    \end{bmatrix}^{-1}
    \begin{bmatrix}
    \tilde{\vx}_i \\ \tilde{\vy}
    \end{bmatrix} \nonumber \\
    & = \begin{bmatrix}
    \frac{1}{\tilde{c}_1 + 1 / \gamma} \mI_m 
    - \frac{n^2}{ \left( \tilde{c}_1 + 1 / \gamma \right)^2 }  \widetilde{\mB}_i^\top \mD_i^{-1} \widetilde{\mB}_i &
    - \frac{n}{\tilde{c}_1 + 1 / \gamma} \widetilde{\mB}_i^\top \mD_i^{-1} \\
    \frac{n}{\tilde{c}_1 + 1 / \gamma} \mD_i^{-1} \widetilde{\mB}_i &
    \mD_i^{-1}
    \end{bmatrix}
    \begin{bmatrix}
    \tilde{\vx}_i \\ \tilde{\vy}
    \end{bmatrix} \nonumber \\
    & = \begin{bmatrix}
    \frac{1}{\tilde{c}_1 + 1 / \gamma} \tilde{\vx}_i 
    - \frac{n^2}{ \left( \tilde{c}_1 + 1 / \gamma \right)^2 } \sum_{l \in \fL_i} d_{i.l}^{-1} \tilde{\vb}_l \tilde{\vb}_l^\top \tilde{\vx}_i 
    - \frac{n}{\tilde{c}_1 + 1 / \gamma} \sum_{l \in \fL_i} \tilde{\vb}_l \ve_l^\top \mD_i^{-1} \tilde{\vy} \\
    \frac{n}{\tilde{c}_1 + 1 / \gamma} \sum_{l \in \fL_i} d_{i.l}^{-1} \ve_l \tilde{\vb}_l^\top \tilde{\vx}_i + \mD_i^{-1} \tilde{\vy} \label{proof:cc:uv-equ}
    \end{bmatrix}.
\end{align}
Note that for $1 \le k \le m$, $\vy \in \fF_k $ implies $\mD_i^{-1} \tilde{\vy} \in \fF_k$ and $\vx \in \fF_k $ implies $ \tilde{\vx}_i \in \fF_k$.
And recall that 
\begin{align*}
    \tilde{\vb}_l \tilde{\vb}_l^\top \vx
    = \begin{cases}
        (x_l - x_{l+1}) (\ve_l - \ve_{l+1}), & l < m, \\
        \zeta^2 x_m \ve_m, & l = m.
    \end{cases}
\end{align*}
Then for $\vx \in \fF_k$ with $1 \le k < m$, we have
\begin{align}
\label{proof:convex-concave:bbx}
    \tilde{\vb}_l \tilde{\vb}_l^\top \vx 
    \in \begin{cases}
        \fF_{k+1}, & l = k, \\
        \fF_k, & l \neq k.
    \end{cases}
\end{align}
By Inclusions (\ref{proof:convex-concave:ebx}),
(\ref{proof:convex-concave:bbx}) and Equations (\ref{proof:cc:uv-equ}), we have the following results.
\begin{enumerate}
    \item Suppose that $\vx, \vy \in \fF_0$. 
    It holds that 
    $\tilde{\vx}_1 \in \fF_1$ and $\tilde{\vx}_j = \vzero$ for $j \ge 2$, 
    which implies $\vu_1 \in \fF_1$ and $\vu_j = \vzero $ for $j \ge 2$. 
    Moreover, $\vv_j = \vzero$ for any $j$.
    
    \item Suppose that $\vx \in \fF_1$, $\vy \in \fF_0$ and $1 \in \fL_i$. 
    It holds that $\vu_i \in \fF_2$, $\vv_i \in \fF_1$ and $\vu_j \in \fF_1$, $\vv_j \in \fF_0$ for $j \neq i$.
    
    
    \item Suppose that $\vx \in \fF_{k+1}$, $\vy \in \fF_k$, $1 \le k < m-1$ and $k + 1 \in \fL_i$. 
    It holds that $\vu_i \in \fF_{k+2} $, $\vv_i \in \fF_{k+1}$ and $\vu_j \in \fF_{k+1}$, $\vv_j \in \fF_k $ for $j \neq i$.
\end{enumerate}
This completes the proof.
\end{proof}

\subsection{Proofs of Corollary \ref{coro:convex-concave:stopping-time} and Lemma \ref{lem:minimiax:base} }\label{appendix:frame:other}
\begin{proof}[Proof of Corollary \ref{coro:convex-concave:stopping-time}]
First, we note that by Lemma \ref{lem:proj}, the projection operations $\fP_{\fX} (\vx)$ and $\fP_{\fY} (\vy)$ do not affect the nonzero elements of the vectors $\vx$ and $\vy$. 

Then we prove the first claim by induction on $k$.
Clearly, it holds that $(\vx_0, \vy_0) = (\vzero, \vzero) \in \fF_{0} \times \fF_{-1}$.
Suppose that $(\vx_t, \vy_t) \in \fF_{k(t)-1} \times \fF_{k(t)-2}$ for any $t \le t_0$ where $k(t)$ is the positive integer such that $T_{k(t)-1} \le t < T_{k(t)}$.
By Lemma \ref{lem:convex-concave:jump},
for $t < T_{k(t_0)-1}$,
$\nabla r_i^{\mathrm{CC}} (\vx_t, \vy_t), \prox_{r_i^{\mathrm{CC}}}^\gamma (\vx_t, \vy_t ) \in \fF_{k(t_0)-1} \times \fF_{k(t_0)-2} $ for any $i$;
for $T_{k(t_0)-1} \le t \le t_0$, $a_t = 0$ and
$\nabla r_{i_j}^{\mathrm{CC}} (\vx_t, \vy_t), \prox_{r_{i_j}^{\mathrm{CC}}}^\gamma (\vx_t, \vy_t ) \in \fF_{k(t_0)-1} \times \fF_{k(t_0)-2} $ for any $t < j \le t_0$.
It remains to check $\nabla r_{i_{t_0+1}}^{\mathrm{CC}} (\vx_t, \vy_t), \prox_{r_{i_{t_0+1}}^{\mathrm{CC}}}^\gamma (\vx_t, \vy_t )$ for $T_{k(t_0)-1} \le t \le t_0$ and the value of $a_{t_0+1}$. 
By Lemma \ref{lem:convex-concave:jump},
\begin{align*}
    \nabla r^{\mathrm{CC}}_{i_{t_0+1}}(\vx_t, \vy_t),\, \prox_{r^{\mathrm{CC}}_{i_{t_0+1}}}^{\gamma} (\vx_t, \vy_t) \in \begin{cases}
        \fF_{k(t_0)} \times \fF_{k(t_0)-1}, ~&\text{ if } i_{t_0+1} \equiv k(t_0)~(\bmod~n), \\
        \fF_{k(t_0)-1} \times \fF_{k(t_0)-2}, ~&\text{ otherwise.}
    \end{cases}
\end{align*}
Thus, if $i_{t_0+1} \equiv k(t_0)~(\bmod~n) $ or $a_{t_0+1} = 1$, we have $T_{k(t_0)} = t_0+1 < T_{k(t_0)+1}$. Thus, $k(t_0+1) = k(t_0)+1$ and $(\vx_{t_0+1}, \vy_{t_0+1}) \in \fF_{k(t_0+1) - 1} \times \fF_{k(t_0 + 1)-2}$.
Otherwise, we still have $T_{k(t_0)-1} \le t_0+1 < T_{k(t_0)} $.
Thus, $k(t_0+1) = k(t_0) $ and $(\vx_{t_0+1}, \vy_{t_0+1}) \in \fF_{k(t_0+1)-1} \times \fF_{k(t_0+1)-2}$.

Consequently, we have $ (\vx_t, \vy_t) \in \fF_{k(t) - 1} \times \fF_{k(t)-2} $ for any $t$.
Since $k(t)$ is monotone increasing,
we have $(\vx_t, \vy_t) \in \fF_{k-1} \times \fF_{k-2}$ for any $t > T_k$ and $k \ge 1$.

Next, note that 
\begin{align*}
    &\quad \pr{T_{k} - T_{k-1} = s} \\
    &= \pr{ i_{T_{k{-}1} {+} 1} {\neq} k', \dots,  i_{T_{k{-}1} {+} s {-} 1} {\neq} k', a_{T_{k{-}1} {+} 1}=0, \dots, a_{T_{k{-}1} {+} s {-} 1} = 0, i_{T_{k{-}1} {+} s} {=} k' \mbox{ or } a_{T_{k{-}1} {+} s} {=} 1 } \\
    &= (1 - p_{k'})^{s-1} (1-q)^{s-1} (p_{k'} + q - p_{k'}q ),
\end{align*}
where $k' \equiv k (\bmod ~n), 1 \le k' \le n$ and the last equality is due to the independence of $\{ (i_t, a_t) \}_{t \ge 1}$. So $Y_k = T_k - T_{k-1}$ is a geometric random variable with success probability $p_{k'} + q - p_{k'}q$. The independence of $\{Y_k\}_{k \ge 1}$ is just according to the independence of $\{(i_t, a_t)\}_{t\ge 1}$. 
\end{proof}

\begin{proof}[Proof of Lemma \ref{lem:minimiax:base}]
For $t \le N$, we have
\begin{align*}
    &\quad \E \left(\max_{\vv \in \fY} r^{\mathrm{CC}}(\vx_t, \vv) - \min_{\vu \in \fX} r^{\mathrm{CC}}(\vu, \vy_t)\right) \\
    &\ge \E \left(\max_{\vv \in \fY} r^{\mathrm{CC}}(\vx_t, \vv) - \min_{\vu \in \fX} r^{\mathrm{CC}}(\vu, \vy_t) \bigg\vert N < T_{M+1}\right) \pr{N < T_{M+1}} \\
    &\ge 9 \eps \pr{N < T_{M+1}},
\end{align*}
where $T_{M+1}$ is defined in (\ref{def:convex-concave:stopping-time}), and the second inequality follows from Corollary \ref{coro:convex-concave:stopping-time} (if $N < T_{M+1}$, then $\vx_t \in \fF_M$ and $\vy_t \in \fF_{M-1} \subset \fF_{M}$ for $t \le N$).

By Corollary \ref{coro:convex-concave:stopping-time}, $T_{M+1}$ can be written as $T_{M+1} = \sum_{l=1}^{M+1} Y_l$,
where $\{Y_l\}_{1 \le l \le M+1}$ are independent random variables, and $Y_l$ follows a geometric distribution with success probability
$q_l \triangleq p_{l'} + q - p_{l'}q$ where
$l' \equiv l (\bmod ~n)$, $1 \le l' \le n$.
Moreover, recalling that $p_1 \le p_2 \le \cdots \le p_n$, we have $\sum_{l=1}^{M+1} q_l \le 
(M+1) \left( \frac{1}{n} + q \right) \le (M+1) (1 + c_0) / n.
$
Therefore, by Lemma \ref{lem:geo}, we have
$$\pr{T_{M+1} > N} = \pr{\sum_{l=1}^{M+1} Y_l > \frac{(M+1)n}{4 (1 + c_0) }} 
\ge \frac{1}{9},$$
which implies our desired result.
\end{proof}

\subsection{Proofs of Proposiiton \ref{prop:nonconvex-strongly:base} and Lemma \ref{lem:nonconvex-strongly:jump}}\label{appendix:frame:ncc}
Let $\widehat{\mB}(m, \omega)$ denote the first $m$ rows of $\mB(m, \omega, 0)$ by  and $\hat{\vb}_{l}(m, \omega) = \vb_l (m, \omega, 0) $ for $0 \le l \le m$. Note that $\hat{\vb}_m (m, \omega) = \vzero$.
For simplicity, we omit the parameters of $\widehat{\mB}$, $\hat{\vb}_l$ and $\hat{r}_i$.
Then we have $\widehat{\mB} = ( \hat{\vb}_0, \hat{\vb}_1, \dots, \hat{\vb}_{m-1} )^\top$.

Let $G(\vx) \triangleq \sum\limits_{i=1}^{m-1} \Gamma (x_i)$. Recall that
$$\fL_i = \{ l: 0 \le l \le m, l \equiv i - 1 (\bmod n) \},\, i = 1, 2, \dots, n.$$ 
For $1 \le i \le n$, let $\widehat{\mB}_i$ be the submatrix whose rows are $\big\{ \hat{\vb}_l^\top \big\}_{l \in \fL_i}$.
Note that $\widehat{\mB} = \sum_{l=0}^{m-1} \ve_{l+1} \hat{\vb}_l^\top$ and $\widehat{\mB}_i = \sum_{l \in \fL_i} \ve_{l+1} \hat{\vb}_l^\top$.
Then $\hat{r}_i$ can be written as
\begin{align*}
\hat{r}_i(\vx, \vy) &= n \inner{\vy}{\widehat{\mB}_i \vx} - \frac{\hat{c}_1}{2} \norm{\vy}^2 +  \hat{c}_2  G (\hat{c}_3  \vx) - n \inner{\ve_1}{\vy} \bone_{\{i=1\}}.
\end{align*}

\begin{proof}[Proof of Proposition \ref{prop:nonconvex-strongly:base}]
Denote $s_i(\vx, \vy) = \hat{r}_i(\vx, \vy) - \hat{c}_2 G ( \hat{c}_3 \vx )$.
Similar to the proof of Proposition \ref{prop:convex-concave:base},
we can establish that for any $\vx_1, \vx_2, \vy_1, \vy_2$,
\begin{align*}
    \norm{ \nabla s_i(\vx_1, \vy_1) - \nabla s_i (\vx_2, \vy_2) }^2 \le \left( 4 n^2 + 2 \hat{c}_1^2 \right) \left( \norm{\vx_1 - \vx_2}^2 + \norm{\vy_1 - \vy_2}^2 \right),
\end{align*}
and 
\begin{align*}
    \frac{1}{n} \sum_{i=1}^n \norm{ \nabla s_i(\vx_1, \vy_1) - \nabla s_i (\vx_2, \vy_2) }^2 \le \left( 8 n + 2 \hat{c}_1^2 \right) \left( \norm{\vx_1 - \vx_2}^2 + \norm{\vy_1 - \vy_2}^2 \right).
\end{align*}
By Proposition \ref{prop:nonconvex:prop:base} and the inequality $(a+b)^2 \le 2 (a^2 + b^2)$, we conclude that $\hat{r}_i$ is $\left( -45(\sqrt{3} - 1) \hat{c}_2 \hat{c}_3^2, \hat{c}_1 \right)$-convex-concave, 
\begin{align*}
    \norm{ \nabla \hat{r}_i(\vx_1, \vy_1) - \nabla \hat{r}_i (\vx_2, \vy_2) } \le \left( \sqrt{4 n^2 + 2 \hat{c}_1^2} + 180 \hat{c}_2 \hat{c}_3^2 \right) \sqrt{ \norm{\vx_1 - \vx_2}^2 + \norm{\vy_1 - \vy_2}^2 },
\end{align*}
and 
\begin{align*}
    \frac{1}{n} \sum_{i=1}^n \norm{ \nabla \hat{r}_i(\vx_1, \vy_1) - \nabla \hat{r}_i (\vx_2, \vy_2) }^2 \le \left( 16 n + 4 \hat{c}_1^2 + 64800 \hat{c}_2 \hat{c}_3^2 \right) \left( \norm{\vx_1 - \vx_2}^2 + \norm{\vy_1 - \vy_2}^2 \right).
\end{align*}
\end{proof}

Now we prove the Lemma \ref{lem:nonconvex-strongly:jump}.
\begin{proof}[Proof of Lemma \ref{lem:nonconvex-strongly:jump}]
Note that 
\begin{align*}
    \ve_{l+1} \hat{\vb}_l^\top \vx =
    \begin{cases}
        \omega x_1 \ve_1, & l = 0, \\
        (x_{l} - x_{l+1}) \ve_{l+1}, & 1 \le l < m.
    \end{cases}
\mbox{ and } 
    \hat{\vb}_l \ve_{l+1}^\top \vy =
    \begin{cases}
        \omega y_1 \ve_1, & l = 0, \\
        y_{l+1} (\ve_{l} - \ve_{l+1}), & 1 \le l < m.
    \end{cases}
\end{align*}
For $\vx, \vy \in \fF_k$ with $1 \le k < m$, we have
\begin{align}
\label{proof:nonconvex-strongly:ebx}
    \ve_{l+1} \hat{\vb}_l^\top \vx \in
    \begin{cases}
        \fF_{k+1}, & l = k, \\
        \fF_{k}, & l \neq k.
    \end{cases}
\mbox{ and }
    \hat{\vb}_l \ve_{l+1}^\top \vy \in
    \begin{cases}
        \fF_{k}, & l = k-1, \\
        \fF_{k-1}, & l \neq k-1.
    \end{cases}
\end{align}
Recall that
\begin{align*}
    \nabla_\vx \hat{r}_i (\vx, \vy)
    = &\, n  \sum_{l \in \fL_i} \hat{\vb}_l \ve_{l+1}^\top \vy + \hat{c}_2 \hat{c}_3 \nabla G (\hat{c}_3 \vx) , \\
    \nabla_\vy \hat{r}_i (\vx, \vy)
    = &\, n \sum_{l \in \fL_i} \ve_{l+1} \hat{\vb}_l^\top \vx - \hat{c}_1 \vy + n \ve_1 \bone_{ \{ i=1 \}}.
\end{align*}
By Inclusions (\ref{proof:nonconvex-strongly:ebx}),
we have the following results.
\begin{enumerate}
    \item Suppose that $\vx, \vy \in \fF_0$. 
    It holds that 
    $ \nabla_\vx \hat{r}_j(\vx, \vy) = \vzero $ for any $j$,
    $ \nabla_\vy \hat{r}_1(\vx, \vy) = n \ve_1 \in \fF_1 $
    and $ \nabla_\vy \hat{r}_j(\vx, \vy) = \vzero $ for $j \ge 2$.
    
    
    \item Suppose that $\vx, \vy \in \fF_k$, $1 \le k < m$ and $k \in \fL_i$. 
    It holds that 
    $ \nabla_\vx \hat{r}_j(\vx, \vy) \in \fF_k $ for any $j$,
    $ \nabla_\vy \hat{r}_i(\vx, \vy) \in \fF_{k+1} $ and
    $ \nabla_\vy \hat{r}_j(\vx, \vy) \in \fF_k $ for $j \neq i$.
    
\end{enumerate}

\vskip 5pt
Now we turn to consider $(\vu_i, \vv_i) = \prox_{\hat{r}_i}^\gamma (\vx, \vy)$. We have
\begin{align*}
    \nabla_\vx \hat{r}_i (\vu_i, \vv_i) + \frac{1}{\gamma} (\vu_i - \vx) & = \vzero, \\
    \nabla_\vy \hat{r}_i (\vu_i, \vv_i) - \frac{1}{\gamma} (\vv_i - \vy) & = \vzero,
\end{align*}
that is
\begin{align*}
    \begin{bmatrix}
    \frac{1}{\gamma } \mI_m & n \widehat{\mB}_i^\top \\
    - n \widehat{\mB}_i & \left( \hat{c}_1 + \frac{1}{\gamma} \right) \mI_m
    \end{bmatrix}
    \begin{bmatrix}
    \vu_i \\ \vv_i
    \end{bmatrix}
    =
    \begin{bmatrix}
    \hat{\vx} - \hat{\vu}_i \\ \hat{\vy}_i
    \end{bmatrix},
\end{align*}
where $\hat{\vx} = \vx / \gamma $, 
$ \hat{\vy}_i = \vy / \gamma + n \ve_1 \bone_{\{i=1\}} $ and
$\hat{\vu}_i = \hat{c}_2 \hat{c}_3 \nabla G(\hat{c}_3 \vu_i)$.
Recall that for $l_1, l_2 \in \fL_i$ and $l_1 \neq l_2$, $\hat{\vb}_{l_1}^\top \hat{\vb}_{l_2} = 0 $. It follows that
\begin{align*}
    \widehat{\mB}_i \widehat{\mB}_i^\top
    = \left( \sum_{l \in \fL_i} \ve_{l+1} \hat{\vb}_l^\top \right) \left( \sum_{l \in \fL_i}  \hat{\vb}_l \ve_{l+1}^\top \right)
    = \sum_{l \in \fL_i} \ve_{l+1} \hat{\vb}_l^\top \hat{\vb}_l \ve_{l+1}^\top,
\end{align*}
which is a diagonal matrix.
Denote
\begin{align*}
    \mD_i \triangleq
    \left( \hat{c}_1 + \frac{1}{\gamma} \right) \mI_m + 
    \gamma n^2 \widehat{\mB}_i \widehat{\mB}_i^\top
    = \diag \left( d_{i,1}, d_{i,2}, \dots, d_{i,m} \right).
\end{align*}
For $0 < l < m$, $l \in \fL_i$ implies $d_{i, l+1} = \hat{c}_1 + \frac{1}{\gamma} + 2 \gamma n^2$.
Then we have
\begin{align*}
    \begin{bmatrix}
    \vu_i \\ \vv_i
    \end{bmatrix}
    & = \begin{bmatrix}
    \frac{1}{\gamma } \mI_m & n \widehat{\mB}_i^\top \\
    - n \widehat{\mB}_i & \left( \hat{c}_1 + \frac{1}{\gamma} \right) \mI_m
    \end{bmatrix}^{-1}
    \begin{bmatrix}
    \hat{\vx} - \hat{\vu}_i \\ \hat{\vy}_i
    \end{bmatrix} \\
    & = \begin{bmatrix}
    \gamma \mI_m 
    -  \gamma^2 n^2  \widehat{\mB}_i^\top \mD_i^{-1} \widehat{\mB}_i &
    - \gamma n \widehat{\mB}_i^\top \mD_i^{-1} \\
    \gamma n \mD_i^{-1} \widehat{\mB}_i &
    \mD_i^{-1}
    \end{bmatrix}
    \begin{bmatrix}
    \hat{\vx} - \hat{\vu}_i \\ \hat{\vy}_i
    \end{bmatrix} \\
    & = \begin{bmatrix}
    \gamma ( \hat{\vx} - \hat{\vu}_i) 
    - \gamma^2 n^2 \sum_{l \in \fL_i} d_{i.l+1}^{-1} \hat{\vb}_l \hat{\vb}_l^\top ( \hat{\vx} - \hat{\vu}_i ) 
    - \gamma \sum_{l \in \fL_i} \hat{\vb}_l \ve_{l+1}^\top \mD_i^{-1} \hat{\vy}_i \\
    \gamma \sum_{l \in \fL_i} d_{i.l+1}^{-1} \ve_{l+1} \hat{\vb}_l^\top ( \hat{\vx} - \hat{\vu}_i ) + \mD_i^{-1} \hat{\vy}_i
    \end{bmatrix},
\end{align*}
that is
\begin{align}
    \vu_i + \gamma \hat{\vu}_i 
    - \gamma^2 n^2 \sum_{l \in \fL_i} d_{i.l+1}^{-1} \hat{\vb}_l \hat{\vb}_l^\top \hat{\vu}_i
    & = \gamma \hat{\vx} - \gamma^2 n^2 \sum_{l \in \fL_i} d_{i.l+1}^{-1} \hat{\vb}_l \hat{\vb}_l^\top  \hat{\vx}
    - \gamma \sum_{l \in \fL_i} \hat{\vb}_l \ve_{l+1}^\top \mD_i^{-1} \hat{\vy}_i. \label{proof:ncsc:u-equ} \\ 
    {\vv}_i & = \gamma \sum_{l \in \fL_i} d_{i.l+1}^{-1} \ve_{l+1} \hat{\vb}_l^\top ( \hat{\vx} - \hat{\vu}_i ) + \mD_i^{-1} \hat{\vy}_i. \label{proof:ncsc:v-equ}
\end{align}
We first focus on Equations (\ref{proof:ncsc:u-equ}).
Recall that $\hat{\vu}_i = \hat{c}_2 \hat{c}_3 \nabla G(\hat{c}_3 \vu_i)$ and
\begin{align*}
    \hat{\vb}_l \hat{\vb}_l^\top x
    = \begin{cases}
        \omega^2 x_1 \ve_1, & l = 0, \\
        (x_l - x_{l+1}) (\ve_l - \ve_{l+1}), & 0 < l < m.
    \end{cases}
\end{align*}
For simplicity, let $\vu_i = ( u_1, u_2, \dots, u_m )^\top$ and $\hat{\vu}_i = ( \hat{u}_1, \hat{u}_2, \dots, \hat{u}_m )^\top$,
and denote the right hand side of Equations (\ref{proof:ncsc:u-equ}) by $\vw$.
Recalling the definition of $G(\vx)$, we have $\hat{u}_l = 120 \hat{c}_2 \hat{c}_3 \frac{ \hat{c}_3^2 u_{l}^2 (\hat{c}_3 u_{l} - 1)}{ 1 + \hat{c}_3^2 u_{l}^2} $ for $l < m$ and $\hat{u}_m = 0$.
We can establish the following claims.

\begin{enumerate}
    \item If $0 < l < m-1$ and $l \in \fL_i$, we have
    \begin{equation}
    \label{proof:nonconvex-strongly:case1}
    \begin{aligned}
        u_l + \left( \gamma - \gamma^2 n^2 d_{i,l+1}^{-1} \right) \hat{u}_l + \gamma^2 n^2 d_{i,l+1}^{-1} \hat{u}_{l+1} & = w_l, \\
        u_{l+1} + \gamma^2 n^2 d_{i,l+1}^{-1} \hat{u}_{l} + \left( \gamma - \gamma^2 n^2 d_{i,l+1}^{-1} \right) \hat{u}_{l+1} & = w_{l+1}.
    \end{aligned}
    \end{equation}
    Setting $w_l = w_{l+1} = 0$ yields
    \begin{equation*}
    \begin{aligned}
        \left( 1 - 2 \gamma n^2 d_{i,l+1}^{-1} \right) u_l + \gamma n^2 d_{i,l+1}^{-1}  (u_l - u_{l+1}) +  \left( \gamma - 2 \gamma^2 n^2 d_{i,l+1}^{-1} \right) \hat{u}_{l} & = 0,\\
        \left( 1 - 2 \gamma n^2 d_{i,l+1}^{-1} \right) u_{l+1} + \gamma n^2 d_{i,l+1}^{-1}  (u_{l+1} - u_{l}) +  \left( \gamma - 2 \gamma^2 n^2 d_{i,l+1}^{-1} \right) \hat{u}_{l+1} & = 0. 
    \end{aligned}
    \end{equation*}
    Recalling that $d_{i,l+1} = \hat{c}_1 + 1 / \gamma + 2 \gamma n^2$, we find $ 1 - 2 \gamma n^2 d_{i,l+1}^{-1} > 0$.
    Since $\gamma < \frac{ \sqrt{2} + 1 }{ 60 \hat{c}_2 \hat{c}_3^2 }$, we can apply Lemma \ref{lem:solution:z1z2} with $z_1 = \hat{c}_3 u_l$ and $z_2 = \hat{c}_3 u_{l+1}$ and conclude that $u_l = u_{l+1} = 0$.
    
    \item If $m-1 \in \fL_i$, we have
    \begin{equation}
    \label{proof:nonconvex-strongly:case2}
    \begin{aligned}
        u_{m-1} + \left( \gamma - \gamma^2 n^2 d_{i,m}^{-1} \right) \hat{u}_{m-1}  & = w_{m-1}, \\
        u_{m} + \gamma^2 n^2 d_{i,m}^{-1} \hat{u}_{m-1} & = w_{m}.
    \end{aligned}
    \end{equation}
    Setting $w_{m-1} = w_m = 0$ yields
    \begin{equation*}
    \begin{aligned}
        u_{m-1} + \left( \gamma - \gamma^2 n^2 d_{i,m}^{-1} \right) \hat{u}_{m-1}  & = 0, \\
        \gamma n^2 d_{i,m}^{-1} u_{m-1} -  \left( 1 -  \gamma n^2 d_{i,m}^{-1} \right) u_m & = 0.
    \end{aligned}
    \end{equation*}
    Recalling that $d_{i,l+1} = \hat{c}_1 + 1 / \gamma + 2 \gamma n^2$ and $\gamma < \frac{ \sqrt{2} + 1 }{ 60 \hat{c}_2 \hat{c}_3^2 }$, we have $0 <  \gamma - \gamma^2 n^2 d_{i,m}^{-1}  <  \gamma  < \frac{ \sqrt{2} + 1 }{ 60 \hat{c}_2 \hat{c}_3^2 } $.
    Applying Lemma \ref{lem:solution:z} with $z = \hat{c}_3 u_{m-1}$,  we conclude that $u_{m-1}=0$. It follows that $u_m = 0$.
    
    \item If $0 < l < m$ and $l, l-1 \notin \fL_i$, we have
    \begin{align}\label{proof:nonconvex-strongly:case3}
        u_l + \gamma \hat{u}_l = w_l.
    \end{align}
    Setting $w_l = 0$ and applying Lemma \ref{lem:solution:z} with $z = \hat{c}_3 u_l$, we conclude that $u_{l} = 0$.
\end{enumerate}
Note that for $1 \le k \le m$, $\vx \in \fF_k$ implies $\hat{\vx} \in \fF_k$ and  $\vy \in \fF_k$ implies $\mD_i^{-1} \hat{\vy}_i \in \fF_k $.
And for $\vx \in \fF_k$ with $1 \le k < m$, we have
\begin{align}
\label{proof:nonconvex-strongly:bbx}
    \hat{\vb}_l \hat{\vb}_l^\top x 
    \in \begin{cases}
        \fF_{k+1}, & l = k, \\
        \fF_k, & l \neq k.
    \end{cases}
\end{align}
Then we can provide the following analysis.
\begin{enumerate}
    \item Suppose that $\vx, \vy \in \fF_0$. Note that $0 \in \fL_1$.
    
    For $j = 1$, we have $\hat{\vx} = \vzero$ and $\hat{\vy}_1 \in \fF_1$. 
    Since $0 \in \fL_1$, Inclusion 
    (\ref{proof:nonconvex-strongly:ebx})
    implies $\vw \in \fF_1$. Then we consider the solution to Equations (\ref{proof:ncsc:u-equ}). Since $n \ge 2$, we have $1 \notin \fL_1$. If $2 \in \fL_1$, we can consider the solution to Equations (\ref{proof:nonconvex-strongly:case1}) or (\ref{proof:nonconvex-strongly:case2}) and conclude that $u_2 = 0$.
    If $2 \notin \fL_1$, we can consider the solution to Equation (\ref{proof:nonconvex-strongly:case3}) and conclude that $u_2 = 0$. Similarly, we obtain $u_l = 0$ for $l \ge 2$, which implies $\vu_1 \in \fF_1$. Since $1 \notin \fL_1$, by Inclusion (\ref{proof:nonconvex-strongly:ebx}) and Equations (\ref{proof:ncsc:v-equ}), we have $\vv_1 \in \fF_1$. 
    
    For $j \neq 1$, we have $\hat{\vx} = \hat{\vy}_j = \vzero$. It follows that $\vw = \vzero$. Note that $0 \not \in \fL_j$. If $1 \in \fL_j$, we can consider the solution to Equations (\ref{proof:nonconvex-strongly:case1}) or (\ref{proof:nonconvex-strongly:case2}) and conclude that $u_1 = 0$.
    If $1 \not \in \fL_j$, we can consider the solution to Equation (\ref{proof:nonconvex-strongly:case3}) and conclude that $u_1 = 0$. Similarly, we obtain $u_l = 0$ for all $l$, which implies $\vu_j = \vzero$. By Equations (\ref{proof:ncsc:v-equ}), we have $\vv_j = \vzero$.
    
    \item Suppose that $\vx, \vy \in \fF_k$, $1 \le k < m$ and $k \in \fL_i$.
    
    For $j = i$, we have $\hat{\vx}, \hat{\vy}_i \in \fF_k$. If $k = m-1$, clearly $\vu_i, \vv_i \in \fF_m$. Now we assume $k < m-1$.  Inclusions (\ref{proof:nonconvex-strongly:bbx}) and 
    (\ref{proof:nonconvex-strongly:ebx}) 
    imply $\vw \in \fF_{k+1}$. Then we consider the solution to Equations (\ref{proof:ncsc:u-equ}). Since $n \ge 2$, we have $k+1 \notin \fL_i$. 
    If $k+2 \in \fL_i$, we can consider the solution to Equations (\ref{proof:nonconvex-strongly:case1}) or (\ref{proof:nonconvex-strongly:case2}) and conclude that $ u_{k+2} = 0$. 
    If $k+2 \notin \fL_i$, we can consider the solution to Equation (\ref{proof:nonconvex-strongly:case3}) and conclude that $u_{k+2} = 0$. Similarly, we obtain $u_l = 0$ for $l \ge k+2$, which implies $\vu_i \in \fF_{k+1}$. Since $k+1 \notin \fL_i$, by Inclusion (\ref{proof:nonconvex-strongly:ebx}) and Equations (\ref{proof:ncsc:v-equ}), we have $\vv_i \in \fF_{k+1}$.
    
    For $j \neq i$, we also have $\hat{\vx}, \hat{\vy}_i \in \fF_k$.
    Since $k \notin \fL_j$, by Inclusions 
    (\ref{proof:nonconvex-strongly:ebx})
    and (\ref{proof:nonconvex-strongly:bbx}), we have $\vw \in \fF_k$.
    If $k+1 \in \fL_j$, we can consider the solution to Equations (\ref{proof:nonconvex-strongly:case1}) or (\ref{proof:nonconvex-strongly:case2}) and conclude that $u_{k+1} = 0$.
    If $k+1 \notin \fL_j$, we can consider the solution to Equation (\ref{proof:nonconvex-strongly:case3}) and conclude that $u_{k+1} = 0$. Similarly, we obtain $u_l = 0$ for $l \ge k+1$, which implies $\vu_j \in \fF_k$. Since $k \notin \fL_j$, by Inclusion (\ref{proof:nonconvex-strongly:ebx}) and Equations (\ref{proof:ncsc:v-equ}), we have $\vv_j \in \fF_k$.
\end{enumerate}
This completes the proof.
\end{proof}
\section{Proofs for Section~\ref{sec:minimax}}\label{appendix:minimax}

\subsection{Proofs for the Strongly-Convex-Strongly-Concave Case}\label{appendix:minimax:scsc}
With $f_\mathrm{SCSC}$ and $\{ f_{\mathrm{SCSC}, i} \}_{i=1}^n$ defined in Definition~\ref{defn:scsc},
we have the following proposition.
\begin{proposition}\label{prop:strongly-strongly}
For any $n \ge 2$, $m \ge 2$, $f_{\mathrm{SCSC}, i}$ and $f_{\mathrm{SCSC}}$ in Definition \ref{defn:scsc} satisfy:
\begin{enumerate}
    \item $\{ f_{\mathrm{SCSC}, i} \}_{i=1}^n$ is $L$-average smooth and each $ f_{\mathrm{SCSC}, i}$ is $(\mu_x, \mu_y)$-convex-concave. Thus, $f_{\mathrm{SCSC}}$ is $(\mu_x, \mu_y)$-convex-concave.
    \item The saddle point of Problem (\ref{prob:scsc}) is
        \begin{align*}
            \begin{cases}
            \vx^* = 
            \frac{2 \beta \mu_y}{1 - q} \sqrt{ \frac{2 n}{L^2 - 2 \mu_y^2} }
            (q, q^{2}, \dots, q^m)^{\top}, \\[0.15cm]
            \vy^* = \beta \left(q, q^2, \dots, q^{m-1}, \sqrt{\frac{\alpha+1}{2}} q^m\right)^{\top},
            \end{cases}
        \end{align*}
        where $q=\frac{\alpha-1}{\alpha+1}$.
    Moreover, $\norm{\vx^*} \le R_x$, $\norm{\vy^*} \le R_y$.
    \item For $1 \le k \le m - 1$, we have
    \begin{align*}
    \min_{\vx \in \fX \cap \fF_k} \phi_{\mathrm{SCSC}} (\vx) -
    \max_{\vy \in \fY \cap \fF_k} \psi_{\mathrm{SCSC}} (\vy) \ge  
    \frac{ \beta^2 \left( L^2 - 2 \mu_y^2 \right) }{ 8n (\alpha + 1) \mu_x } q^{2k}.
    \end{align*}
\end{enumerate}
\end{proposition}

\begin{proof}
\begin{enumerate}
    \item Just recall Proposition  \ref{prop:convex-concave:base} and Lemma \ref{lem:scale}.
    \item It is easy to check
    $
    f_{\mathrm{SCSC}}(\vx, \vy) 
    = 
    \xi \inner{\vy}{ \widetilde{\mB} \left( m, \zeta
    \right) \vx  }  
    + \frac{\mu_x}{2} \norm{\vx}^2 - \frac{\mu_y}{2} \norm{\vy}^2 
    - 
    \beta \xi
    \inner{\ve_1}{\vx}, $
    where $\zeta = \sqrt{ \frac{2}{\alpha + 1} }$ and $\xi = \lambda / \beta^2 = \frac{1}{2} \sqrt{ \frac{ L^2 - 2 \mu_y^2 }{ 2n } }$. Letting the gradient of $f_{ \mathrm{SCSC} } (\vx, \vy)$ be zero, we obtain
    \begin{align}
        &\vy = \frac{\xi}{\mu_y} \widetilde{\mB}(m, \zeta) \vx
        , \quad
        \left( \mu_x \mI + \frac{ \xi^2 }{ \mu_y } \widetilde{\mB}(m, \zeta)^{\top} \widetilde{\mB}(m, \zeta) \right) \vx = \beta \xi \ve_1. \label{eq:Mx=b}
    \end{align}
    Note that 
    \[
    \frac{\mu_x \mu_y }{ \xi^2 }
    = \frac{ 8 n \mu_x \mu_y }{ L^2 - 2 \mu_y^2 }
    = \frac{ 8 n  \mu_x }{ (\kappa_y^2 - 2) \mu_y }
    = \frac{ 8 n }{ \left( \kappa_y - 2 / \kappa_y \right) \kappa_x }
    = \frac{ 4 }{ \alpha^2 - 1 }.
    \]
    One can check $q$ is a root of the equation
    $
        z^2 - \left( 2 + \frac{\mu_x \mu_y}{ \xi^2 } \right)z + 1 = 0. \label{eq:ssp-scc-q}
    $
    By some calculation, 
    the solution of (\ref{eq:Mx=b}) equation is 
    \begin{align*}
        \vx^* = \frac{ \beta \mu_y }{ (1 - q) \xi } (q, q^{2}, \dots, q^m)^{\top} , \quad
        \vy^* = \beta \left( q, q^2, \dots, q^{m-1},  \frac{q^m}{\zeta} \right)^{\top}.
    \end{align*}
    Moreover, from the definition of $\beta$, we have
    \[
    \norm{ \vx^* }^2
    = \frac{\beta^2 \mu_y^2 \left( q^2 - q^{2m+2} \right) }{ (1-q)^2 \left(1-q^2 \right) \xi^2 }
    \le \frac{\beta^2 \mu_y^2 q^2}{ (1-q)^2 \left(1-q^2 \right) \xi^2 }
    \le \frac{\beta^2
    \kappa_x^2 (1 - 2 / \kappa_y^2)
    }{ 8 n \alpha }
    \le R_x^2,
    \]
    and
    \[
    \norm{ \vy^* }^2
    = \beta^2 \left( \frac{ q^2 - q^{2m} }{ 1 - q^2 } + \frac{q^{2m}}{\zeta^2} \right)
    = \beta^2 \frac{ q^2 + q^{2m+1} }{ 1 - q^2 }
    \le \beta^2 \frac{2 q^2}{1 - q^2}
    = \beta^2 \frac{ ( \alpha - 1 )^2 }{ 4 \alpha }
    \le R_y^2.
    \]
    \item Define $\tilde{\phi}_{\mathrm{SCSC}} (\vx) = \max_{\vy \in \BR^m } f_{\mathrm{SCSC}} (\vx, \vy)$ and $\tilde{\psi}_{\mathrm{SCSC}} (\vy) = \min_{\vx \in \BR^m } f_{\mathrm{SCSC}} (\vx, \vy)$. 
    We first show that 
    \[
    \min_{\vx \in \fF_k} \tilde{\phi}_{\mathrm{SCSC}} (\vx) - \max_{\vy \in \fF_k} \tilde{\psi}_{\mathrm{SCSC}} (\vy) \ge \frac{\beta^2 \xi^2}{(\alpha+1) \mu_x } q^{2k},
    \]
    where $\xi = \frac{1}{2} \sqrt{ \frac{  L^2 - 2 \mu_y^2 }{ 2n } }$ .
    Recall that 
    $
    f_{\mathrm{SCSC}}(\vx, \vy) 
    = \xi \inner{\vy}{ \widetilde{\mB} \left( m, \zeta  \right) \vx  }  
    + \frac{\mu_x}{2} \norm{\vx}^2 - \frac{\mu_y}{2} \norm{\vy}^2
    -  \beta \xi \inner{\ve_1}{\vx},
    $
    where $\zeta = \sqrt{ \frac{2}{\alpha + 1} }$. 
    Then we can rewrite $f_{\mathrm{SCSC}} (\vx, \vy) $ as
    \begin{align}
    \label{eq:scsc:maxy}
    f_{\mathrm{SCSC}} (\vx, \vy) = - \frac{\mu_y}{2} \norm{ \vy - \frac{\xi}{\mu_y} \widetilde{\mB} (m, \zeta) \vx }^2 + \frac{\xi^2}{2 \mu_y } \norm{\widetilde{\mB} (m,\zeta) \vx}^2 + \frac{\mu_x}{2} \norm{\vx}^2 -  \beta \xi \inner{\ve_1}{\vx}.
    \end{align}
    Thus $\tilde{\phi}_{\mathrm{SCSC}} (\vx) = \frac{\xi^2}{2 \mu_y } \norm{\widetilde{\mB}(m, \zeta) \vx}^2 + \frac{\mu_x}{2} \norm{\vx}^2 -  \beta \xi \inner{\ve_1}{\vx} $.
    For $\vx \in \fF_k$, let $\tilde{\vx}$ be the first $k$ coordinates of $\vx$.
    Then we can rewrite $\tilde{\phi}_{\mathrm{SCSC}}$ as
    $
    \tilde{\phi}_k (\tilde{\vx}) \triangleq \tilde{\phi}_{\mathrm{SCSC}} (\vx) 
    =  \frac{\xi^2}{2 \mu_y } \norm{\widetilde{\mB}(k,1) \tilde{\vx} }^2 + \frac{\mu_x}{2} \norm{\tilde{\vx}}^2 -  \beta \xi \inner{\ve_1}{\tilde{\vx}},
    $
    where $\hat{\ve}_1$ is the first $k$ coordinates of $\ve_1$.
    Letting $\nabla \tilde{\phi}_k (\tilde{\vx}) = \vzero_k$, we obtain
    \begin{align}\label{eq:scsc:phi}
    \frac{\xi^2}{\mu_y} \widetilde{\mB}(k,1)^\top \widetilde{\mB}(k,1) \tilde{\vx} + \mu_x \tilde{\vx} = \beta \xi \hat{\ve}_1.
    \end{align}
    Recall that 
    $ \frac{\mu_x \mu_y }{ \xi^2 } = \frac{ 4 }{ \alpha^2 - 1 }$
    and $q = \frac{\alpha - 1}{\alpha + 1}$. One can check $q$ and $1/q$ are two roots of the equation
    $z^2 - \left( 2 + \frac{\mu_x \mu_y}{\xi^2} \right)z + 1 = 0 $.
    By some calculation,
    the solution to Equations (\ref{eq:scsc:phi}) is 
    \begin{align*}
        \tilde{\vx}^* = 
        \frac{\beta \mu_y (\alpha + 1) q^{k+1}}{2 \xi \left( 1 + q^{2k+1} \right) }
        \left( q^{-k} - q^{k}, q^{-k+1} - q^{k-1}, \dots, q^{-1} - q \right)^{\top},
    \end{align*}
    and the value of $\min_{\vx \in \fF_k} \tilde{\phi}_{\mathrm{SCSC}}(\vx)$ is
    $
        \min_{\vx \in \fF_k } \tilde{\phi}_{\mathrm{SCSC}}(\vx) = - \frac{\beta^2 \mu_y (\alpha + 1)}{4} ~ \frac{q - q^{2k+1}}{1 + q^{2k+1}}.$
    
    On the other hand, observe that
    \begin{align}
    \label{eq:scsc:minx}
    f_{\mathrm{SCSC}} (\vx, \vy) = \frac{\mu_x}{2} 
    \norm{ \vx + \frac{\xi}{\mu_x} \widetilde{\mB} (m, \zeta)^\top \vy - \frac{\beta \xi}{\mu_x} \ve_1 }^2 
    - \frac{\xi^2}{2 \mu_x} \norm{ \widetilde{\mB} (m, \zeta)^\top \vy - \beta \ve_1 }^2 - \frac{\mu_y}{2} \norm{\vy}^2.
    \end{align}
    It follows that
    $
    \tilde{\psi}_{\mathrm{SCSC}} (\vy) =  - \frac{\xi^2}{2 \mu_x} \norm{ \widetilde{\mB} (m, \zeta)^\top \vy - \beta \ve_1 }^2 - \frac{\mu_y}{2} \norm{\vy}^2.$
    For $\vy \in \fF_k$, let $\tilde{\vy}$ be the first $k$ coordinated of $\vy$.
    Then we can rewrite $\tilde{\psi}_{\mathrm{SCSC}}$ as
    $
    \tilde{\psi}_k (\tilde{\vy}) \triangleq \tilde{\psi}_{\mathrm{SCSC}} (\vy)
    = - \frac{\xi^2}{2 \mu_x} \norm{ \widetilde{\mB}(k, 1)^\top  \tilde{\vy} - \beta \hat{\ve}_1 }^2 - - \frac{\xi^2}{2 \mu_x} \inner{\hat{\ve}_k}{\tilde{\vy}}^2 - \frac{\mu_y}{2} \norm{\tilde{\vy}}^2,$
    where $\hat{\ve}_1, \hat{\ve}_k $ are the first $k$ ordinates of $\ve_1$ and $\ve_k$ respectively.
    Letting $\nabla \tilde{\psi}_k (\tilde{\vy}) = \vzero_k$, we obtain
    \begin{align*}
         \frac{\xi^2}{\mu_x} \left(  \widetilde{\mB}(k, 1) \widetilde{\mB}(k,1)^\top + \hat{\ve}_k \hat{\ve}_k^{\top} \right) \tilde{\vy} + \mu_y \tilde{\vy} = \frac{\beta \xi^2}{\mu_x} \widetilde{\mB}(k,1) \hat{\ve}_1.
    \end{align*}
    Then, we can check that the solution to the above equations is 
    \begin{align*}
        \tilde{\vy}^* = \frac{\beta q^{k+1}}{1 - q^{2k+2}} (q^{-k} - q^{k}, q^{-k+1} - q^{k-1}, \dots, q^{-1} - q)^{\top},
    \end{align*}
    and the optimal value of $ \tilde{\psi}_{\mathrm{SCSC}}(\vy)$ is
    $
        \min_{\vy \in \fF_k } \tilde{\psi}_{\mathrm{SCSC}}(\vy)
        = -\frac{\beta^2 \xi^2 }{\mu_x(\alpha + 1)} ~ \frac{1 + q^{2k+1}}{1 - q^{2k+2}}.$
    It follows that
    \begin{align*}
    & \quad\, \min_{\vx \in \fF_k } \tilde{\phi}_{\mathrm{SCSC}}(\vx) - \max_{\vy \in \fF_k} \tilde{\psi}_{\mathrm{SCSC}}(\vy) \\
    & = - \frac{\beta^2 \mu_y (\alpha + 1)}{4} ~ \frac{q - q^{2k+1}}{1 + q^{2k+1}} + \frac{\beta^2 \xi^2 }{\mu_x(\alpha + 1)} ~ \frac{1 + q^{2k+1}}{1 - q^{2k+2}} \\
    & = - \frac{\beta^2 \xi^2 }{\mu_x(\alpha + 1)} ~ \frac{\mu_x \mu_y (\alpha + 1)^2 q }{4 \xi^2} ~ \frac{1 - q^{2k}}{1 + q^{2k+1}} + \frac{\beta^2 \xi^2 }{\mu_x(\alpha + 1)} ~ \frac{1 + q^{2k+1}}{1 - q^{2k+2}} \\
    &= \frac{\beta^2 \xi^2 }{\mu_x(\alpha + 1)} \left( \frac{1 + q^{2k+1}}{1 - q^{2k+2}} - \frac{1 - q^{2k}}{1 + q^{2k+1}} \right) \\
    &= \frac{\beta^2 \xi^2 }{\mu_x(\alpha + 1)} ~ \frac{2 q^{2k+1} + q^{2k} + q^{2k+2}}{(1 - q^{2k+2})(1 + q^{2k+1})} \\
    &\ge  \frac{\beta^2 \xi^2 }{\mu_x(\alpha + 1)} q^{2k}.
    \end{align*}
    Clearly, we have $\min_{\vx \in \fX \cap \fF_k} \phi_{\mathrm{SCSC}} (\vx) -
    \max_{\vy \in \fY \cap \fF_k} \psi_{\mathrm{SCSC}} (\vy) \ge \min_{\vx \in \fF_k} \phi_{\mathrm{SCSC}} (\vx) -
    \max_{\vy \in \fF_k} \psi_{\mathrm{SCSC}} (\vy)$.
    It remains to show that $\min_{\vx \in \fF_k } \phi_{\mathrm{SCSC}} (\vx) = \min_{\vx \in \fF_k } \tilde{\phi}_{\mathrm{SCSC}} (\vx) $ and $\max_{\vy \in \fF_k } \psi_{\mathrm{SCSC}} (\vy) = \max_{\vy \in \fF_k } \tilde{\psi}_{\mathrm{SCSC}} (\vy) $.
    Recall the expressions (\ref{eq:scsc:maxy}) and (\ref{eq:scsc:minx}). It suffices to prove
    $\norm{\hat{\vx} } \le R_x$ and 
    $\norm{\hat{\vy} } \le R_y$
    where
    \begin{align*}
        \hat{\vx} & =
        - \frac{\xi}{\mu_x} \widetilde{\mB} (m, \zeta) \top
        \begin{bmatrix} \tilde{\vy}^* \\ \vzero_{m-k} \end{bmatrix}
        + \frac{\beta \xi}{\mu_x} \ve_1,  \quad
        \hat{\vy}  = 
        \frac{\xi}{\mu_y} \widetilde{\mB} (m, \zeta) \begin{bmatrix} \tilde{\vx}^* \\ \vzero_{m-k} \end{bmatrix}.
    \end{align*}
    By some calculation, we have
    \begin{align*}
    \norm{ \hat{\vx} }^2 & = \frac{\beta^2 \xi^2 (1-q)^2 }{\mu_x^2 \left( 1 - q^{2k+2} \right)^2 } \left( \frac{1 - q^{4k+2}}{1-q^2} + 2(k+1) q^{2k+1 }\right),\\
    \norm{ \hat{\vy} }^2 & = \frac{\beta^2}{ \left( 1 + q^{2k+1} \right)^2 } \left( \frac{q^2 - q^{4k+2}}{1-q^2} + 2k q^{2k+1} \right) .
    \end{align*}
    Note that 
    $\max_{x > 0} x q^x = \log \frac{1}{q} \, e^{- \left( \log \frac{1}{q} \right)^2}$
    and
    $\log r - r^2 \le - r$ for any $r > 0$.
    It follows that
    $\max_{x > 0} x q^x \le e^{- \log \frac{1}{q}} = q  $. 
    Then we have
    \begin{align*}
        \norm{ \hat{\vx}  }^2
        & \le  \frac{\beta^2 \xi^2 (1-q)^2 }{\mu_x^2 \left( 1 - q \right)^2 } \left( \frac{1 }{1-q^2} + 1 \right) 
        \le \frac{2 \beta^2 \xi^2}{\mu_x^2 \left( 1 - q^2 \right) } 
        = \frac{\beta^2 \left( L^2 - 2 \mu_x^2 \right) (\alpha + 1)^2 }{ 16 n \mu_y^2  \alpha  } 
        \le R_x^2, \\
        \norm{ \hat{\vy}  }^2
        & \le \beta^2 \left( \frac{q^2}{1 - q^2} + q^2 \right) 
        \le \frac{2 \beta^2 q^2}{1 - q^2}
        \le \frac{\beta^2  (\alpha - 1)^2 }{2 \alpha}
        \le R_y^2.
    \end{align*}
\end{enumerate}
    This completes the proof.
\end{proof}

\begin{proof}[Proof of Theorem \ref{thm:scsc:example}]
Let $q = \frac{\alpha - 1}{\alpha + 1}$. 
For $\kappa_x \ge \kappa_y \ge \sqrt{2n + 2}$, we have
$\alpha = \sqrt{ \frac{ \left( \kappa_y - 2 / \kappa_y \right) \kappa_x  }{ 2n } + 1} \ge \sqrt{2}$, 
$q = \frac{\alpha - 1}{\alpha + 1} \ge \frac{\sqrt{2} - 1}{\sqrt{2} + 1}$ and $\kappa_y - 2 / \kappa_y \ge \kappa_y / 2$.

Let $M = \floor{\frac{ \log \left( 9 (\alpha + 1) \mu_x \eps / \beta^2 \xi^2 \right) }{2\log q}}$ where $\xi = \frac{1}{2} \sqrt{ \frac{ L^2 - 2 \mu_y^2 }{2n} }$. Then we have
\[
\min_{\vx \in \fX \cap \fF_M} \phi_{\mathrm{SCSC}} (\vx)
- \max_{\vy \in \fY \cap \fF_M} \psi_{\mathrm{SCSC}} (\vy)
\ge \frac{\beta^2 \xi^2}{(\alpha+1) \mu_x } q^{2M}
\ge 9 \eps.
\]
where the first inequality follows from the third property of Proposition \ref{prop:strongly-strongly}.

First, we need to ensure $1 \le M < m$.
Note that $M \ge 1$ is equivalent to $\eps \le \frac{q^2 \beta^2 \xi^2}{9 (\alpha + 1) \mu_x}$.
Recall that 
{ \small
\[
\beta = \min \left\{ 
2  R_x \sqrt{ \frac{ 2 \alpha n }{ \kappa_x^2 (1 - 2 / \kappa_y^2)  }}
,\, 
\frac{ 4 R_x }{ \alpha + 1 } \sqrt{ \frac{ \alpha n }{ \kappa_x^2 (1 - 2 / \kappa_y^2) } }
,\,
\frac{ \sqrt{ 2 \alpha} R_y  }{ \alpha - 1 }
\right\}  .
\] }
When $\beta = 2 R_x \sqrt{ \frac{ 2 \alpha n }{ \kappa_x^2 (1 - 2 / \kappa_y^2)  }}$, noticing that $\frac{\alpha (\alpha - 1)^2 }{ (\alpha + 1)^3 }$ is increasing for  $\alpha > 1$,
we have
{ \small
\[
\frac{q^2 \beta^2 \xi^2}{9 (\alpha + 1) \mu_x}
=\frac{ \alpha (\alpha - 1)^2 }{9 (\alpha + 1)^3} \mu_x R_x^2
\ge \frac{\sqrt{2} \left( \sqrt{2} - 1 \right)^5 }{9} \mu_x R_x^2,
\] }
When $\beta = \frac{ 4 R_x }{ \alpha + 1 } \sqrt{ \frac{ \alpha n }{ \kappa_x^2 (1 - 2 / \kappa_y^2) } }$, noticing that $\alpha^2 - 1 = \frac{ \left( \kappa_y - 2 / \kappa_y \right) \kappa_x }{2 n} \le \frac{\kappa_x \kappa_y}{2 n}$ and $\frac{\alpha (\alpha - 1)^3}{ (\alpha + 1)^4 }$ is increasing for $\alpha > 1$,  we have
{ \small
\[
\frac{q^2 \beta^2 \xi^2}{9 (\alpha + 1) \mu_x}
= \frac{2 \alpha (\alpha - 1)^3 }{ 9 (\alpha + 1)^5 (\alpha-1)} \mu_x R_x^2
\ge \frac{4 \sqrt{2} \left( \sqrt{2} - 1 \right)^7}{9} \frac{n \mu_x R_x^2}{\kappa_x \kappa_y}. 
\] }
When $\beta = \frac{ \sqrt{ 2 \alpha} R_y  }{ \alpha - 1 }$, noticing that $\frac{\mu_x \mu_y}{\xi^2} = \frac{4}{\alpha^2 - 1}$ and $\frac{\alpha (\alpha - 1)}{ (\alpha + 1)^2 }$ is increasing for $\alpha > 1$, we have
{ \small
\[
\frac{q^2 \beta^2 \xi^2}{9 (\alpha + 1) \mu_x}
= \frac{ \alpha (\alpha - 1) }{18 (\alpha+1)^2 } \mu_y R_y^2
\ge \frac{\sqrt{2} \left( \sqrt{2} - 1 \right)^3 }{18}  \mu_y R_y^2.
\] }
Thus, $\eps \le \frac{1}{800} \min\left\{ \frac{n \mu_x R_x^2}{\kappa_x \kappa_y}, \mu_y R_y^2 \right\}$ is a sufficient condition for $M \ge 1$.
Similarly, we can obtain
{ \small
\begin{align}\label{eq:proof:scsc:tmp1}
    \frac{\beta^2 \xi^2}{9 (\alpha + 1) \mu_x }
    \ge \frac{1}{25} \min\left\{ \frac{n \mu_x R_x^2}{\kappa_x \kappa_y}, \mu_y R_y^2 \right\}.
\end{align} }

On the other hand, since $\frac{\alpha}{\alpha - 1} \le \frac{\sqrt{2}}{\sqrt{2} - 1}$, we have
{ \small
\begin{align}\label{eq:proof:scsc:tmp2}
\frac{\beta^2 \xi^2}{9 (\alpha + 1) \mu_x }
& \le \min \left\{ 
\frac{ \alpha \mu_x R_x^2 }{9 (\alpha + 1)} ,
\frac{2 \alpha (\alpha - 1) \mu_x R_x^2 }{ 9 (\alpha + 1)^3 (\alpha-1)} ,
\frac{ \alpha \mu_y R_y^2 }{18 (\alpha-1) } 
\right\} 
\le \frac{2}{9} \min \{ \mu_x R_x^2, \mu_y R_y^2 \}.
\end{align} }
Note that the function $h(\beta) = \frac{1}{\log \left(\frac{\beta + 1}{\beta - 1}\right)} - \frac{\beta}{2}$ is increasing when $\beta>1$ and $\lim_{\beta \rightarrow +\infty} h(\beta) = 0$. With $q = \frac{\alpha - 1}{\alpha + 1}$,there holds
$
     h(\sqrt{2}) \le -\frac{1}{\log q} - \frac{\alpha}{2} \le 0,
$
which implies $\frac{\alpha}{2} \ge - \frac{1}{\log q} \ge \frac{\alpha}{2} + h (\sqrt{2})$.
Then by (\ref{eq:proof:scsc:tmp2}) we have
\[
m 
= \left\lfloor \frac{\alpha}{4}\log \left(\frac{ 2 \min  \left\{ \mu_x R_x^2, \mu_y R_y^2 \right\} }{9\eps}\right) \right\rfloor + 1
\ge \left\lfloor - \frac{ \log \left(\frac{\beta^2 \xi^2}{9 (\alpha + 1) \mu_x \eps}\right) }{2 \log q}  \right\rfloor + 1
> M.
\]
Thus, we have verified $1 \le M < m$.
Moreover, 
$- \frac{1}{\log q} \ge \frac{\alpha}{2} + h(\sqrt2) $ implies
{
\begin{align*}
    -\frac{1}{\log(q)} &
    \ge 
    \frac{1}{2}\sqrt{\frac{ \left( \kappa_y - 2 / \kappa_y \right) \kappa_x }{ 2n } + 1} + h \left( \sqrt{2} \right) 
    \ge \frac{\sqrt{2}}{4} \left( \sqrt{ \frac{ \kappa_x \kappa_y }{4 n} } + 1 \right) + h \left( \sqrt{2} \right),
\end{align*}}
where the last inequality is due to $\kappa_y - 2 / \kappa_y \ge \kappa_y / 2$ and 
$\sqrt{2(a + b)} \ge \sqrt{a} + \sqrt{b}$ for $a,b > 0$. 

By Lemma \ref{lem:minimiax:base}, for $M \ge 1$ and $N = \frac{(M + 1)n}{4 (1+c_0) }$, we have
$
\min_{t \le N}  \E \phi_{\mathrm{SCSC}} (\vx_t) - \min_{t \le N} \E \psi_{\mathrm{SCSC}} (\vy_t) \ge \eps.
$
Therefore, in order to find $(\hat{\vx}, \hat{\vy}) \in \fX \times \fY$ such that $\E \phi_{\mathrm{SCSC}} (\hat{\vx}) - \E \psi_{\mathrm{SCSC}} (\hat{\vy}) \ge \eps$,  $\fA$ needs at least $N$ PIFO queries, where
{\small\begin{align*}
    N &= \frac{ (M+1) n}{4 (1+c_0) } \\
    &\ge \frac{n}{4(1+c_0)} \left( -\frac{1}{\log(q)} \right) \log \left( \frac{ \beta^2 \xi^2 }{ 9 (\alpha + 1) \mu_x \eps } \right) \\
    &\ge \frac{n}{4(1+c_0)} \left( \sqrt{  \frac{ \kappa_x \kappa_y  }{32n} } + \frac{\sqrt{2}}{4} + h \left( \sqrt{2} \right)  \right) \log \left( \frac{ \min \left\{ n \mu_x R_x^2 / (\kappa_x \kappa_y), \mu_y R_y^2 \right\} }{ 25 \eps }  \right) \\
    &= \Omega\left( \left( n + \sqrt{n \kappa_x \kappa_y} \right) \log\left( \frac{1}{\eps} \right) \right),
\end{align*}}
where the second inequality is by (\ref{eq:proof:scsc:tmp2}).
This completes the proof.
\end{proof}

\begin{proof}[Proof of Theorem \ref{thm:scsc:example:2}]
Let $\alpha = \sqrt{ \frac{ 2 (\kappa_x - 1) }{n} + 1 }$.
Consider the functions $\{ f_{\text{SC}, i} \}_{i=1}^n$ and $f_{\text{SC}}$ defined in Definition \ref{defn:sc} with $\mu$ and $R$ replaced by $\mu_x$ and $R_x$.
We construct  $ \{ G_{\mathrm{SCSC}, i} \}_{i=1}^n, G_{\mathrm{SCSC}}: \BR^m \times \BR^m \rightarrow \BR $ as follows
\begin{align*}
    G_{\mathrm{SCSC}, i} (\vx, \vy) & =  f_{\text{SC},i}(\vx) - \frac{\mu_y}{2} \norm{\vy}^2, \\
    G_{\mathrm{SCSC}} (\vx, \vy) & =  \frac{1}{n} \sum_{i=1}^n G_{\mathrm{SCSC}, i} (\vx, \vy) =  f_{\text{SC}}(\vx) - \frac{\mu_y}{2} \norm{\vy}^2.
\end{align*}
By Proposition \ref{prop:base} and Lemma \ref{lem:scale}, 
we can check that each component function $G_{\mathrm{SCSC}, i}$ is $L$-smooth and $(\mu_x, \mu_y)$-convex-concave. 
Then $G_{\mathrm{SCSC}}$ is $(\mu_x, \mu_y)$-convex-concave. 
Moreover, we have
\[
\max_{\vy \in \fY} G_{\mathrm{SCSC}}(\vx, \vy) = f_{\text{SC}} (\vx)
\quad \text{and} \quad
\min_{\vx \in \fX} G_{\mathrm{SCSC}}(\vx, \vy) = \min_{\vx \in \fX} f_{\text{SC}}(\vy) - \frac{\mu_y}{2} \norm{\vy}^2.
\]
It follows that for any $(\hat{\vx}, \hat{\vy}) \in \fX \times \fY$, we have
\[
\max_{\vy \in \fY} G_{\mathrm{SCSC}}(\hat{\vx}, \vy) - \min_{\vx \in \fX} G_{\mathrm{SCSC}}(\vx, \hat{\vy})
\ge f_{\text{SC}}(\hat{\vx}) -  \min_{\vx \in \fX} f_{\text{SC}} (\vx).
\]
Note that $\kappa_x \ge \sqrt{2n + 2 } = \Omega( \sqrt{n} )$.
By Theorem \ref{thm:average:strongly:example},
for $\tilde{L} = \sqrt{ \frac{n (L^2 - \mu_x^2)}{2} - \mu_x^2 }$,
\[
    \eps \le \frac{\mu_x R_x^2}{18} \left(\frac{\alpha-1}{\alpha+1}\right)^{2} \text{ and } 
    m = \left\lfloor \frac{1}{4}\left(\sqrt{\frac{ 2 ( \tilde{L} / \mu_x - 1) }{n} + 1}\right) \log \left(\frac{\mu_x R_x^2}{9\eps}\right) \right\rfloor + 1,
\]
in order to find $(\hat{\vx}, \hat{\vy}) \in \fX \times \fY$ such that
$\E\left( \max_{\vy \in \fY} G_{\mathrm{SCSC}}(\hat \vx, \vy) - \min_{\vx \in \fX} G_{\mathrm{SCSC}}(\vx, \hat \vy) \right) < \eps$, 
PIFO algorithm $\fA$ needs at least $N = \Omega\left(\left(n + n^{3/4} \sqrt{ \kappa_x}\right)\log\left( \frac{1}{\eps} \right)\right)$ queries.

Moreover, $\kappa_x \ge n/2 + 1$ implies $\alpha \ge \sqrt{2}$. Then we have
$ \left(\frac{\alpha-1}{\alpha+1}\right)^{2} \ge \left( \frac{\sqrt{2} - 1}{\sqrt{2} + 1} \right)^2 \ge \frac{1}{40}$. This completes the proof.
\end{proof}

\begin{proof}[Proof of Lemma \ref{lem:scsc:example:3}]
Consider the functions $\{ H_{\mathrm{SCSC}, i}: \BR \times \BR \rightarrow \BR \}_{i=1}^n$ where
\begin{align*}
    H_{\mathrm{SCSC}, i}(x, y) & = 
    \begin{cases}
    \frac{L}{2} \left( x^2 - y^2 \right) - n L R_x x,  &\text{ for } i = 1, \\
    \frac{L}{2} \left( x^2 - y^2 \right), &\text{ otherwise, }
    \end{cases}
\end{align*}
and $H_{\mathrm{SCSC}} (x,y) = \frac{1}{n} \sum_{i=1}^n H_{\mathrm{SCSC}, i} (x, y) = \frac{L}{2} (x^2 - y^2) - L R_x x$.
It is easy to check that 
$\{ H_{\mathrm{SCSC}, i} \}_{i=1}^n$
$L$-average smooth and $(\mu_x, \mu_y)$-convex-concave for any $0 \le \mu_x, \mu_y \le L$. 
Moreover, we have
\[
\max_{|y| \le R_y} H_{\mathrm{SCSC}}(x,y) = \frac{L}{2} x^2 - L R_x x
\quad \text{and} \quad
\min_{|x| \le R_x} H_{\mathrm{SCSC}}(x,y) = - \frac{L R_x^2}{2} - \frac{L}{2} y^2.
\]
Note that for $i \ge 2$, it holds that
\begin{align*}
    \nabla_x H_{\mathrm{SCSC}, i}(x, y) = L x ~ \text{ and } ~
    \prox_{H_{\mathrm{SCSC}, i}}^{\gamma}(x, y) = \left(\frac{x}{L\gamma + 1}, \frac{y}{L\gamma + 1}\right).
\end{align*}
This implies $x_t = x_0 = 0$ will hold till the PIFO algorithm $\fA$ draws $H_{\mathrm{SCSC}, 1}$. 
Denote $T = \min\{t: i_t = 1\}$. 
Then, the random variable $T$ follows a geometric distribution with success probability $p_1$, and satisfies
$
    \pr{T \ge n/2} = (1 - p_1)^{\floor{(n-1)/2}} \ge (1 - 1/n)^{(n-1)/2} \ge 1/2,
$
where the last inequality is according to that $h(\beta) = (\frac{\beta}{\beta + 1})^{\beta/2}$ is a decreasing function and $\lim_{\beta \to \infty} h(\beta) = 1/\sqrt{e} \ge 1/2$.
Consequently, for $N = n / 2$ and $t < N$, we know that
\begin{align*}
    &\quad \E\left( \max_{|y| \le R_y} H_{\mathrm{SCSC}}(x_t, y) - \min_{|x| \le R_x} H_{\mathrm{SCSC}}(x, y_t) \right) \\
    & \ge \E\left( \max_{|y| \le R_y} H_{\mathrm{SCSC}}(x_t, y) - \min_{|x| \le R_x} H_{\mathrm{SCSC}}(x, y_t) \bigg\vert t < T \right) \pr{T > t} \\
    &= \E\left( \max_{|y| \le R_y} H_{\mathrm{SCSC}}(0, y) - \min_{|x| \le R_x} H_{\mathrm{SCSC}}(0, y_t) \bigg\vert t < T \right) \pr{T > t}  \\
    & \ge \frac{L R_x^2}{2} \pr{T \ge N} \ge L R_x^2/4 \ge \eps.
\end{align*}
Thus, 
to find $(\hat{x}, \hat{y}) \in \fX \times \fY$ such that 
$
\E
\max_{|y| \le R_y} H_{\mathrm{SCSC}}(\hat x, y) - \E \min_{|x| \le R_x} H_{\mathrm{SCSC}}(x, \hat y) 
< \eps, $
PIFO algorithm $\fA$ needs at least $N = \Omega(n)$ queries.
\end{proof}

\subsection{Proofs for the Convex-Strongly-Concave Case}\label{appendix:minimax:csc}
With $f_\mathrm{CSC}$ and $\{ f_{\mathrm{CSC}, i} \}_{i=1}^n$ defined in Definition~\ref{defn:csc},
we have the following proposition.
\begin{proposition}\label{prop:convex-strongly}
For any $n \ge 2$, $m \ge 2$, $f_{\mathrm{CSC}, i}$ and $f_{\mathrm{CSC}}$ in Definition \ref{defn:csc} satisfy:
\begin{enumerate}
    \item $\{ f_{\mathrm{CSC}, i} \}_{i=1}^n$ is $L$-smooth and each $f_{\mathrm{CSC}, i}$ is $(0, \mu_y)$-convex-concave. Thus, $f_{\mathrm{CSC}}$ is $(0, \mu_y)$-convex-concave.
    \item For $1 \le k \le m - 1$, we have
    \begin{align*}
        \min_{\vx \in \fX \cap \fF_k } \phi_{\mathrm{CSC}}(\vx) - \max_{\vy \in \fY \cap \fF_k} \psi_{\mathrm{CSC}}(\vy) 
        \ge - \frac{k \mu_y \beta^2 }{2} + 
        \frac{R_x \beta}{2} \sqrt{ \frac{L^2 - 2 \mu_y^2}{ 2n(k+1) } },
    \end{align*}
    where $\beta  = \min \left\{ \frac{ R_x \sqrt{ ( L^2 / \mu_y^2 - 2 ) / (2n) }  }{ 2(m + 1)^{3/2}}, \frac{R_y}{\sqrt{m}} \right\}$.
\end{enumerate}
\end{proposition}

\begin{proof}
\begin{enumerate}
    \item Just recall Proposition \ref{prop:convex-concave:base} and Lemma \ref{lem:scale}.
    \item It it easy to check
    $
    f_{\mathrm{CSC}}(\vx, \vy) 
    = \xi
    \inner{\vy}{ \widetilde{\mB} \left( m, 1  \right) \vx  }  
    - \frac{\mu_y}{2} \norm{\vy}^2 
    - \beta \xi 
    \inner{\ve_1}{\vx},$
    where $\xi = \lambda / \beta^2 = \frac{1}{2} \sqrt{ \frac{L^2 - 2 \mu_y^2 }{2n} }$.
    Define $\tilde{\phi}_{\mathrm{CSC}} (\vx) = \max_{\vy \in \BR^m} f_{\mathrm{CSC}} (\vx, \vy)$.
    We first show that
    \begin{align*}
        \min_{\vx \in \fX \cap \fF_k } \tilde{\phi}_{\mathrm{CSC}}(\vx) - \max_{\vy \in \fY \cap \fF_k} \psi_{\mathrm{CSC}}(\vy) \ge 
        - \frac{k \mu_y \beta^2 }{2} +  \frac{R_x \xi \beta}{\sqrt{k+1}} .
    \end{align*}
    On one hand, we have
    \begin{equation}
    \label{eq:csc:maxy}
    \begin{aligned}
        \tilde{\phi}_{\mathrm{CSC}}(\vx) 
        & = \max_{\vy \in \BR^m} \left( \xi \inner{\vy}{ \widetilde{\mB} \left( m, 1  \right) \vx } - \frac{\mu_y}{2} \norm{\vy}^2 - \beta \xi \inner{\ve_1}{\vx} \right) \\
        & = \max_{\vy \in \BR^m} \left( - \frac{\mu_y}{2} \norm{ \vy - \frac{\xi}{\mu_y} \widetilde{\mB}(m,1) \vx }^2 + \frac{\xi^2}{2 \mu_y} \norm{\widetilde{\mB}(m,1) \vx}^2 - \beta \xi \inner{\ve_1}{\vx} \right) \\
        & = \frac{\xi^2}{2 \mu_y} \norm{\widetilde{\mB}(m,1) \vx}^2 - \beta \xi \inner{\ve_1}{\vx}.
    \end{aligned}
    \end{equation}
    For $\vx \in \fF_k$, let $\tilde{\vx}$ be the first $k$ coordinates of $\vx$.
    We can rewrite $\tilde{\phi}_{\mathrm{CSC}}(\vx)$ as
    $
        \tilde{\phi}_k(\tilde\vx) \triangleq \tilde{\phi}_{\mathrm{CSC}}(\vx) = \frac{\xi^2}{2 \mu_y} \norm{\widetilde{\mB}(k, 1) \tilde\vx}^2 - \beta \xi \inner{\hat\ve_1}{\tilde\vx}, $
    where $\hat\ve_1$ is the first $k$ coordinates of $\ve_1$. Letting $\nabla \tilde{\phi}_k (\tilde\vx) = \vzero_k$, we get
    $
        \widetilde{\mB}(k, 1)^{\top} \widetilde{\mB}(k, 1) \tilde\vx = \frac{\beta \mu_y}{\xi} \hat\ve_1.$
    The solution is $\tilde{\vx}^* = \frac{\beta \mu_y}{\xi} (k, k-1, \dots, 1)^{\top}$.
    Noting that 
    $
    \norm{\tilde{\vx}^* }^2 = \frac{\beta^2 \mu_y^2}{\xi^2} \frac{k (k+1) (2k+1)}{6} \le \frac{8 n \beta^2 }{L^2 / \mu_y^2 - 2} (m+1)^3 \le R_x^2,$
    we obtain
    $
        \min_{\vx \in \fX \cap \fF_k } \tilde{\phi}_{\mathrm{CSC}}(\vx) = - \frac{k \mu_y \beta^2 }{2}. $
    
    On the other hand,
    \begin{equation}\label{cauchy:example}
    \begin{aligned}
        \min_{\vx \in \fX} \inner{\vx}{ \widetilde{\mB}(m,1)^{\top} \vy - \beta \ve_1} 
        & \ge \min_{\norm{\vx} \le R_x} -\norm{\vx} \norm{ \widetilde{\mB}(m,1)^{\top} \vy - \beta \ve_1} \\
        & \ge - R_x \norm{ \widetilde{\mB}(m,1)^{\top} \vy - \beta \ve_1},
    \end{aligned}
    \end{equation}
    where the equality will hold when either $\vx = -\frac{R_x}{ \widetilde{\mB}(m,1)^{\top} \vy - \beta \ve_1} \left( \widetilde{\mB}(m,1)^{\top} \vy - \beta \ve_1 \right)~$ or $~ \widetilde{\mB}(m,1)^{\top} \vy - \beta \ve_1 = \vzero$.
    It follows that
    \begin{equation}\label{min:example}
    \begin{aligned}
        \psi_{\mathrm{CSC}}(\vy) 
        & = \min_{\vx \in \fX} \left( \xi \inner{\vx}{ \widetilde{\mB} \left( m, 1  \right)^\top \vy - \beta \ve_1 } - \frac{\mu_y}{2} \norm{\vy}^2 \right) \\
        & = - R_x \xi \norm{\widetilde{\mB} \left( m, 1  \right)^\top \vy - \beta \ve_1 } - \frac{\mu_y}{2} \norm{\vy}^2. 
    \end{aligned}
    \end{equation}
    We can upper bound $\max_{\vy \in \fY \cap \fF_k} \psi_{\mathrm{CSC}}(\vy)$ as
    \begin{align*}
        \max_{\vy \in \fY \cap \fF_k} \psi_{\mathrm{CSC}}(\vy) 
        & = \max_{\vy \in \fY \cap \fF_k}  \left( - R_x \xi \norm{\widetilde{\mB}(m,1)^{\top} \vy - \beta \ve_1} - \frac{\mu_y}{2} \norm{\vy}^2 \right) \\
        & \le \max_{\vy \in \fF_k} \left( - R_x \xi \norm{\widetilde{\mB}(m,1)^{\top} \vy - \beta \ve_1} \right)\\
        &= - R_x \xi \sqrt{J_{k, \beta}(y_1, y_2, \dots, y_k)}
        \le - \frac{R_x \xi \beta}{\sqrt{k+1}} ,
    \end{align*}
    where $J_{k,\beta}$ is defined in (\ref{eq:J_func}) and the last inequality follows from Lemma \ref{lem:appendix:gap}.
    
    It remains to prove $\min_{\vx \in \fX \cap \fF_k} \phi_{\mathrm{CSC}} (\vx) = \min_{\vx \in \fX \cap \fF_k} \tilde{\phi}_{\mathrm{CSC}} (\vx)$. Recall the expression (\ref{eq:csc:maxy}). It suffices to show that $\norm{\hat{\vy} } \le R_y$ where
    $
        \hat{\vy} = 
        \frac{\xi}{\mu_y} \widetilde{\mB} (m, 1) \begin{bmatrix} \tilde{\vx}^* \\ \vzero_{m-k} \end{bmatrix}$.
    Since $\beta \le \frac{R_y}{\sqrt{m} }$, one can check $\norm{\hat{\vy} } \le R_y$ does hold.
\end{enumerate}
This completes the proof.
\end{proof}

\begin{proof}[Proof of Theorem \ref{thm:csc:example}]
Since $L / \mu_y \ge 2$, we have $L^2 - 2 \mu_y^2 \ge L^2 / 2$. Then $\eps \le \frac{L^2 R_x^2}{5184\, n \mu_y} \le  \frac{(L^2 - 2\mu_y^2) R_x^2}{2592 n \mu_y}$, which implies that
$m \ge 4$ and
$ \frac{R_x}{6 } \sqrt{\frac{L^2 - 2 \mu_y^2}{ 2 n \mu_y \eps}} - 2 \ge  \frac{R_x}{12 } \sqrt{\frac{L^2 - 2 \mu_y^2}{ 2 n \mu_y \eps}} + 1 $.
It follows that $m \ge \frac{R_x}{12} \sqrt{\frac{L^2 - 2 \mu_y^2}{ 2 n \mu_y \eps}} $.
Then with $\eps \le \frac{\mu_y R_y^2}{36}$, we have
\begin{align*}
    \frac{ R_x \sqrt{ \frac{L^2 / \mu_y^2 - 2}{2n} } }{ 2(m + 1)^{3/2}}
    < \frac{ R_x \sqrt{ \frac{L^2 / \mu_y^2 - 2}{2n}  } }{ 2 m^{3/2}}
    \le 6 \sqrt{\frac{\eps}{\mu_y m }} \le \frac{R_y}{\sqrt{m}},
\end{align*}
which imlpies that 
$\beta = \min\left\{  \frac{ R_x \sqrt{ (L^2 / \mu_y^2 - 2) / (2n) } }{ 2(m + 1)^{3/2}}, \frac{R_y}{\sqrt{m}} \right\} =  \frac{ R_x \sqrt{ (L^2 / \mu_y^2 - 2 ) / (2n)} }{ 2(m + 1)^{3/2}}$.
Following Proposition \ref{prop:convex-strongly}, for $1 \le k \le m-1$, we have
{ \small
\begin{align*}
    \min_{\vx \in \fX \cap \fF_k} \phi_{\mathrm{CSC}}(\vx) {-} \max_{\vy \in \fY \cap \fF_k} \psi_{\mathrm{CSC}}(\vy) 
    {\ge} {-} \frac{k \mu_y \beta^2 }{2} {+} 
    \frac{R_x \beta}{2} \sqrt{ \frac{L^2 {-} 2 \mu_y^2}{ 2n(k+1) } }
    {=} \frac{(L^2 {-} 2\mu_y^2) R_x^2}{16 n \mu_y} \frac{2(m+1)^{3/2} {-} k\sqrt{k+1}}{(m+1)^3 \sqrt{k+1}}.
\end{align*} }
Define $M \triangleq \floor{\frac{m}{2}}$. Then we have $M = \floor{\frac{R_x}{12 } \sqrt{\frac{L^2 - 2 \mu^2}{ 2 n \mu \eps}}} - 1 \ge 2 $ and $M < m$.

Since $2 (M + 1) = 2 \floor{\frac{m}{2}} + 2 \ge m+1$
and 
 $h(\beta) = \frac{2\beta^{3/2} - \beta_0^{3/2}}{\beta^3}$ is a decreasing function when $\beta > \beta_0$,
 for $k=M$
 we have
\begin{align*}
    \min_{\vx \in \fX \cap \fF_M} \phi_{\mathrm{CSC}}(\vx) - \max_{\vy \in \fY \cap \fF_M} \psi_{\mathrm{CSC}}(\vy) 
    \ge \frac{(L^2 - 2\mu_y^2) R_x^2}{16 n \mu_y} \frac{4\sqrt{2} - 1}{8(M+1)^{2}} 
    > \frac{(L^2 - 2\mu_y^2) R_x^2}{32 n \mu_y (M+1)^2} 
    \ge 9\eps,
\end{align*}
where the last inequality is due to $M + 1 \le \frac{R_x}{12 } \sqrt{\frac{L^2 - 2 \mu_y^2}{ 2 n \mu_y \eps}}$.

By Lemma \ref{lem:minimiax:base},
for $N = \frac{n (M+1)}{4 (1+c_0) }$, we know that
$
    \min_{t\le N} \E\left( \phi_{\mathrm{CSC}}(\vx_t) - \psi_{\mathrm{CSC}}(\vy_t) \right) \ge \eps.$
Therefore, in order to find suboptimal solution $(\hat{\vx}, \hat{\vy}) \in \fX \times \fY$ such that $\E \left(\phi_{\mathrm{CSC} } (\hat{\vx}) - \psi_{\mathrm{CSC}} (\hat{\vy})\right) < \eps$, algorithm $\fA$ needs at least $N$ PIFO queries, where 
\begin{align*}
    N =  \frac{n}{4(1+c_0)} \left(\floor{\frac{R_x}{12 } \sqrt{\frac{L^2 - 2 \mu_y^2}{ 2 n \mu_y \eps}}} \right) 
    = \Omega\left( n + 
    R_x L \sqrt{ \frac{n}{\mu_y \eps} } \right).
\end{align*}
This completes the proof.
\end{proof}

\begin{proof}[Proof of Theorem \ref{thm:csc:example:2}]
Consider the functions $\{ f_{\text{C}, i} \}_{i=1}^n$ and $f_{\text{C}}$ defined in Definition \ref{defn:c:average} with $R$ replaced by $R_x$.
We construct  $ \{ G_{\mathrm{CSC}, i} \}_{i=1}^n, G_{\mathrm{CSC}}: \BR^m \times \BR^m \rightarrow \BR $ as follows
\begin{align*}
    G_{\mathrm{CSC}, i} (\vx, \vy) & = f_{\text{C}, i} (\vx) -  \frac{\mu_y}{2} \norm{\vy}^2, \\
    G_{\mathrm{CSC}} (\vx, \vy) & =  \frac{1}{n} \sum_{i=1}^n G_{\mathrm{CSC}, i} (\vx, \vy) = f_{\text{C}} (\vx) - \frac{\mu_y}{2} \norm{\vy}^2.
\end{align*}
By Proposition \ref{prop:base} and Lemma \ref{lem:scale},
we can check that each component function $G_{\mathrm{CSC}, i}$ is $L$-smooth and $(0, \mu_y)$-convex-concave. 
Then $G_{\mathrm{CSC}}$ is $(0, \mu_y)$-convex-concave. 
Moreover, we have
\[
\max_{\vy \in \fY} G_{\mathrm{CSC}}(\vx, \vy) = f_{\text{C}}(\vx) 
\quad \text{and} \quad
\min_{\vx \in \fX} G_{\mathrm{CSC}}(\vx, \vy) = \min_{\vx \in \fX} f_{\text{C}}(\vx) - \frac{\mu_y}{2} \norm{\vy}^2.
\]
It follows that for any $(\hat{\vx}, \hat{\vy}) \in \fX \times \fY$, we have
\[
\max_{\vy \in \fY} G_{\mathrm{CSC}}(\hat{\vx}, \vy) - \min_{\vx \in \fX} G_{\mathrm{CSC}}(\vx, \hat{\vy})
\ge f_{\text{C}}(\hat{\vx}) -  \min_{\vx \in \fX} f_{\text{C}} (\vx).
\]
By Theorem \ref{thm:average:convex:example},
for 
\[
    \eps \le \frac{ \sqrt{2} R_x^2 L}{768 \sqrt{n}} \text{ and } 
     m = \floor{
     \frac{ \sqrt[4]{18} }{12} R_x n^{-1/4} \sqrt{ \frac{L}{\eps} }
     } - 1 ,
\]
in order to find $(\hat{\vx}, \hat{\vy}) \in \fX \times \fY$ such that $\E\left( \max_{\vy \in \fY} G_{\mathrm{CSC}}(\hat \vx, \vy) - \min_{\vx \in \fX} G_{\mathrm{CSC}}(\vx, \hat \vy) \right) < \eps$, PIFO algorithm $\fA$ needs at least $N = \Omega\left(n + R_x n^{3/4} \sqrt{ \frac{L}{\eps} }\right)$ queries.
\end{proof}

\subsection{Proofs for the Convex-Concave Case}\label{appendix:minimax:cc}
With $f_\mathrm{CC}$ and $\{ f_{\mathrm{CC}, i} \}_{i=1}^n$ defined in Definition~\ref{defn:cc},
we have the following proposition.
\begin{proposition}\label{prop:convex-concave}
For any $n \ge 2$, $m \ge 3$, $f_{\mathrm{CC}, i}$ and $f_{\mathrm{CC}}$ in Definition \ref{defn:cc} satisfy:
\begin{enumerate}
    \item $\{ f_{\mathrm{CC}, i} \}_{i=1}^n$ is $L$-average smooth and each $f_{\mathrm{CC}, i} $ convex-concave. Thus, $f_{\mathrm{CC}}$ is convex-concave. 
    \item For $1 \le k \le m-1$, we have
    \begin{align*}
        \min_{\vx \in \fX \cap \fF_k } \phi_{\mathrm{CC}}(\vx) - \max_{\vy \in \fY \cap \fF_k} \psi_{\mathrm{CC}} (\vy) \ge
        \frac{L R_x R_y}{ \sqrt{ 8 nm (k+1)} }.
    \end{align*}
\end{enumerate}
\end{proposition}

\begin{proof}
\begin{enumerate}
    \item Just recall Proposition \ref{prop:convex-concave:base} and Lemma \ref{lem:scale}.
    \item It is easy to check
    $
    f_{\mathrm{CC}}(\vx, \vy) 
    = \frac{L }{ \sqrt{8n} } \inner{\vy}{ \widetilde{\mB} \left( m, 1  \right) \vx  }  
    - \frac{ L R_y }{  \sqrt{8nm}} \inner{\ve_1}{\vx}.$
    By similar analysis from Equation (\ref{cauchy:example}) to Equation (\ref{min:example}) of the proof of Proposition \ref{prop:convex-strongly}, we can conclude that
    \begin{align*}
        \phi_{\mathrm{CC}}(\vx) & = \frac{L R_y}{ \sqrt{8n} } \norm{\widetilde{\mB}(m, 1)\vx} - \frac{L R_y}{ \sqrt{8nm}} \inner{\ve_1}{\vx}, \\
        \psi_{\mathrm{CC}}(\vy) & = - \frac{L R_x}{ \sqrt{8n} } \norm{\widetilde{\mB}(m, 1)^{\top}\vy - \frac{R_y}{\sqrt{m}}\ve_1}.
    \end{align*}
    
    Note that 
    $
        \phi_{\mathrm{CC}}(\vx) = \max_{\vy \in \fY} f_{\mathrm{CC}}(\vx, \vy) \ge \max_{\vy \in \fY} \min_{\vx \in \fX} f_{\mathrm{CC}}(\vx, \vy) = \max_{\vy \in \fY} \psi(\vy) \ge \psi(\vy^*) = 0,$
    where $\vy^* = \frac{R_y}{\sqrt{m}} \vone_m \in \fY$. Therefore, we have
    $
        \min_{\vx \in \fX \cap \fF_k} \phi_{\mathrm{CC}}(\vx) = \phi_{\mathrm{CC}}(\vzero) = 0.$
    On the other hand, following Lemma \ref{lem:appendix:gap}, we can obtain 
    \begin{align*}
        \max_{\vy \in \fY \cap \fF_k } \psi_{\mathrm{CC}}(\vy) = \max_{\vy \in \fY \cap \fF_k } - \frac{L R_x}{ \sqrt{8n} } \norm{\widetilde{\mB}(m, 1)^{\top}\vy - \frac{R_y}{\sqrt{m}}\ve_1} 
        = - \frac{L R_x}{ \sqrt{8n} } \frac{R_y}{\sqrt{m(k+1)}},
    \end{align*}
    where the optimal point is $\tilde\vy^* = \frac{R_y}{(k + 1)\sqrt{m}}(k, k-1, \dots, 1, 0, \dots, 0)^{\top}$, which satisfies
    $
        \norm{\tilde\vy^*} = \frac{R_y}{(k + 1)\sqrt{m}} \sqrt{\frac{k(k+1)(2k+1)}{6}} \le R_y.$
    Finally, note that $k + 1 \ge m/2$. Thus we obtain 
    \begin{align*}
        \min_{\vx \in \fX \cap \fF_k } \phi_{\mathrm{CC}}(\vx) - \max_{\vy \in \fY \cap \fF_k} \psi_{\mathrm{CC}} (\vy) 
        \ge \frac{L R_x R_y}{ \sqrt{ 8nm (k+1)} }.
    \end{align*}
\end{enumerate}
This completes the proof.
\end{proof}

\begin{proof}[Proof of Theorem \ref{thm:cc:example}]
    The assumption on $\eps$ implies $m \ge 3$.
    Let $M \triangleq \floor{(m - 1)/2}
    = \floor{\frac{L R_x R_y}{36 \eps \sqrt{n} }} - 1$. 
    Then we have $M \ge 1$ and $m / 2 \le M + 1 \le (m+1) / 2$.
    By Proposition \ref{prop:convex-concave}, we have
    {\small\begin{align*}
        \min_{\vx \in \fX \cap \fF_M} \phi_{\mathrm{CC}}(\vx) - \max_{\vy \in \fY \cap \fF_M} \psi_{\mathrm{CC}}(\vy) 
        \ge & \frac{L R_x R_y}{ \sqrt{ 8nm (M+1)} }
        \ge \frac{L R_x R_y}{4  (M+1) \sqrt{n} }
        \ge \frac{L R_x R_y}{ 2(m + 1) \sqrt{n} } \ge 9\eps.
    \end{align*}}
    Hence, 
    by Lemma \ref{lem:minimiax:base},
    for $N = \frac{ n (M+1 )}{4 (1 + c_0) }$, we know that
    $
        \min_{t\le N} \E\left( \phi_{\mathrm{CC}}(\vx_t) - \psi_{\mathrm{CC}}(\vy_t) \right) \ge \eps.$
   Thus, to find an approximate solution $(\hat{\vx}, \hat{\vy}) \in \fX \times \fY$ such that 
    $ \E \left(\phi_{\mathrm{CC}}(\hat{\vx}) - \psi_{\mathrm{CC}}(\hat{\vy})\right) < \eps,$
    the PIFO algorithm $\fA$ needs at least $N$ queries, where 
    $
        N =  \frac{n}{4 (1 + c_0) } \left(\floor{\frac{L R_x R_y}{36 \eps \sqrt{n} }} \right) = \Omega\left( n + \frac{\sqrt{n} L R_x R_y}{\eps} \right).$
\end{proof}

\begin{proof}[Proof of Lemma \ref{lem:cc:example:2}]
Consider the functions $\{ H_{\mathrm{CC}, i}: \BR \times \BR \rightarrow \BR \}_{i=1}^n$ where
\begin{align*}
    H_{\mathrm{CC}, i} (x, y) = 
    \begin{cases}
    L x y - n L R_x y,  &\text{ for } i = 1, \\
    L x y, &\text{ otherwise,}
    \end{cases}
\end{align*}
and $H_{\mathrm{CC}}(x,y) = \frac{1}{n} \sum_{i=1}^n H_{\mathrm{CC}, i}(x, y) =  L x y - L R_x y$.
Consider the minimax problem
$$\min_{|x|\le R_x} \max_{|y| \le R_y} H_{\mathrm{CC}} (x, y).$$
It is easy to check that 
$\{ H_{\mathrm{CC}, i} \}_{i=1}^n$ is $L$-smooth and each $H_{\mathrm{CC}, i}$ is convex-concave. 
Moreover, we have
\begin{align*}
    \max_{|y|\le R_y} H_{\mathrm{CC}} (x, y) = L R_y |x - R_x|, ~~~\text{and}~~~ \min_{|x| \le R_x} H_{\mathrm{CC}} (x, y) = -L R_x (|y| + y) \le 0,
\end{align*}
and it holds that
$
    \min_{|x| \le R_x} \max_{|y|\le R_y} H_{\mathrm{CC}} (x, y) = \max_{|y|\le R_y} \min_{|x| \le R_x} H_{\mathrm{CC}} (x, y) = 0.$

Note that for $i \ge 2$, we have
\begin{align*}
    \nabla_x H_{\mathrm{CC}, i}(x, y) = L y, \, \nabla_y H_{\mathrm{CC}, i}(x, y) = L x, \text{~and~}
    \prox_{H_{\mathrm{CC}, i}}^{\gamma}(x, y) = \left(\frac{L \gamma x + y}{L^2 \gamma^2 + 1}, \frac{x - L \gamma y}{L^2 \gamma^2 + 1}\right),
\end{align*}
which implies $x_t = y_t = x_0 = y_0 = 0$ will hold till the PIFO algorithm $\fA$ draws $H_{\mathrm{CC}, 1}$. 

Let $T = \min\{t: i_t = 1\}$. Then, the random variable $T$ follows a geometric distribution with success probability $p_1$, and satisfies
$
    \pr{T \ge n/2} = (1 - p_1)^{\floor{(n-1)/2}} \ge (1 - 1/n)^{(n-1)/2} \ge 1/2,$
where the last inequality is according to that $h(\beta) = (\frac{\beta}{\beta + 1})^{\beta/2}$ is a decreasing function and $\lim_{\beta \to \infty} h(\beta) = 1/\sqrt{e} \ge 1/2$.
For $N = n / 2$ and $t < N$, we know that
\begin{align*}
    &\quad \E\left( \max_{|y| \le R_y} H_{\mathrm{CC}} (x_t, y) - \min_{|x|\le R_x} H_{\mathrm{CC}} (x, y_t) \right) \\
    & \ge \E\left( \max_{|y| \le R_y} H_{\mathrm{CC}} (x_t, y) - \min_{|x|\le R_x} H_{\mathrm{CC}} (x, y_t) ~\big\vert~ t < T \right) \pr{T > t} \\
    &= \E\left( \max_{|y| \le R_y} H_{\mathrm{CC}} (0, y) - \min_{|x|\le R_x} H_{\mathrm{CC}} (x, 0) ~\big\vert~ t < T \right) \pr{T > t} \\
    & = \frac{L R_x R_y}{2} \pr{T \ge N} \ge L R_x R_y/4 \ge \eps.
\end{align*}
Thus, 
to find $(\hat{x}, \hat{y}) \in \fX \times \fY$ such that 
$
\E
\max_{|y| \le R_y} H_{\mathrm{CC}} (\hat x, y) - \E \min_{|x| \le R_x} H_{\mathrm{CC}} (x, \hat y) 
< \eps, $
algorithm $\fA$ needs at least $N = \Omega(n)$ PIFO queries.
\end{proof}

\subsection{Proofs for the Nonconvex-Strongly-Concave Case}\label{appendix:minimax:ncsc}
With $f_\mathrm{NCSC}$, $\phi_{\mathrm{NCSC}}$ and $\{ f_{\mathrm{NCSC}, i} \}_{i=1}^n$ defined in Definition~\ref{defn:ncsc},
we have the following proposition.
\begin{proposition}\label{prop:nonconvex-strongly}
For any $n \ge 2$, $L / \mu_y \ge 4$ and $\eps^2 \le \frac{\Delta L^2 \alpha}{6967296 n \mu_y}$, the following properties hold:
\begin{enumerate}
    \item $\{ f_{\mathrm{NCSC}, i} \}_{i=1}^n$ is $L$-average smooth and each $f_{\mathrm{NCSC}, i}$ is $ (- \mu_x, \mu_y)$-convex-concave. 
    \item $\phi_{\mathrm{NCSC}}(\vzero_{}) - \min_{\vx \in \BR^{m+1}} \phi_{\mathrm{NCSC}}(\vx^*) \le \Delta$.
    \item $m \ge 2$ and for $M = m-1$, $\min_{\vx \in \fF_{M}} \norm{\nabla \phi_{\mathrm{NCSC}}(\vx)} \ge  9 \eps.$
\end{enumerate}
\end{proposition}

\begin{proof}[Proof of Proposition \ref{prop:nonconvex-strongly}]
\begin{enumerate}
    \item By Proposition \ref{prop:nonconvex-strongly:base} and Lemma \ref{lem:scale}, ${f}_{\mathrm{NCSC}, i}$ is $(-\mu_1, \mu_2)$-convex-concave and $\{ {f}_{\mathrm{NCSC}, i} \}_{i=1}^n$ $l$-average smooth where
    \begin{align*}
        \mu_1 & = \frac{45 (\sqrt{3} - 1) L^2 \alpha }{256 n \mu_y} \le \mu_x, 
        \qquad \mu_2  = \mu_y, \\
        l & = \frac{L}{8 \sqrt{n}} \sqrt{4 n + \frac{256 n \mu_y^2}{L^2} + 16200 \frac{\alpha^2 L^2}{256 n \mu_y^2 } } 
        \le \frac{L}{8 \sqrt{n}} \left( 2 \sqrt{n} + 16 \frac{ \sqrt{n} \mu_y}{L} + \frac{ 45 \sqrt{2} \alpha L }{ 8 \sqrt{n} \mu_y } \right) \le L.
    \end{align*}
    Thus each component ${f}_{\text{NCSC, i}}$ is $(-\mu_x, \mu_y)$-convex-concave and $\{ {f}_{\mathrm{NCSC}, i} \}_{i=1}^n$ is $L$-smooth.
    
    \item We first give a closed form expression of ${\phi}_{\mathrm{NCSC}}$. For simplicity, we omit the parameters of $\widehat{\mB}$. It is easy to check 
    { \small
    \begin{align*}
    {f}_{\mathrm{NCSC}} (\vx, \vy) 
    = \frac{L}{16 \sqrt{n}} \inner{\vy}{\widehat{\mB}
    \vx } - \frac{\mu_y}{2} \norm{\vy}^2 + \frac{\sqrt{\alpha} \lambda L }{16 \sqrt{n} \mu_y} \sum\limits_{i=1}^{m} \Gamma \left( \frac{1}{4} \sqrt{\frac{ \sqrt{\alpha} L}{ \lambda \sqrt{n}}} x_i \right) 
    - \frac{1}{4} \sqrt{\frac{\lambda L }{ \sqrt{n} }}\inner{\ve_1}{\vy}.
    \end{align*}}
    Then we can rewrite ${f}_{\mathrm{NCSC}}(\vx, \vy)$ as
    \begin{align*}
        {f}_{\mathrm{NCSC}} (\vx, \vy) = & -\frac{\mu_y}{2} \norm{ \vy -  \frac{1}{\mu_y} \left(  \frac{L}{16 \sqrt{n} } \widehat{\mB} \vx - \frac{1}{4} \sqrt{ \frac{\lambda L}{ \sqrt{n} }} \ve_1  \right)  }^2\\
        & + \frac{1}{2 \mu_y} \norm{\frac{L}{16 \sqrt{n}} \widehat{\mB} \vx - \frac{1}{4} \sqrt{ \frac{\lambda L}{ \sqrt{n} }} \ve_1  }^2 + \frac{\sqrt{\alpha} \lambda L }{16 \sqrt{n} \mu_y} \sum\limits_{i=1}^{m} \Gamma \left( \frac{1}{4} \sqrt{\frac{ \sqrt{\alpha} L}{ \lambda \sqrt{n}}} x_i \right).
    \end{align*}
    It follows that
    \begin{align*}
        {\phi}_{\mathrm{NCSC}} (\vx) & = \frac{1}{2 \mu_y} \norm{\frac{L}{16 \sqrt{n}} \widehat{\mB} \vx - \frac{1}{4} \sqrt{ \frac{\lambda L}{ \sqrt{n} }} \ve_1  }^2 + \frac{\sqrt{\alpha} \lambda L }{16 \sqrt{n} \mu_y} \sum\limits_{i=1}^{m} \Gamma \left( \frac{1}{4} \sqrt{\frac{ \sqrt{\alpha} L}{ \lambda \sqrt{n}}} x_i \right) \\
        & = \frac{L^2}{512 n \mu_y} \norm{\widehat{\mB} \vx}^2 - \frac{L}{64 \mu_y} \sqrt{ \frac{ \sqrt{\alpha} \lambda L}{ n^{3/2} } } \inner{\vx}{\ve_1} + \frac{\sqrt{\alpha} \lambda L}{ 16 \sqrt{n} \mu_y} \sum\limits_{i=1}^{m} \Gamma \left( \frac{1}{4} \sqrt{\frac{ \sqrt{\alpha} L}{ \lambda \sqrt{n} }} x_i \right) + \frac{\lambda L}{32 \sqrt{n} \mu_y}.
    \end{align*}
    Letting $\tilde{\vx} = \frac{1}{4} \sqrt{\frac{ \sqrt{\alpha} L}{ \lambda \sqrt{n}}} \vx$, we have
    \begin{align*}
        \tilde{\phi}_{\mathrm{NCSC}} (\tilde{\vx})
        \triangleq {\phi}_{\mathrm{NCSC}} (\vx)
        = \frac{\lambda L}{16 \mu_y \sqrt{\alpha n}} \left( \frac{1}{2} \norm{\widehat{\mB} \tilde{\vx}}^2 - \sqrt{\alpha} \inner{\tilde{\vx}}{\ve_1} + \alpha \sum\limits_{i=1}^{m} \Gamma (\tilde{x}_i) \right) + \frac{\lambda L}{32 \sqrt{n} \mu_y}.
    \end{align*}
    By Proposition \ref{prop:nonconvex:prop:base},
    \begin{align*}
        {\phi}_{\mathrm{NCSC}} (\vzero_{}) - \min_{\vx \in \BR^{m+1}} {\phi}_{\mathrm{NCSC}} (\vx)
        & = \tilde{\phi}_{\mathrm{NCSC}} (\vzero_{}) - \min_{\tilde{\vx} \in \BR^{m+1}}  \tilde{\phi}_{\mathrm{NCSC}} (\tilde{\vx}) \\
        & = \frac{\lambda L}{16 \mu_y \sqrt{\alpha n}} \left( \frac{\sqrt{\alpha}}{2} + 10 \alpha m \right) \\
        & \le \frac{165888 n \mu_y \eps^2 }{L^2 \alpha} + \frac{3311760 n \mu_y \eps^2 m}{L^2 \sqrt{\alpha} } \\
        & \le \frac{165888}{3483648} \Delta + \frac{3317760}{3483648} \Delta \le \Delta.
    \end{align*}
    
    \item Since $\alpha \le 1$, we have $\frac{\Delta L^2 \sqrt{\alpha}}{3483648 n \eps^2 \mu_y} \ge \frac{\Delta L^2 \alpha}{3483648 n \eps^2 \mu_y} \ge 2$ and consequently $m \ge 2$.
    By Proposition \ref{prop:nonconvex:prop:base},
    \begin{align*}
        \min_{\vx \in \fF_M} \norm{\nabla {\phi}_{\mathrm{NCSC}} (\vx) } 
        & = \frac{1}{4} \sqrt{\frac{ \sqrt{\alpha} L}{ \lambda \sqrt{n}}} \min_{\hat{\vx} \in \fF_M} \norm{ \nabla \tilde{\phi}_{ \mathrm{NCSC} } ( \tilde{\vx}) } \\
        & \ge \frac{1}{4} \sqrt{\frac{ \sqrt{\alpha} L}{ \lambda \sqrt{n} }} \frac{\lambda L}{16 \mu_y \sqrt{\alpha n}} \frac{\alpha^{3/4}}{4} \ge 9 \eps.
    \end{align*}
\end{enumerate}
This completes the proof.
\end{proof}

\subsection{Results for the Smooth Cases}\label{appendix:minimax:smooth}
In this subsection, we give the formal statements of the lower bounds in Table \ref{table-minimax}.

\paragraph{Function class} We develop lower bounds for PIFO algorithms that find a suboptimal solution or near stationary point of Problem~\ref{prob:main} in the following sets
{ \small
\begin{align*}
    \fF^*_{\mathrm{CC}}(R_x, R_y, L, \mu_x, \mu_y) = \bigg\{ f(\vx, \vy) = \frac{1}{n} \sum_{i=1}^n f_i(\vx, \vy) ~\Big\vert~ f\colon \fX \times \fY \to \BR, \, \diam (\fX) \le 2 R_x, \\
    \diam (\fY) \le 2 R_y,\, f_i \text{ is } L\text{-smooth}, f \text{ is } (\mu_x, \mu_y) \text{-convex-concave } \bigg\}, 
    \\
    \fF^*_{\mathrm{NCC}}(\Delta, L, \mu_x, \mu_y) = \bigg\{ f(\vx, \vy) = \frac{1}{n} \sum_{i=1}^n f_i(\vx, \vy) ~\Big\vert~ f\colon \fX \times \fY \to \BR, \,  \phi(\vzero) - \inf_{\vx \in \fX} \phi(\vx) \le \Delta,  \\
    f_i \text{ is } L\text{-smooth}, f \text{ is } (-\mu_x, \mu_y)\text{-convex-concave } \bigg\}.
\end{align*}}
We can also add a condition that each component function is
$(\mu_x, \mu_y)$-convex-concave to the definition. This induces a more restrictive function class but will not affect our construction.
Such a definition better matches the assumptions of some upper bounds, e.g., \citet{luo2019stochastic}.

\paragraph{Optimization complexity}

We formally define the optimization complexity as follows.


\begin{definition}\label{defn:complexity:smooth}
    The optimization complexity with respect to the function class $\fF^*_{\mathrm{CC}}(\Delta, R, L, \mu)$ and ${\fF}^*_{\mathrm{NCC}}(\Delta, R, L, \mu)$ is defined as
    \begin{align*}
        \mathfrak{m}_{*}^{\mathrm{CC}}(\eps, R_x, R_y, L, \mu_x, \mu_y)       & \triangleq \inf_{\fA \in \mathscr{A}} \sup_{f \in \fF_{\mathrm{CC}}(R_x, R_y, L, \mu_x, \mu_y)}  T(\fA, f, \eps),       \\
        \mathfrak{m}_{*}^{\mathrm{NCC}}(\eps, \Delta, L, \mu_x, \mu_y)       & \triangleq \inf_{\fA \in \mathscr{A}} \sup_{f \in \fF_{\mathrm{NCC}}(\Delta, L, \mu_x, \mu_y)}  T(\fA, f, \eps),       
    \end{align*}
    where $T(\fA, f, \eps)$ is defined in Definition \ref{defn:complexity} with $\fF_{\mathrm{CC}}$ and $\fF_{\mathrm{NCC}}$ replaced by $\fF^*_{\mathrm{CC}}$ and $\fF^*_{\mathrm{NCC}}$.
\end{definition}
The lower bounds are listed as follows. Let $\kappa_x = L / \mu_x$ and $\kappa_y = L / \mu_y$ denote the condition number if they are well-defined.

\begin{theorem}\label{thm:strongly-strongly}
Let $n \ge 4$ be a positive integer and 
$L, \mu_x, \mu_y, R_x, R_y, \eps$ be positive parameters. Assume additionally that
$\kappa_x \ge \kappa_y \ge 2$ and
$\eps = \fO \left( \min \left\{ \frac{ n^2 \mu_x R_x^2}{  \kappa_x \kappa_y},  \mu_y R_y^2,  \mu_y R_y^2 \right\} \right)$.
Then we have
    \begin{align*}
        \mathfrak{m}_{*}^{\mathrm{CC}}(\eps, R_x, R_y, L, \mu_x, \mu_y)=
        \begin{cases}
            \Omega\left(\left(n {+} \sqrt{ \kappa_x \kappa_y } \right)\log\left(1/\eps \right)\right), &\text{ for } \kappa_x, \kappa_y = \Omega(n), \\
            \Omega\left(\left(n {+} \sqrt{\kappa_x n }\right)\log\left(1/\eps \right)\right), &\text{ for } \kappa_x = \Omega(n), \kappa_y = \fO(n), \\
            \Omega\left( n \right), &\text{ for } \kappa_x, \kappa_y = \fO(n).
        \end{cases}
    \end{align*}
\end{theorem}

The best known upper bound complexity in this case for IFO/PIFO algorithms is $\fO \Big( \big(n +  \frac{\sqrt{n} L}{\min\{\mu_x, \mu_y\}} \big) \log(1/\eps) \Big)$~\citep{luo2019stochastic}. There still exists a $\sqrt{n}$ gap to our lower bound.
\begin{theorem}\label{thm:convex-strongly}
Let $n \ge 2$ be a positive integer and 
$L, \mu_y, R_x, R_y, \eps$ be positive parameters. Assume additionally that $\kappa_y \ge 2$ and 
$\eps = \fO \left( \min \left\{ L R_x^2, \mu_y R_y^2 \right\} \right)$.
Then we have
    \begin{align*}
        \mathfrak{m}_{*}^{\mathrm{CC}}(\eps, R_x, R_y, L, 0, \mu_y)=
        \begin{cases}
            \Omega\left(n {+} R_x \sqrt{ \frac{n L}{\eps} }  {+} 
            R_x \sqrt{ \frac{L \kappa_y}{\eps} }
            {+} \sqrt{
            n \kappa_y
            } \log \left( \frac{1}{\eps} \right)  \right), &\text{for }
            \kappa_y 
            = \Omega(n), \\
            \Omega\left(n {+} R_x \sqrt{ \frac{n L}{\eps} }  {+} 
            R_x \sqrt{ \frac{L \kappa_y}{\eps} }
            \right), &\text{for } 
           \kappa_y 
           = \fO(n).
        \end{cases}
    \end{align*}
\end{theorem}


\begin{theorem}\label{thm:convex-concave}
Let $n \ge 2$ be a positive integer and 
$L, R_x, R_y, \eps$ be positive parameters. Assume additionally that $\eps \le \frac{L R_x R_y}{4} 
$. Then we have
    \begin{align*}
        \mathfrak{m}_{*}^{\mathrm{CC}}(\eps, R_x, R_y, L, 0, 0)=
        \Omega\left(n {+} \frac{ L R_x R_y }{\eps} {+} (R_x + R_y) \sqrt{ \frac{nL}{\eps} } \right).
    \end{align*}
\end{theorem}


\begin{theorem}\label{thm:nonconvex-strongly}
Let $n \ge 2$ be a positive integer and 
$L, \mu_x, \mu_y, \Delta, \eps$ be positive parameters. Assume additionally that $\eps^2 \le \frac{\Delta L^2 \alpha}{27216 n^2 \mu_y}$, where $\alpha = \min \left\{1, \frac{8(\sqrt{3} + 1)n^2\mu_x \mu_y}{45 L^2}, \frac{n^2\mu_y}{90 L} \right\}$. Then we have
    \begin{align*}
        \mathfrak{m}_{*}^{\mathrm{NCC}}(\eps, \Delta, L, \mu_x, \mu_y)=
        \Omega \left( n + \frac{\Delta L^2 \sqrt{\alpha}}{n \mu_y \eps^2}\right).
    \end{align*}
For $\kappa_y 
\ge n^2 / 90$, we have 
\[
\Omega \left( n + \frac{\Delta L^2 \sqrt{\alpha}}{n \mu_y \eps^2}\right)
= \Omega \left( n + \frac{\Delta L}{\eps^2} \min \left\{ \sqrt{\kappa_y}, \sqrt{\frac{\mu_x}{\mu_y} }  \right\} \right).
\]
\end{theorem}
With $\mu_x = L$, we obtain the result in Table~\ref{table-minimax}.

The proofs of these theorems are similar to those of Theorems \ref{thm:average:strongly-strongly} to \ref{thm:average:nonconvex-strongly}. We just list some key lemmas here and omit the lengthy proofs.

When $f$ is convex-concave, without loss of generality, we still assume $\mu_x \le \mu_y$.
The hard instances for Theorems \ref{thm:strongly-strongly} to \ref{thm:convex-concave} can be directly derived previous constructions.
Specially, for the hard instances constructed in Definitions \ref{defn:scsc}, \ref{defn:csc} and \ref{defn:cc}.
it suffices to replace $L$ by $\tilde{L} \triangleq \sqrt{ \frac{ 2 (L^2 - 2 \mu_y^2) }{n} + 2 \mu_y^2 } $.
One can check that by Proposition \ref{prop:base} and Lemma \ref{lem:scale}, each component function is $L$-smooth and $(\mu_x, \mu_y)$-convex-concave after this replacement.
Moreover, we have $\tilde{L} \ge \sqrt{ \frac{2}{n} } L $ for $n \ge 2$ and $\tilde{L} \le \sqrt{ \frac{4}{n} } L $ as long as $n \ge 4$ and $L^2 / \mu_y^2 \ge n - 2 \ge 2$.
For the hard instances constructed in the proofs of Theorems \ref{thm:scsc:example:2} and \ref{thm:csc:example:2}, there are also corresponding lower bounds in terms of the smoothness parameter.
And the hard instances constructed in the proofs of Lemmas \ref{lem:scsc:example:3} and \ref{lem:cc:example:2} also have $L$-smooth and $(\mu_x,\mu_y)$-convex-concave component functions.
As a result, the lower bounds in terms of the average smooth parameter
can be transformed into those in terms of the smooth parameter.


When $f$ is nonconvex in $\vx$ and strongly-concave in $\vy$, the hard instance is constructed as follows.
\begin{definition}\label{defn:ncsc:smooth}
For fixed $L, \mu_x, \mu_y, \Delta, n$, we define $f_{\mathrm{NCSC}, i}: \BR^{m+1} \times \BR^{m+1}  \rightarrow \BR$ as follows
\begin{align*}
    f^*_{\mathrm{NCSC}, i}(\vx, \vy) = \lambda {r}^{\mathrm{NCC}}_i\left(\vx / \beta, \vy / \beta; 
    m + 1,
    \sqrt[4]{\alpha}, {\vc}^{\mathrm{NCSC}}_* \right), \text{ for } 1 \le i \le n,
\end{align*}
where
\begin{align*}
    \alpha & = \min \left\{1, \frac{n^2 \mu_y}{90 L}, \frac{8 (\sqrt{3} + 1) n^2 \mu_x \mu_y}{45 L^2} \right\}, \;
    {\vc}^{\mathrm{NCSC}}_* = \left( \frac{4 n \mu_y}{L}, \frac{\sqrt{\alpha} L }{4 n \mu_y }, \sqrt[4]{\alpha}  \right) ,\; \\
    \lambda & = \frac{82944 n^3 \mu_y^2 \eps^2}{L^3 \alpha}, \;
    \beta = 2 \sqrt{\lambda n / L}
    \; \mbox{ and } \; 
    m = \floor{\frac{\Delta L^2 \sqrt{\alpha}}{217728 n^2 \eps^2 \mu_y}}.
\end{align*}
Consider the minimax problem
\begin{equation}
    \min_{\vx \in \BR^{m+1}}\max_{\vy \in \BR^{m+1}} f^*_{\mathrm{NCSC}}(\vx, \vy) \triangleq \frac{1}{n}\sum_{i=1}^n f^*_{\mathrm{NCSC},i}(\vx, \vy). \label{prob:ncsc:smooth}
\end{equation}
Define $\phi^*_{\mathrm{NCSC}} (\vx) = \max_{\vy \in \BR^{m+1}} f^*_{\mathrm{NCSC}} (\vx, \vy)$.
\end{definition}
Then we have the following proposition, whose proof is similar to that of Proposition \ref{prop:nonconvex-strongly} and is omitted.

\begin{proposition}\label{prop:nonconvex-strongly:smooth}
For any $n \ge 2$, $L / \mu_y \ge 4$ and $\eps^2 \le \frac{\Delta L^2 \alpha}{435456 n^2 \mu_y}$, the following properties hold:
\begin{enumerate}
    \item $f^*_{\mathrm{NCSC}, i}$ is $L$-smooth and $ (- \mu_x, \mu_y)$-convex-concave. 
    \item $\phi^*_{\mathrm{NCSC}}(\vzero_{m+1}) - \min_{\vx \in \BR^{m+1}} \phi^*_{\mathrm{NCSC}}(\vx^*) \le \Delta$.
    \item $m \ge 2$ and for $M = m-1$, $\min_{\vx \in \fF_{M}} \norm{\nabla \phi^*_{\mathrm{NCSC}}(\vx)} \ge  9 \eps.$
\end{enumerate}
\end{proposition}
The proof of Theorem \ref{thm:nonconvex-strongly} is similar to that of Theorem \ref{thm:average:nonconvex-strongly} and is omitted.

\section{Details for Section \ref{sec:min}}

\subsection{The Setup}\label{sec:min:setup:appendix}

\paragraph{Function class} We develop lower bounds for PIFO algorithms that find the suboptimal solution or near stationary point of Problem~(\ref{prob:min}) in the following four sets.
{ \small
\begin{align*}
    \fF^*_{\mathrm{C}}(R, L, \mu) = \bigg\{ f(\vx) = \frac{1}{n} \sum_{i=1}^n f_i(\vx) ~\Big\vert~ f: \fX \to \BR, \, \diam (\fX) \le 2 R,  
    \\
    f_i \text{ is } L\text{-smooth} , f \text{ is } \mu\text{-strongly convex } \bigg\}, 
    \\
    \fF^*_{\mathrm{NC}}(\Delta, L, \mu) = \bigg\{ f(\vx) = \frac{1}{n} \sum_{i=1}^n f_i(\vx) ~\Big\vert~ f: \fX \to \BR, \,  f(\vzero) - \inf_{\vx \in \fX} f(\vx) \le \Delta,  
    \\
    f_i \text{ is } L\text{-smooth} , f \text{ is } (-\mu)\text{-weakly convex } \bigg\}, 
    \\
    {\fF}_{\mathrm{C}}(R, L, \mu) = \bigg\{f(\vx) = \frac{1}{n} \sum_{i=1}^n f_i(\vx) ~\Big\vert~ f: \fX \to \BR, \, \diam (\fX) \le 2 R, \\
    \{f_i\}_{i=1}^n \text{ is } L\text{-average smooth} , f \text{ is } \mu\text{-strongly convex } \bigg\}.
    \\
    {\fF}_{\mathrm{NC}}(\Delta, L, \mu) = \bigg\{f(\vx) = \frac{1}{n} \sum_{i=1}^n f_i(\vx) ~\Big\vert~ f: \fX \to \BR, \,  f(\vzero) - \inf_{\vx \in \fX} f(\vx) \le \Delta,\\
    \{f_i\}_{i=1}^n \text{ is } L\text{-average smooth} , f \text{ is } (-\mu)\text{-weakly convex } \bigg\}. 
    \\
\end{align*}
}
For the definitions of $\fF^*_{\mathrm{C}}$ and $\fF^*_{\mathrm{NC}}$,
we can also add a condition that each component function is 
$\mu$-strongly convex or $(-\mu)$-weakly convex to the definitions respectively. This induces a more restrictive function class but will not affect our construction.
In fact, the component function of the hard instances constructed in
\citet{woodworth2016tight}, \citet{hannah2018breaking}
and \citet{zhou2019lower}
is also $\mu$-strongly convex or $(-\mu)$-weakly convex.

\paragraph{Optimization complexity}

We formally define the optimization complexity as follows.


\begin{definition}
    For a function $f$, a PIFO algorithm $\fA$ and a tolerance $\eps > 0$, the number of queries needed by $\fA$ to find $\eps$-suboptimal solution to the Problem (\ref{prob:min}) or the $\eps$-stationary point of $f(\vx)$ is defined as
    \begin{align*}
        T(\fA, f, \eps) = \begin{cases}
        \inf \left\{ T \in \BN ~\vert~ \E f(\vx_{\fA, T}) - \min_{\vx \in \fX} f(\vx) < \eps \right\}
        ~~~ \text{if } f \in \fF^*_{\mathrm{C}}(R, L, \mu) \cup {\fF}_{\mathrm{C}}(R, L, \mu) \\
        \inf \left\{ T \in \BN ~\vert~ \E \norm{\nabla f(\vx_{\fA, T})} < \eps \right\}, ~~~\text{if } f \in \fF^*_{\mathrm{NC}}(\Delta, L, \mu) \cup {\fF}_{\mathrm{NC}}(\Delta, L, \mu)
        \end{cases}
    \end{align*}
    where $\vx_{\fA, T}$ is the point obtained by the algorithm $\fA$ at time-step $T$.

    Furthermore, the optimization complexity with respect to 
    these function classes
    are defined as
    \begin{align*}
        \mathfrak{m}_{*}^{\mathrm{C}}(\eps, R, L, \mu)       & \triangleq \inf_{\fA \in \mathscr{A}} \sup_{f \in \fF^*_{\mathrm{C}}(R, L, \mu)}  T(\fA, f, \eps),       \\
        {\mathfrak{m}}^{\mathrm{C}}(\eps, R, L, \mu) & \triangleq \inf_{\fA \in \mathscr{A}} \sup_{f \in {\fF}_{\mathrm{C}}(R, L, \mu)}  T(\fA, f, \eps). \\
        \mathfrak{m}_{*}^{\mathrm{NC}}(\eps, \Delta, L, \mu)       & \triangleq \inf_{\fA \in \mathscr{A}} \sup_{f \in \fF^*_{\mathrm{NC}}(\Delta, L, \mu)}  T(\fA, f, \eps),       \\
        {\mathfrak{m}}_{}^{\mathrm{NC}}(\eps, \Delta, L, \mu) & \triangleq \inf_{\fA \in \mathscr{A}} \sup_{f \in {\fF}_{\mathrm{NC}}(\Delta, L, \mu)}  T(\fA, f, \eps).
    \end{align*}

\end{definition}
\subsection{More Properties of the Hard Instances}\label{sec:min:instance:appendix}
In this subsection, we 
present more properties of
the hard instance $\{ r_i \}_{i=1}^n$ 
constructed in Section~\ref{sec:min:instance}.
First, We can determine the smoothness and strong convexity parameters of $r_i$ as follows.
\begin{proposition}\label{prop:base}
Suppose that $0 \le \omega, \zeta \le \sqrt{2}$ and $ c_1 \ge 0$.
\begin{enumerate}
    \item \label{prop:base-c} \textbf{Convex case.} For $c_2 = 0$, we have that $r_i$ is $(2n + c_1)$-smooth and $c_1$-strongly-convex, and $\{r_i\}_{i=1}^n$ is $L'$-average smooth where
    \begin{align*}
        L' = \sqrt{\frac{4}{n} \left[ (n + c_1)^2 + n^2 \right] + c_1^2}.
    \end{align*}
    \item \label{prop:base-nc} \textbf{Non-convex case.} For $c_1 = 0$, 
    we have that $r_i$ is $(2n + 180 c_2)$-smooth and $[-45 (\sqrt{3} - 1) c_2]$-weakly-convex, and $\{r_i\}_{i=1}^n$ is $4\sqrt{n + 4050 c_2^2}$-average smooth.
\end{enumerate}
\end{proposition}
The proof of Proposition \ref{prop:base} is given in Appendix~\ref{appendix:min:proof:base}.

Recall the subspaces $\{\fF_k\}_{k=0}^m$ which are defined as
\begin{align*}
\fF_k = \begin{cases}
\spn \{ \ve_1, \ve_{2}, \cdots, \ve_{k}\}, & \text{for } 1 \le k \le m, \\
 \{\vzero\}, & \text{for } k=0.
\end{cases}
\end{align*}
When we apply a PIFO algorithm $\fA$ to solve the Problem (\ref{prob:r}), 
Lemma \ref{lem:jump} implies that $\vx_t = \vzero$ will hold until algorithm $\fA$ draws the component $r_1$ or calls the FO.
Then for any $t < T_1 = \min_t \{t: i_t = 1 \mbox{ or } a_t=1 \}$, we have $\vx_t \in \fF_0$ while $\vx_{T_1} \in \fF_1$ holds. The value of $T_1$ can be regarded as the smallest integer such that $\vx_{T_1} \in \fF_1 \setminus \fF_0$ could hold.
Similarly, for $T_1 \le t < T_2 = \min_t \{ t > T_1: i_t = 2 \mbox{ or } a_t=1 \} $, there holds $\vx_t \in \fF_1$ while we can ensure that $\vx_{T_2} \in \fF_2$.

We can define $T_k$ to be the smallest integer such that $\vx_{T_k} \in \fF_k \setminus \fF_{k-1}$ could hold. We give the formal definition of $T_k$ recursively and connect it to geometrically distributed random variables in the following corollary.

\begin{corollary}\label{coro:stopping-time}
Assume we employ a PIFO algorithm $\fA$ to solve the Problem (\ref{prob:r}). Let
\begin{align}\label{def:stoppong-time}
    T_0 = 0,~\text{ and } 
    ~T_k = \min_t \{t: t > T_{k-1}, i_t \equiv k~(\bmod ~n) \text{ or } a_t=1 \}~\text{ for } k \ge 1.
\end{align}
Then we have 
\[
\vx_t \in \fF_{k-1}, ~~~\text{ for } t < T_k, k \ge 1. 
\]
Moreover, the random variables $\{ Y_k \}_{k \ge 1}$ such that $Y_k \triangleq T_k - T_{k-1} $ are mutual independent and $Y_k$ follows a geometric distribution with success probability $p_{k'} + q - p_{k'}q$ where $k' \equiv k ~ (\bmod ~ n)$ and $k' \in [n]$.
 \end{corollary}
The proof of Corollary \ref{coro:stopping-time} is similar to that of Corollary \ref{coro:convex-concave:stopping-time}.
The basic idea of our analysis is that we guarantee that the minimizer of $r$ does not lie in $\fF_k $ for $k < m$
and assure that the PIFO algorithm extends the space of $\spn\{\vx_0,\vx_1,\dots,\vx_t\}$ slowly with $t$ increasing. 
We know that $\spn\{\vx_0,\vx_1,\dots,\vx_{T_k}\} \subseteq \fF_{k-1}$ by Corollary \ref{coro:stopping-time}.   
Hence, $T_k$ is just the quantity that measures how  $\spn\{\vx_0,\vx_1,\dots,\vx_t\}$ expands. 
Note that $T_k$ can be written as the sum of geometrically distributed random variables.
Recalling Lemma \ref{lem:geo}, 
we can obtain how many PIFO calls we need.
\begin{lemma}\label{lem:base}
Let $H_r(\vx)$ be a criterion of measuring how $\vx$ is close to solution to Problem~(\ref{prob:r}).
If $M$ satisfies $1 \le M < m$,
$\min_{\vx \in \fX \cap \fF_M } H_r(\vx) \ge 9\eps$
and $N = \frac{n(M+1)}{4(1+c_0)}$, 
then we have
\begin{align*}
    \min_{t \le N} \E H_r(\vx_t)  \ge \eps
\end{align*}
\end{lemma}

\begin{remark}
If $r(\vx)$ is convex in $\vx$, we set $H_r(\vx) = r(\vx) - \min_{\vx \in \fX} r(\vx)$. If $r(\vx)$ is nonconvex, we set $H_r(\vx) = \norm{ \nabla r(\vx) }.$
\end{remark}
The proof of Lemma \ref{lem:base} is similar to that of Lemma \ref{lem:minimiax:base}.

\renewcommand{\arraystretch}{1}

\renewcommand{\arraystretch}{1}

\subsection{Main Results}\label{sec:min:result:appendix}
In this subsection, we present the formal statements of the lower bounds in Tables \ref{table-min} and \ref{table-min-average}.

\subsubsection{Smooth Cases}
We first focus on the smooth cases, i.e., the results in Table \ref{table-min}.
When $\mu \neq 0$, the conditon number is denoted by $\kappa = L / \mu$.

When $f$ is strongly-convex, we have the following result.
\begin{theorem}\label{thm:strongly}
Let $n \ge 2$ be a positive integer and $L, \mu, R, \eps$ be positive parameters. Assume additionally that $\kappa=L/\mu \ge 2$ and $\eps \le LR^2/4$. Then we have
    \begin{align*}
        \mathfrak{m}_*^{\mathrm{C}}(\eps, R, L, \mu)=
        \begin{cases}
            \Omega\left(\left(n {+} \sqrt{\kappa n}\right)\log\left(1/\eps \right)\right), &\text{ for } \kappa = \Omega(n), \\
            \Omega\left( n + \left( \frac{n}{1 + (\log(n/\kappa))_{+}} \right)\log\left(1/\eps\right) \right), &\text{ for } \kappa = \fO(n).
        \end{cases}
    \end{align*}
\end{theorem}
From Appendix \ref{sec:min:strongly}, our construction only requires the dimension to be $\fO ( 1 + \sqrt{\kappa / n} \log (1 / \eps)  )$, which is much smaller than $\tilde{\fO} ( \kappa n / \eps )$ in \citet{woodworth2016tight}.

Next, we give the lower bound when the objective function is not strongly-convex.
\begin{theorem}\label{thm:convex}
Let $n \ge 2$ be a positive integer and $L, R, \eps$ be positive parameters. Assume additionally that $\eps \le LR^2/4$.
Then we have
    \begin{align*}
        \mathfrak{m}_{*}^{\mathrm{C}}(\eps, R, L, 0)=
        \Omega\left(n {+} R\sqrt{nL/\eps} \right)
    \end{align*}
\end{theorem}
From Appendix \ref{sec:min:convex}, our construction requires the dimension to be $\fO (1 + R \sqrt{L / (n \eps)} )$, which is much smaller than $\tilde{\fO} (L^2 R^4 / \eps^2 )$ in \citet{woodworth2016tight}.



Finally, we give the lower bound when the objective function is non-convex.


\begin{theorem}\label{thm:nonconvex}
Let $n \ge 2$ be a positive integer and $L, \mu, \Delta, \eps$ be positive parameters. Assume additionally that $\eps^2 \le \frac{\Delta L \alpha}{81648 n}$, where $\alpha = \min \left\{1, \frac{(\sqrt{3} + 1)n\mu}{30 L}, \frac{n}{180} \right\}$.
Then we have
    \begin{align*}
        \mathfrak{m}_{*}^{\mathrm{NC}}(\eps, \Delta, L, \mu)=
        \Omega \left( n + \frac{\Delta L \sqrt{\alpha}}{\eps^2}\right)
    \end{align*}
\end{theorem}
For $n > 180$, we have
\begin{align*}
    \Omega \left(n + \frac{\Delta L \sqrt{\alpha}}{\eps^2}\right) 
    = \Omega \left( n + \frac{\Delta}{\eps^2} \min\{L, \sqrt{n\mu L}\} \right).
\end{align*}
From Appendix \ref{sec:min:nonconvex}, our construction only requires the dimension to be $\fO \left(1 + \frac{\Delta}{\eps^2} \min \{L/n, \sqrt{\mu L /n}\} \right)$, which is much smaller than 
$\fO \left( \frac{\Delta}{\eps^2} \min \{L, \sqrt{n \mu L}\} \right) $ in \citet{zhou2019lower}.


\subsubsection{Average Smooth Case}
Then we give the results for the average smooth cases, i.e., the results in Table \ref{table-min-average}.
When $\mu \neq 0$, the condition number is still denoted by $\kappa = L / \mu$.

When $f$ is strongly-convex, we have the following result.
\begin{theorem}\label{thm:average:strongly}
Let $n \ge 4$ be a positive integer and $L, \mu, R, \eps$ be positive parameters. Assume additionally that $\kappa=L/\mu \ge 2$ and $\eps \le LR^2/4$. Then we have
    \begin{align*}
        {\mathfrak{m}}_{}^{\mathrm{C}}(\eps, R, L, \mu)=
        \begin{cases}
            \Omega\left(\left(n {+} n^{3/4} \sqrt{\kappa}\right)\log\left( 1/\eps \right)\right) , &\text{ for } \kappa = \Omega(\sqrt{n}), \\
            \Omega\left( n + \left( \frac{n}{1 + (\log(\sqrt{n}/\kappa))_{+}} \right)\log\left(1/\eps\right) \right), &\text{ for } \kappa = \fO(\sqrt{n}).
        \end{cases}
    \end{align*}
\end{theorem}
From Appendix \ref{sec:min:strongly}, our construction only requires the dimension to be $\fO ( 1 + \sqrt{\kappa} / \sqrt[4]{n} \log (1 / \eps)  )$, which is much smaller than ${\fO} ( n + n^{3/4} \sqrt{\kappa} \log(1 / \eps)) $ in \citet{zhou2019lower}.

The next theorem gives the lower bound when $f$ is only convex.
\begin{theorem}\label{thm:average:convex}
Let $n \ge 2$ be a positive integer and $L, R, \eps$ be positive parameters. Assume additionally that $\eps \le LR^2/4$. Then we have
    \begin{align*}
        {\mathfrak{m}}_{}^{\mathrm{C}}(\eps, R, L, 0)=
        \Omega\left(n + R n^{3/4} \sqrt{L/\eps} \right)
    \end{align*}
\end{theorem}
From Appendix \ref{sec:min:convex}, our construction requires the dimension to be $\fO (1 + R \sqrt{L / (\sqrt{n} \eps)} )$, which is much smaller than ${\fO} (n + n^{3/4} \sqrt{L / \eps} )$ in \citet{zhou2019lower}.

Finally, we give the lower bound when the objective function is non-convex.

\begin{theorem}\label{thm:average:nonconvex}
Let $n \ge 2$ be a positive integer and $L, \mu, \Delta, \eps$ be positive parameters. Assume additionally that $\eps^2 \le \frac{\Delta L \alpha}{435456 \sqrt{n}}$, where $\alpha = \min \left\{1, \frac{8(\sqrt{3} + 1) \sqrt{n} \mu}{45 L}, \sqrt{\frac{n}{270}} \right\}$.
Then we have
    \begin{align*}
        {\mathfrak{m}}_{\eps}^{\mathrm{NC}}(\eps, \Delta, L, \mu)=
        \Omega \left( n + \frac{\Delta L \sqrt{n \alpha}}{\eps^2}\right)
    \end{align*}
\end{theorem}
For $n > 270$, we have
\begin{align*}
    \Omega \left( n + \frac{\Delta L \sqrt{n \alpha}}{\eps^2}\right) 
    & = \Omega \left( n + \frac{\Delta}{\eps^2} \min \left\{ \sqrt{n} L, n^{3/4} \sqrt{\mu L} \right\} \right).
\end{align*}
From Appendix \ref{sec:min:nonconvex}, our construction only requires the dimension to be $\fO \big(1 + \frac{\Delta}{\eps^2} \min \{L/\sqrt{n}, \sqrt{\mu L /\sqrt{n}}\} \big)$, which is much smaller than 
$\fO \left( \frac{\Delta}{\eps^2} \min \{\sqrt{n}L, n^{3/4} \sqrt{ \mu L}\} \right) $ in \citet{zhou2019lower}.

\subsection{Construction for the Strongly-Convex Case}\label{sec:min:strongly}

The analysis of lower bound complexity for the strongly-convex case depends on the following construction.

\begin{definition}\label{defn:sc}
For fixed $L, \mu, R, n$ such that $L / \mu \ge 2$, 
let $\alpha = \sqrt{ \frac{2(L/\mu - 1)}{n} + 1}$.
We define $f_{\mathrm{SC}, i} : \BR^m \rightarrow \BR$ as follows
\begin{align*}
    f_{\mathrm{SC}, i} (\vx) = \lambda r_i\left(\vx / \beta; m, 0, \sqrt{\frac{2}{\alpha + 1}}, \vc^{\mathrm{SC}} \right), \text{ for } 1 \le i \le n,
\end{align*}
where
$
    \vc^{\mathrm{SC}} = \left( \frac{2 n}{ L / \mu -1 }, 0, 1 \right),\;
    \lambda = \frac{2\mu R^2 \alpha n }{L / \mu - 1}\;
    \text{ and }\; \beta = \frac{ 2R\sqrt{\alpha} n }{ L / \mu - 1 }.$
Consider the minimization problem 
\begin{align}
    \min_{\vx \in \fX} f_{\mathrm{SC}}(\vx) \triangleq \frac{1}{n}\sum_{i=1}^n f_{\mathrm{SC},i}(\vx). \label{prob:sc}\\[-0.8cm]\nonumber
\end{align}
where $\fX = \{ \vx \in \BR^m : \norm{\vx} \le R  \}$.
\end{definition}
With this definition, we have the following proposition.
\begin{proposition}\label{prop:strongly}
For any $n \ge 2$, $m \ge 2$, $f_{\mathrm{SC}, i}$ and $f_{\mathrm{SC}}$ in Definition \ref{defn:sc} satisfy:
\begin{enumerate}
    \item $f_{\mathrm{SC}, i}$ is $L$-smooth and $\mu$-strongly-convex. Thus, $f_{\mathrm{SC}}$ is $\mu$-strongly-convex. \label{scp1}
    \item The minimizer of the function $f_{\mathrm{SC}}$ is
    $$\vx^{*} = \argmin_{\vx \in \BR^m}f_{\mathrm{SC}}(\vx) = 
    \frac{ 2R \sqrt{\alpha} }{\alpha - 1}
    (q^{1}, q^{2}, \dots, q^m)^{\top},
    $$ 
    where $\alpha = \sqrt{\frac{2(L/\mu - 1)}{n} + 1}$ and $q = \frac{\alpha - 1}{\alpha + 1}$.
    Moreover, $f_{\mathrm{SC}}(\vx^*) = - \frac{\mu R^2 \alpha}{ \alpha + 1 } $ and $\norm{\vx^*} \le R$.
    \item For $1 \le k \le m - 1$, we have
    \begin{align*}
    \min_{\vx \in \fX \cap \fF_k} f_{\mathrm{SC}}(\vx) - \min_{\vx \in \fX} f_{\mathrm{SC}}(\vx) \ge 
    \frac{\mu R^2 \alpha}{ \alpha + 1}
    q^{2k}. 
    \end{align*}\label{scp3}
\end{enumerate}
\end{proposition}
\begin{proof}
\begin{enumerate}
    \item Just recall Proposition \ref{prop:base} and Lemma \ref{lem:scale}.
    \item It is easy to check 
    $
    f_{\mathrm{SC}}(\vx) 
    = \frac{\xi}{2} \norm{\mB\left(m, 0, 
    \zeta
    \right) \vx}^2 + \frac{\mu}{2} \norm{\vx}^2 
    - \xi \beta \inner{\ve_1}{\vx},$
    where $\xi = \lambda / \beta^2 = (L - \mu) / (2n)$ and $\zeta = \sqrt{ \frac{2}{\alpha + 1} }$.
    Letting $ \nabla f_{\mathrm{SC}} (\vx) = \vzero $, we have 
    \begin{align}\label{proof:strongly:minimizer0}
    \left( \xi \mA \left(m, 0, 
    \zeta
    \right) + \mu \mI \right)\vx = \xi \beta \ve_1.
    \end{align}
Since $\frac{\mu}{\xi} = \frac{2n\mu}{L-\mu} = \frac{4}{\alpha^2 - 1}$, $q = \frac{\alpha - 1}{\alpha + 1}$ is a root of the equation
$    z^2 - \left(2 + \frac{2n\mu}{L-\mu}\right) z + 1 = 0$.
Note that $\zeta^2 + 1 + \frac{2n\mu}{L-\mu} = \frac{1}{q}$,
one can
check that the solution to Equations (\ref{proof:strongly:minimizer0}) is 
$\vx^{*} = \frac{\beta(\alpha+1)}{2} (q^{1}, q^{2}, \dots, q^m)^{\top} = \frac{ 2 R \sqrt{\alpha} }{\alpha - 1} (q^{1}, q^{2}, \dots, q^m)^{\top},$
and 
$ f_{\mathrm{SC}}(\vx^*) 
= - \frac{ \xi \beta^2 (\alpha + 1) q }{4}
= - \frac{\lambda (\alpha - 1) }{4}
= - \frac{ \mu R^2 \alpha }{\alpha + 1}.
$
Moreover, we have
\[
\norm{\vx^*}^2
= \frac{4R^2 \alpha}{ (\alpha - 1)^2 } \cdot \frac{q^2 - q^{2m+2}}{1 - q^2}
\le \frac{4R^2 \alpha}{ (\alpha - 1)^2 } \cdot \frac{q^2}{1 - q^2}
= 
R^2.
\]

    \item 

If $\vx \in \fF_k$, $1 \le k < m$, then $x_{k+1} = x_{k+2} = \cdots = x_{m} = 0$. 

Let $\vy$ be the first $k$ coordinates of $\vx$
and $\mA_k$ be first $k$ rows and columns of $\mA(m,0,\zeta)$.
Then we can rewrite $f_{\mathrm{SC}}(\vx)$ as
    $f_k(\vy) \triangleq f_{\mathrm{SC}}(\vx) = \frac{\xi}{2} \vy^{\top} \mA_k \vy - 
    \xi \beta
    \inner{\hat{\ve}_1}{\vy},$
where $\hat{\ve}_1$ is the first $k$ coordinates of $\ve_1$.
Let $\nabla f_k (\vy) = \vzero_k$.
By some calculation, the solution 
is 
\begin{align*}
    \frac{ (\alpha - 1) \beta q^{k}}{2 (1+q^{2k+1}) } \left( q^{-k} - q^k, q^{-k+1} - q^{k-1}, \dots, q^{-1} - q^{1} \right)^{\top}.
\end{align*}
Thus,
$\min_{\vx \in \fF_k} f_{\mathrm{SC}}(\vx) = \min_{\vy \in \BR^k} f_k(\vy) = -\frac{ \lambda ( \alpha - 1) }{4} \cdot \frac{1 - q^{2k}}{1 + q^{2k+1}} 
,$
and 
\begin{align*}
    & \quad \min_{\vx \in \fX \cap \fF_k} f_{\mathrm{SC}}(\vx) - \min_{\vx \in \fX} f_{\mathrm{SC}}(\vx) \\
    & \ge \min_{\vx \in \fF_k} f_{\mathrm{SC}}(\vx) - f_{\mathrm{SC}}(\vx^*) 
    = \frac{\lambda (\alpha - 1) }{4}
    \left( 1 - \frac{1-q^{2k}}{1+q^{2k+1}} \right) \\
    & = \frac{\lambda (\alpha - 1) }{4}
    q^{2k}\frac{1+q}{1+q^{2k+1}} 
    \ge \frac{\mu R^2 \alpha}{\alpha + 1 } q^{2k}
\end{align*}
\end{enumerate}
This completes the proof.
\end{proof}
With this hard instance, we have the following result.
\begin{theorem}\label{thm:strongly:example}
Consider the minimization problem (\ref{prob:sc}) and $\eps > 0$. Suppose that
\begin{align*}
    n \ge 2,\;\,
    \eps \le \frac{\mu R^2}{18} \left(\frac{\alpha-1}{\alpha+1}\right)^{2}\text{ and } 
    m = \left\lfloor \frac{1}{4}\left(\sqrt{2\frac{L/\mu - 1}{n} + 1}\right) \log \left(\frac{\mu R^2}{9\eps}\right) \right\rfloor + 1,
\end{align*}
where $\alpha = \sqrt{\frac{2(L/\mu - 1)}{n} + 1}$.
In order to find $\hat{\vx} \in \fX$ such that $ \E f_{\mathrm{SC}} (\hat{\vx}) - \min_{\vx \in \fX} f_{\mathrm{SC}} (\vx) < \eps$, PIFO algorithm $\fA$ needs at least $N$ queries, where
\begin{align*}
    N =
    \begin{cases}
    \Omega\left(\left(n + \sqrt{\frac{nL}{\mu}}\right)\log\left( \frac{1}{\eps} \right)\right), &\text{ for } \frac{L}{\mu} \ge \frac{n}{2} + 1, \\
    \Omega\left(n + \left(\frac{n}{1 + (\log(n\mu/L))_+}\right)\log\left( \frac{1}{\eps} \right)\right), &\text{ for } 2 \le \frac{L}{\mu} < \frac{n}{2} + 1.
    \end{cases}
\end{align*}
\end{theorem}
\begin{proof}
Let $\Delta = \frac{\mu R^2 \alpha}{ \alpha + 1 }$.
Since $\alpha > 1$, we have $\frac{\mu R^2}{2} < \Delta < \mu R^2$.
Let $M = \floor{\frac{\log (9\eps/\Delta)}{2\log q}}$, then we have 
$    \min_{\vx \in \fX \cap \fF_M} f_{\mathrm{SC}}(\vx) - \min_{\vx \in \fX} f_{\mathrm{SC}}(\vx) \ge \Delta q^{2M} \ge 9 \eps,$
where the first inequality is according to the third property of Proposition \ref{prop:strongly}.

By Lemma \ref{lem:base}, if $1 \le M < m$ and $N = \frac{(M+1)n}{4(1+c_0)}$, we have
   $ \min_{t \le N} \E f_{\mathrm{SC}}(\vx_t) - \min_{\vx \in \fX }f_{\mathrm{SC}}(\vx) \ge \eps.$
Therefore, in order to find $\hat{\vx} \in \fX$ such that $\E f_{\mathrm{SC}}(\hat{\vx}) - \min_{\vx \in \fX} f_{\mathrm{SC}}(\vx) < \eps$, $\fA$ needs at least $N$ queries.

We estimate $-\log(q)$ and $N$ in two cases.
\begin{enumerate}
\item If $L/\mu \ge n/2 + 1$, then $\alpha = \sqrt{2 \frac{L/\mu - 1}{n} + 1} \ge \sqrt{2}$. 
Observe that function $h(\beta) = \frac{1}{\log \left(\frac{\beta + 1}{\beta - 1}\right)} - \frac{\beta}{2}$ is increasing when $\beta>1$.
Thus, we have
{ \small
\begin{align*}
    -\frac{1}{\log(q)} &= \frac{1}{\log \left(\frac{\alpha + 1}{\alpha - 1}\right)} 
    \ge \frac{\alpha}{2} + h(\sqrt{2}) 
    = \frac{1}{2}\sqrt{2 \frac{L/\mu - 1}{n} + 1} + h(\sqrt{2}) 
    \ge \frac{\sqrt{2}}{4} \left( \sqrt{2\frac{L/\mu - 1}{n}} + 1 \right) + h(\sqrt{2}) 
\end{align*} }
and 
{ \small
\begin{align*}
    N &= \frac{(M+1)n}{4(1 + c_0)} = \frac{n}{4(1+c_0) } \left(\floor{\frac{\log (9\eps/\Delta)}{2\log q}} + 1\right) 
    \ge \frac{n}{8 (1 + c_0)} \left( -\frac{1}{\log(q)} \right) \log\left( \frac{\Delta}{9\eps} \right) \\
    &\ge \frac{n}{8 (1 + c_0)}\left( \frac{1}{2} \sqrt{\frac{L/\mu - 1}{n}} + \frac{\sqrt{2}}{4} + h(\sqrt{2})  \right) \log\left( \frac{\mu R^2}{18\eps} \right) 
    = \Omega\left( \left( n + \sqrt{\frac{n L}{\mu}} \right) \log\left( \frac{1}{\eps} \right) \right).
\end{align*} }

\item If $2 \le L/\mu < n/2 + 1$, then we have
{ \small
\begin{align}
    \notag -\log(q) &= \log\left( \frac{\alpha+1}{\alpha-1} \right)
    = \log \left( 1 + \frac{2(\alpha - 1)}{\alpha^2 - 1} \right) 
    = \log\left( 1 + \frac{\sqrt{2\frac{L/\mu - 1}{n} + 1} - 1}{\frac{L/\mu - 1}{n}} \right) \\
    & \le \log\left( 1 + \frac{(\sqrt{2} - 1)n}{L/\mu - 1} \right) 
    \label{ineq:strongly:large-n} 
    \le \log\left( \frac{(\sqrt{2} - 1/2)n}{L/\mu - 1} \right)
    \le \log\left( \frac{(2\sqrt{2} - 1)n}{L/\mu} \right),
\end{align} }
where the first inequality and second inequality follow from $L/\mu - 1 < n/2$ and the last inequality is according to $\frac{1}{x - 1} \le \frac{2}{x}$ for $x \ge 2$. 

Note that $n \ge 2$, thus $\frac{n}{n - 1} \le 2 \le \frac{n}{L/\mu - 1}$,
and hence $n \ge L/\mu$, i.e. $\log(n\mu/L) \ge 0$.

Therefore, 
\begin{align*}
    N &= \frac{(M+1)n}{4 (1 + c_0)} \ge \frac{n}{8 (1 + c_0)} \left( -\frac{1}{\log(q)} \right) \log\left( \frac{\mu R^2}{18 \eps} \right) 
    = \Omega \left( \left( \frac{n}{1 + \log(n\mu/L)} \right) \log\left( \frac{1}{\eps} \right)\right).
\end{align*}

Recalling that we assume that $q^2 \ge \frac{18\eps}{\mu R^2} > \frac{9 \eps}{\Delta}$, thus we have 
\begin{align*}
        N &\ge \frac{n}{8(1 + c_0)} \left( -\frac{1}{\log(q)} \right) \log\left( \frac{\Delta}{9 \eps} \right) 
        \ge \frac{n}{8(1 + c_0)} \left( -\frac{1}{\log(q)} \right) \left( -2\log(q) \right) 
        = \frac{n}{4(1 + c_0)}.
    \end{align*}
    
Therefore, $N = \Omega\left(n + \left( \frac{n}{1 + \log(n\mu/L)} \right) \log\left( \frac{1}{\eps} \right) \right)$.
\end{enumerate}

At last, we must ensure that $1 \le M < m$, that is
$    1 \le \frac{\log (9\eps/\Delta)}{2\log q} < m.$
Note that $\lim_{\beta \rightarrow +\infty} h(\beta) = 0$, so $- 1/ \log(q) \le \alpha /2$.
Thus the above conditions are satisfied when 
\begin{align*}
    m = \left\lfloor \frac{\log (\mu R^2 / (9\eps))}{2 (-\log q)} \right\rfloor + 1 \le
    \frac{1}{4}\left(\sqrt{2\frac{L/\mu - 1}{n} + 1}\right) \log \left(\frac{\mu R^2}{9\eps}\right) + 1
    = \fO \left( \sqrt{\frac{L}{n\mu}} \log \left(\frac{1}{\eps}\right)\right),
\end{align*}
and 
    $ \frac{\eps}{\Delta} 
    \le \frac{1}{9} \left(\frac{\alpha-1}{\alpha+1}\right)^{2}.$
\end{proof}
For the average smooth case, the hard instance can be directly derived from Definition \ref{defn:sc}.
\begin{definition}\label{defn:sc:average}
For fixed $L, \mu, R, n$ such that $L/\mu \ge 2$, consider $\{ f_{\mathrm{SC}, i} \}_{i=1}^n$, $f_{\mathrm{SC}}$ and Problem (\ref{prob:sc}) defined in Definition \ref{defn:sc} with $L$ replaced by $\sqrt{\frac{n(L^2 - \mu^2)}{2} - \mu^2}$.
\end{definition}
The following proposition ensures the hard instance is $L$-average smooth and gives the relationship between the smoothness parameter and the average smoothness parameter.
\begin{proposition}\label{prop:average:strongly:example}
Consider  $\{ f_{\mathrm{SC}, i} \}_{i=1}^n$ and $f_{\mathrm{SC}}$ defined in Definition \ref{defn:sc:average} and let $\tilde{L} = \sqrt{\frac{n(L^2 - \mu^2)}{2} - \mu^2} $.
For $n \ge 4$ and $\kappa = \frac{L}{\mu} \ge 2$, we have that
\begin{enumerate}
    \item $f_{\mathrm{SC}}(\vx)$ is $\mu$-strongly-convex and $\{f_{\mathrm{SC}, i}\}_{i=1}^n$ is $L$-average smooth.
    \item 
    $\frac{\sqrt{n}}{2} L \le \tilde{L} \le \sqrt{\frac{n}{2}} L$ and $\tilde{\kappa} = \frac{ \tilde{ L} }{\mu} \ge 2$.
\end{enumerate}
\end{proposition}

\begin{proof}
~
\begin{enumerate}
\item It is easy to check that $f_{\mathrm{SC}}(\vx)$ is $\mu$-strongly-convex.
By Proposition \ref{prop:base} and Lemma \ref{lem:scale}, $\{f_{\mathrm{SC}, i}\}_{i=1}^n$ is $\hat{L}$-average smooth, 
where
\begin{align*}
    \hat{L} 
    = \frac{ \tilde{L} - \mu}{ 2n } \sqrt{ \frac{4}{n} \left[ \left( \frac{n \tilde{L}/ \mu + n}{ \tilde{L} / \mu - 1}  \right)^2 + n^2 \right] + \left( \frac{2n}{ \tilde{L} / \mu - 1}  \right)^2 } 
    = \sqrt{\frac{2(\tilde{L}^2 + \mu^2)}{n} + \mu^2} = L.
\end{align*}

\item Clearly, $\tilde{L} = \sqrt{\frac{n(L^2 - \mu^2)}{2} - \mu^2} \le \sqrt{\frac{n}{2}} L$. \\
Furthermore, according to $\kappa \ge 2$ and $n \ge 4$, we have
\begin{align*}
    \tilde{L}^2 - \frac{n}{4} L^2
    = \frac{n}{4} L^2 - \frac{n}{2} \mu^2 - \mu^2 
    = \mu^2 \left( \frac{n}{4} \kappa^2 - \frac{n}{2} - 1 \right) 
     \ge \mu^2 \left( \frac{n}{2} - 1 \right) \ge 0.
\end{align*}
and $\tilde{\kappa} = \frac{\tilde{L}}{\mu} \ge \frac{\sqrt{n} L}{2\mu} \ge \kappa \ge 2$.
\end{enumerate}
This completes the proof.
\end{proof}
Recalling Theorem \ref{thm:strongly:example}, we have the following result.
\begin{theorem}\label{thm:average:strongly:example}
Consider the minimization problem (\ref{prob:sc}) and $\eps > 0$.
Suppose that
$\kappa = L / \mu \ge 2$, $n \ge 4$,
$ \eps \le \frac{\mu R^2}{18} \left(\frac{\alpha-1}{\alpha+1}\right)^{2}$
and
$m = \left\lfloor \frac{1}{4}\left(\sqrt{2\frac{ \tilde{L}/\mu - 1}{n} + 1}\right) \log \left(\frac{\mu R^2}{9\eps}\right) \right\rfloor + 1$
where 
$\alpha = \sqrt{\frac{2(  \tilde{L}/\mu - 1)}{n} + 1}$,
and
$\tilde{L} = \sqrt{\frac{n(L^2 - \mu^2)}{2} - \mu^2}$,
In order to find $\hat{\vx} \in \fX$ such that $ \E f_{\mathrm{SC}} (\hat{\vx}) - \min_{\vx \in \fX} f_{\mathrm{SC}} (\vx) < \eps$, PIFO algorithm $\fA$ needs at least $N$ queries, where
    \begin{align*}
        N =
        \begin{cases}
            \Omega\left(\left(n {+} n^{3/4} \sqrt{\kappa}\right)\log\left( 1/\eps \right)\right) , &\text{ for } 
            \kappa = \Omega(\sqrt{n}), \\
            \Omega\left( n + \left( \frac{n}{1 + (\log(\sqrt{n}/\kappa))_{+}} \right)\log\left(1/\eps\right) \right), &\text{ for }
            \kappa = \fO(\sqrt{n}).
        \end{cases}
    \end{align*}
\end{theorem}
For larger $\eps$, we can apply the following Lemma.
\begin{lemma}\label{lem:strongly:example:2}
For any $L, \mu, n, R, \eps$ such that $n \ge 2$ and $\eps \le L R^2 /4$, there exist $n$ functions $\{f_i: \BR \rightarrow \BR\}_{i=1}^n$ such that $f_i(x)$ is $L$-smooth, $\{ f_i \}_{i=1}^n$ is $L$-average smooth and $f(x) = \frac{1}{n} \sum_{i=1}^n f_i(x)$ is $\mu$-strongly-convex.
In order to find $|\hat{x}| \le R$ such that $\E f (\hat{x}) - \min_{|x| \le R} f (x) < \eps$, PIFO algorithm $\fA$ needs at least $N = \Omega(n)$ queries.
\end{lemma}

\begin{proof}
Consider the following functions $\{ G_{\mathrm{SC}, i} \}_{1 \le i \le n}$, $G_{\mathrm{SC}}: \BR \rightarrow \BR$, where
\begin{align*}
    G_{\mathrm{SC}, i}(x) &= \frac{L}{2} x^2 - n L R x, \,\text{ for } i = 1, \\
    G_{\mathrm{SC}, i}(x) &= \frac{L}{2} x^2,  \qquad \qquad \text{ for } i = 2, 3, \dots, n, 
\end{align*}
and  $G_{\mathrm{SC}}(x) = \frac{1}{n} \sum_{i=1}^n G_{\mathrm{SC}, i}(x) = \frac{L}{2} x^2 - L R x$.
Note that $\{ G_{\mathrm{SC}, i} \}_{i=1}^n$ is $L$ smooth and $\mu$-strongly-convex for any $\mu \le L$.
Observe that $x^* = \argmin_{x \in \BR} G_{\mathrm{SC}}(x) = R$, $    G_{\mathrm{SC}}(0) - G_{\mathrm{SC}}(x^*) = \frac{LR^2}{2}$
and $|x^*| = R$.
Thus $x^* = \argmin_{|x| \le R} G_{\mathrm{SC}}(x)$.

For $i > 1$, we have $\frac{d G_{\mathrm{SC}, i}(x)}{d x} \vert_{x = 0} = 0$ and $\prox_{G_{\mathrm{SC}, i}}^{\gamma} (0) = 0$.
Thus $x_t = 0$ will hold till our first-order method $\fA$ draws the component $G_{\mathrm{SC}, 1}$.
That is, for $t < T = \argmin \{t: i_t = 1\}$, we have $x_t = 0$.

Hence, for $t \le \frac{1}{2p_1}$, we have
{ 
\begin{align*}
    \E G_{\mathrm{SC}}(x_t) - G_{\mathrm{SC}}(x^*) &\ge \E \left[G_{\mathrm{SC}}(x_t) - G_{\mathrm{SC}}(x^*) \Big\vert \frac{1}{2p_1} < T\right] \pr{\frac{1}{2p_1} < T} 
    = \frac{LR^2}{2} \pr{\frac{1}{2p_1} < T}.
\end{align*}}

Note that $T$ follows a geometric distribution with success probability $p_1 \le 1/n$, and
\begin{align*}
    \pr{ T > \frac{1}{2p_1} } = \pr{ T > \floor{\frac{1}{2p_1}} } = (1 - p_1)^{\floor{\frac{1}{2p_1}}} 
    \ge (1 - p_1)^{\frac{1}{2p_1}} \ge (1 - 1/n)^{n/2} \ge \frac{1}{2},
\end{align*}
where the second inequality follows from $h(z) = \frac{\log(1 - z)}{2z}$ is a decreasing function.

Thus, for $t \le \frac{1}{2p_1}$, we have
$
    \E G_{\mathrm{SC}}(x_t) - G_{\mathrm{SC}}(x^*) \ge \frac{LR^2}{4} \ge \eps.$
Thus, in order to find $|\hat{x}| \le R$ such that $\E G_{\mathrm{SC}}(\hat{x}) - G_{\mathrm{SC}}(x^{*}) < \eps$, $\fA$ needs at least $ \frac{1}{2p_1} \ge n/2 = \Omega\left( n \right)$ queries.
\end{proof}


\begin{proof}[Proof of Theorem \ref{thm:strongly}]
It remains to explain that the lower bound in Lemma \ref{lem:strongly:example:2} is the same as the lower bound in Theorem \ref{thm:strongly:example} for $\eps > \frac{\mu R^2 }{18} \left(\frac{\alpha-1}{\alpha+1}\right)^2$. 
Suppose that
  $  \frac{\eps}{\mu R^2} > \frac{1}{18} \left( \frac{\alpha - 1}{\alpha + 1} \right)^2, \ 
    \alpha = \sqrt{2\,\frac{\kappa - 1}{n} + 1}
    \  \text{and} \ 
    \kappa = \frac{L}{\mu}.$

\begin{enumerate}
    \item If $\kappa \ge n/2 + 1$, then we have $\alpha \ge \sqrt{2}$ and
\begin{align*}
    &\quad \left( n + \sqrt{\kappa n} \right) \log\left( \frac{\mu R^2}{18 \eps} \right) 
    \le 2\left( n + \sqrt{\kappa n} \right) \log\left(\frac{\alpha + 1}{\alpha - 1}\right)\\
    &\le \frac{4\left( n + \sqrt{\kappa n} \right)}{\alpha - 1} 
    = \fO(n) + \frac{4 \sqrt{\kappa n}}{(1 - \sqrt{2}/2)\alpha} \\
    &\le \fO(n) + \frac{4}{\sqrt{2} - 1} \frac{\sqrt{\kappa n}}{\sqrt{\kappa/n}} = \fO(n),
\end{align*}
where the second inequality follows from $\log(1 + x) \le x$ and the last inequality is according to $\alpha \ge \sqrt{2\kappa/n}$.
Then we have
$
    \Omega(n) = \Omega\left(\left( n + \sqrt{\kappa n} \right) \log\left( \frac{1}{\eps} \right)\right).$

    \item If $2 \le L / \mu < n/2 + 1$, then we have
\begin{align*}
    &\quad \left( \frac{n}{1 + (\log(n\mu/L) )_+} \right) \log\left( \frac{\mu R^2}{18\eps} \right)
    \le \left( \frac{n}{1 + (\log(n\mu/L) )_+} \right) \left( 2 \log\left(\frac{\alpha + 1}{\alpha - 1}\right) \right) \\
    &\le \left( \frac{n}{1 + (\log(n\mu/L) )_+} \right) \left( 2 \log\left( \frac{(2\sqrt{2} - 1)n}{L/\mu} \right) \right) = \fO(n),
\end{align*}
where the second inequality 
is by (\ref{ineq:strongly:large-n}).
Then we have
    $\Omega(n) = \Omega\left(\left( \frac{n}{1 + \left( \log(n\mu/L) \right)_+}  \right) \log\left( \frac{1}{\eps} \right) + n\right).$
\end{enumerate}
This completes the proof.
\end{proof}
The proof of Theorem \ref{thm:average:strongly} is similar to that of Theorem \ref{thm:strongly}.

\subsection{Construction for the Convex Case}\label{sec:min:convex}

The analysis of lower bound complexity for the convex case depends on the following construction.
\begin{definition}\label{defn:c}
For fixed $L, R, n$, we define $f_{\mathrm{C}, i}: \BR^m \rightarrow \BR$ as follows
\begin{align*}
    f_{\mathrm{C}, i}(\vx) = \lambda r_i\left(\vx / \beta; m, 0, 1, \vc^{\mathrm{C}} \right), \text{ for } 1 \le i \le n,
\end{align*}
where 
$
    \vc^{\mathrm{C}} = (0,0,1), \; 
    \lambda = \frac{3 LR^2}{2 n (m+1)^3}\;
    \text{ and } \beta = \frac{ \sqrt{3} R }{ (m+1)^{3/2} }.$
Consider the minimization problem
\begin{align}
    \min_{\vx \in \fX} f_{\mathrm{C}}(\vx) \triangleq \frac{1}{n}\sum_{i=1}^n f_{\mathrm{C},i}(\vx). \label{prob:c}\\[-0.8cm]\nonumber
\end{align}
where $\fX = \{ \vx \in \BR^m : \norm{\vx} \le R  \}$.
\end{definition}
Then we have the following proposition.
\begin{proposition}\label{prop:convex}
For any $n \ge 2$, $m \ge 2$, the following properties hold:
\begin{enumerate}
    \item $f_{\mathrm{C}, i}$ is $L$-smooth and convex. Thus, $f_{\mathrm{C}}$ is convex.
    \item The minimizer of the function $f_{\mathrm{C}}$ is
    $$\vx^{*} = \argmin_{\vx \in \BR^m} f_{\mathrm{C}}(\vx) = \frac{2\xi}{L} \left(m, m-1, \dots, 1\right)^{\top},$$ 
    where $\xi = \frac{\sqrt{3}}{2} \frac{RL}{(m+1)^{3/2}}$.
    Moreover, $f_{\mathrm{C}}(\vx^*) = - \frac{m \xi^2}{n L}$ and $\norm{\vx^*} \le R$.
    \item For $1 \le k \le m$, we have
    \begin{align*}
    \min_{\vx \in \fX \cap \fF_k} f_{\mathrm{C}}(\vx) - \min_{\vx \in \fX} f_{\mathrm{C}}(\vx) = \frac{\xi^2}{n L} (m - k).
    \end{align*}
\end{enumerate}
\end{proposition}
\begin{proof}
\begin{enumerate}
    \item Just recall Proposition \ref{prop:base} and Lemma \ref{lem:scale}.
    \item It is easy to check
    $f_{\mathrm{C}} (\vx) 
    = \frac{L}{4n} \norm{\mB(m, 1) \vx}^2 - 
    \frac{\xi}{n}
    \inner{\ve_1}{\vx}$,
    where $\xi = \frac{\sqrt{3}}{2}\frac{BL}{(m+1)^{3/2}n}$. Let $\nabla f_{\mathrm{C}}(\vx) = \vzero$, that is
    $\frac{L}{2n}\mA(m, 0, 1) \vx = \frac{\xi}{n} \ve_1.$
One can 
check that the solution 
is 
$\vx^{*} = \frac{2\xi}{L} (m, m-1, \dots, 1)^{\top},$
and 
$ f_{\mathrm{C}}(\vx^*) = -\frac{m\xi^2}{n L}. $
Moreover, we have
\begin{align*}
    \norm{\vx^*}^2 &= \frac{4\xi^2}{L^2} \frac{m(m+1)(2m+1)}{6} 
    \le \frac{4\xi^2}{3L^2}(m+1)^3 = R^2.
\end{align*}

\item The second property implies $\min_{\vx \in \fX} f_{\mathrm{C}} (\vx) = - \frac{m \xi^2}{n L}$.
    By similar calculation to above proof, we have
$\argmin_{\vx \in \fX \cap \fF_k} f_{\mathrm{C}}(\vx) = \frac{2\xi}{L} (k, k-1, \dots, 1, 0, \dots, 0)^{\top},$
and 
$ \min_{\vx \in \fX \cap \fF_k} f_{\mathrm{C}}(\vx) = - \frac{k \xi^2}{nL}$.
Thus 
    $\min_{\vx \in \fX \cap \fF_k} f_{\mathrm{C}}(\vx) - \min_{\vx \in \fX} f_{\mathrm{C}}(\vx) = \frac{\xi^2}{n L} (m - k)$.
\end{enumerate}
This completes the proof.
\end{proof}
Next we show the lower bound for functions $f_{\mathrm{C}, i}$ defined above.
\begin{theorem}\label{thm:convex:example}
Consider the minimization problem (\ref{prob:c}) and $\eps > 0$. Suppose that
\begin{align*}
    n \ge 2, \;\,
    \eps \le \frac{R^2 L}{384 n}\; \mbox{ and } \; m = \floor{\sqrt{\frac{R^2 L}{24 n \eps}}} - 1.
\end{align*}
In order to find $\hat{\vx} \in \fX$ such that $ \E f_{\mathrm{C}} (\hat{\vx}) - \min_{\vx \in \fX} f_{\mathrm{C}} (\vx) < \eps$, PIFO algorithm $\fA$ needs at least $N$ queries, where
\begin{align*}
    N =
    \Omega\left(n {+} R\sqrt{nL/\eps} \right).
\end{align*}
\end{theorem}

\begin{proof}
Since $\eps \le \frac{R^2 L}{384 n}$, we have $m \ge 3$.
Let $\xi = \frac{\sqrt{3}}{2} \frac{RL}{(m+1)^{3/2}}$.

For $M = \floor{\frac{m-1}{2}} \ge 1$, we have $m - M \ge (m+1) /2$, and
\begin{align*}
    \min_{\vx \in \fX \cap \fF_M} f_{\mathrm{C}}(\vx) - \min_{\vx \in \fX} f_{\mathrm{C}}(\vx) &= \frac{\xi^2}{n L} (m - M) = \frac{3 R^2 L}{4n} \frac{m-M}{(m+1)^3} 
    \ge \frac{3 R^2 L}{8n} \frac{1}{(m+1)^2} \ge 9\eps,
\end{align*}
where the first equation is according to the 3rd property in Proposition \ref{prop:convex} and the last inequality follows from $m + 1 \le R \sqrt{L/(24n\eps)}$.

Similar to the proof of Theorem \ref{thm:strongly:example}, by Lemma \ref{lem:base},  we have
$
    \min_{t \le N} \E f_{\mathrm{C}}(\vx_t) - \min_{\vx \in \fX} f_{\mathrm{C}}(\vx) \ge \eps.$
In other words, in order to find $\hat{\vx} \in \fX$ such that $\E f_{\mathrm{C}}(\hat{\vx}) - \min_{\vx \in \fX} f_{\mathrm{C}}(\vx) < \eps$, $\fA$ needs at least $N$ queries.

At last, observe that
\begin{align*}
    N = \frac{(M+1)n}{4 (1 + c_0) } = \frac{n}{4 (1 + c_0) } \floor{\frac{m+1}{2}} 
    \ge \frac{n(m-1)}{8} 
    \ge \frac{n}{8} \left(\sqrt{\frac{R^2 L}{24 n \eps}} - 2\right) 
    = \Omega \left( n + R\sqrt{\frac{n L}{\eps}} \right),
\end{align*}
where we have recalled $\eps \le \frac{B^2 L}{384 n}$ in last equation.
\end{proof}
The hard instance for the average smooth case can be derived from Definition \ref{defn:c}.
\begin{definition}\label{defn:c:average}
For fixed $L, R,n$,
consider $\{ f_{\mathrm{C}, i} \}_{i=1}^n$ and $f_{\mathrm{C}}$ defined in Definition \ref{defn:c} with $L$ replaced by $ \sqrt{ \frac{n}{2} } L $.
\end{definition}
It follows from Proposition \ref{prop:base} and Lemma \ref{lem:scale} that
$f_{\mathrm{C}}$ is convex and
$\{f_{\mathrm{C}, i}\}_{i=1}^n$ is $L$-average smooth.
By Theorem \ref{thm:convex:example}, we have the following conclusion.
\begin{theorem}\label{thm:average:convex:example}
Consider the minimization problem (\ref{prob:c}) and $\eps > 0$.
Suppose that
\begin{align*}
    n \ge 2, \;
    \eps \le \frac{\sqrt{2}}{768} \frac{R^2 L}{\sqrt{n}}\; \mbox{ and } \;  m = \floor{ \frac{\sqrt[4]{18}}{12} R n^{-1/4}\sqrt{\frac{L}{\eps}}} - 1.
\end{align*}
In order to find $\hat{\vx} \in \fX$ such that $ \E f_{\mathrm{C}} (\hat{\vx}) - \min_{\vx \in \fX} f_{\mathrm{C}} (\vx) < \eps$, PIFO algorithm $\fA$ needs at least $N$ queries, where
    \begin{align*}
        N =
        \Omega\left(n + R n^{3/4} \sqrt{\frac{L}{\eps}}\right).
    \end{align*}
\end{theorem}

\begin{proof}[Proof of Theorem \ref{thm:convex}]
To derive Theorem \ref{thm:convex}, 
it remains to consider the 
case $\eps > \frac{ \sqrt{2} R^2 L}{ 768 \sqrt{n} }$.
By Lemma \ref{lem:strongly:example:2}, there exist $n$ functions $\{ f_i: \BR \rightarrow \BR \}$ such that $f_i(x)$ is $L$-smooth and $f(x) = \frac{1}{n} \sum_{i=1}^n f_i(x)$ is convex.
In order to find $|\hat{x}| \le R$ such that $\E f (\hat{x}) - \min_{|x| \le R} f (x) < \eps$, PIFO algorithm $\fA$ needs at least $N = \Omega(n)$ queries.
Since $\eps > \frac{ \sqrt{2} R^2 L}{768 \sqrt{n} }$, $\Omega (n) = \Omega \left( n + R\sqrt{\frac{nL}{\eps}} \right)$. This completes the proof.
\end{proof}
The proof of Theorem \ref{thm:average:convex} is similar.



\subsection{Construction for the Nonconvex Case}\label{sec:min:nonconvex}
The analysis of lower bound complexity for the nonconvex case depends on the following construction.

\begin{definition}\label{defn:nc}
For fixed $L, \mu, \Delta, n$, we define $f_{\mathrm{NC}, i}: \BR^{m+1} \rightarrow \BR$ as follows
\begin{align*}
    f_{\mathrm{NC}, i}(\vx) = \lambda r_i\left(\vx / \beta; 
    m + 1,
    \sqrt[4]{\alpha}, 0, \vc^{\mathrm{NC}} \right), \text{ for } 1 \le i \le n,
\end{align*}
where
\begin{align*}
    \alpha & = \min \left\{1, \frac{(\sqrt{3} + 1)n\mu}{30 L}, \frac{n}{180} \right\}, \; 
    \vc^{\mathrm{NC}} = \left( 0, \alpha, \sqrt{\alpha} \right),\; \\
    m & = \floor{\frac{\Delta L \sqrt{\alpha}}{40824 n \eps^2}},\; 
    \lambda  = \frac{3888 n \eps^2}{L \alpha^{3/2}} \; \text{ and } \beta = \sqrt{3 \lambda n / L}.
\end{align*}
Consider the minimization problem
\begin{align}
    \min_{\vx \in \BR^{m+1}} f_{\mathrm{NC}}(\vx) \triangleq \frac{1}{n}\sum_{i=1}^n f_{\mathrm{NC},i}(\vx). \label{prob:nc}\\[-0.8cm]\nonumber
\end{align}
\end{definition}
Then we have the following proposition.
\begin{proposition}\label{prop:nonconvex}
For any $n \ge 2$ and $\eps^2 \le \frac{\Delta L \alpha}{81648 n}$, the following properties hold:
\begin{enumerate}
    \item $f_{\mathrm{NC}, i}$ is $L$-smooth and $ (- \mu)$-weakly-convex. Thus, $f_{\mathrm{NC}}$ is $(-\mu)$-weakly-convex.
    \item $f_{\mathrm{NC}}(\vzero_{}) - \min_{\vx \in \BR^{m+1}} f_{\mathrm{NC}}(\vx) \le \Delta$.
    \item $m \ge 2$ and for $M = m-1$, $\min_{\vx \in \fF_{M}} \norm{\nabla f_{\mathrm{NC}}(\vx)} \ge  9 \eps.$
\end{enumerate}
\end{proposition}
\begin{proof}
\begin{enumerate}
\item By Proposition \ref{prop:base} and Lemma \ref{lem:scale}, $f_{\mathrm{NC}, i}$ is $(-l_1)$-weakly convex and $l_2$-smooth where
\begin{align*}
    l_1 &= \frac{45 (\sqrt{3} - 1) \alpha\lambda}{\beta^2} = \frac{45 (\sqrt{3} - 1)L}{3n} \alpha 
    \le \frac{45 (\sqrt{3} - 1) L}{3n} \frac{(\sqrt{3} + 1)n\mu}{30 L} = \mu, \\
    l_2 &= \frac{(2n + 180 \alpha)\lambda}{\beta^2} = \frac{L}{3n} (2n + 180 \alpha) \le L.
\end{align*}
Thus each $f_i$ is $L$-smooth and $(-\mu)$-weakly convex.
\item By Proposition \ref{prop:nonconvex:prop:base}, we know that 
\begin{align*}
    f_{\mathrm{NC}}(\vzero_{}) - \min_{\vx \in \BR^{m+1}} f_{\mathrm{NC}}(\vx) &\le \lambda (\sqrt{\alpha}/2 + 10 \alpha m) = \frac{1944 n \eps^2}{L \alpha} + \frac{38880 n \eps^2}{L \sqrt{\alpha}} m \\
    &\le \frac{1944}{40824} \Delta + \frac{38880}{40824} \Delta = \Delta.
\end{align*}
\item Since $\alpha \le 1$, we have $\frac{\Delta L^2 \sqrt{\alpha}}{40824 n \eps^2 } \ge \frac{\Delta L^2 \alpha}{40824 n \eps^2} $ and consequently $m \ge 2$.
By Proposition \ref{prop:nonconvex:prop:base}, we know that 
\begin{align*}
    \min_{\vx \in \fF_{M}} \norm{\nabla f_{\mathrm{NC}}(\vx)} \ge \frac{\alpha^{3/4}\lambda}{4\beta} 
    = \frac{\alpha^{3/4}\lambda}{4\sqrt{3\lambda n/L}} = \sqrt{\frac{\lambda L}{3n}} \frac{\alpha^{3/4}}{4} = 9 \eps.
\end{align*}
\end{enumerate}
This completes the proof.
\end{proof}
Next we prove Theorem \ref{thm:nonconvex}.

\begin{proof}[Proof of Theorem \ref{thm:nonconvex}]
By Lemma \ref{lem:base} and the third property of Proposition \ref{prop:nonconvex}, in order to find $\hat{\vx} \in \BR^{m+1}$ such that $\E \norm{\nabla f_{\mathrm{NC} }(\hat{\vx})} < \eps$, PIFO algorithm $\fA$ needs at least $N$ queries,
where
$
    N 
    = \frac{n m} {4(1 + c_0)} = \Omega \left(\frac{\Delta L \sqrt{\alpha}}{\eps^2}\right).$
Since $\eps^2 \le \frac{\Delta L \alpha}{81648 n}$ and $\alpha \le 1$, we have $\Omega \left(\frac{\Delta L \sqrt{\alpha}}{\eps^2}\right) = \Omega \left(n + \frac{\Delta L \sqrt{\alpha}}{\eps^2}\right) $.
\end{proof}

The analysis of lower bound complexity for the non-convex case under the average smooth assumption depends on the following construction.
\begin{definition}\label{defn:average:nc}
For fixed $L, \mu, \Delta, n$, we define $\bar{f}_{\mathrm{NC}, i}: \BR^{m+1} \rightarrow \BR$ as follows
\begin{align*}
    \bar{f}_{\mathrm{NC}, i}(\vx) = \lambda r_i\left(\vx / \beta; 
    m + 1,
    \sqrt[4]{\alpha}, 0, \bar{\vc}^{\mathrm{NC}} \right), \text{ for } 1 \le i \le n,
\end{align*}
where
\begin{align*}
\small
    \alpha & = \min \left\{1, \frac{8(\sqrt{3} + 1)\sqrt{n}\mu}{45 L}, \sqrt{\frac{n}{270}} \right\}, \;
    \bar{\vc}^{\mathrm{NC}} = \left( 0, \alpha, \sqrt{\alpha} \right),\; \\
    m & = \floor{\frac{\Delta L \sqrt{\alpha}}{217728 \sqrt{n} \eps^2}},\;
    \lambda = \frac{20736 \sqrt{n} \eps^2}{L \alpha^{3/2}}\;
    \text{ and } \beta = 4 \sqrt{\lambda \sqrt{n} / L}.
\end{align*}
Consider the minimization problem 
\begin{align}
    \min_{\vx \in \BR^{m+1}} \bar{f}_{\mathrm{NC}}(\vx) \triangleq \frac{1}{n}\sum_{i=1}^n \bar{f}_{\mathrm{NC},i}(\vx).\\[-0.8cm]\nonumber
\end{align}
\end{definition}
Then we have the following proposition.
\begin{proposition}\label{prop:average:nonconvex}
For any $n \ge 2$  and $\eps^2 \le \frac{\Delta L \alpha}{435456 \sqrt{n}}$, the following properties hold:
\begin{enumerate}
    \item $\bar{f}_{\mathrm{NC}, i}$ is $ (- \mu)$-weakly-convex and $\{ \bar{f}_{\mathrm{NC}, i} \}_{i=1}^n$ is $L$-average smooth. Thus, $f_{\mathrm{NC}}$ is $(-\mu)$-weakly-convex.
    \item $f_{\mathrm{NC}}(\vzero_{}) - \min_{\vx \in \BR^{m+1}} f_{\mathrm{NC}}(\vx) \le \Delta$.
    \item $m \ge 2$ and for $M = m-1$, $\min_{\vx \in \fF_{M}} \norm{\nabla f_{\mathrm{NC}}(\vx)} \ge  9 \eps.$
\end{enumerate}
\end{proposition}

\begin{proof}
\begin{enumerate}
\item By Proposition \ref{prop:base} and Lemma \ref{lem:scale}, $\bar{f}_{\mathrm{NC}, i}$ is $(-l_1)$-weakly convex and $\{ \bar{f}_{\mathrm{NC}, i} \}_{i=1}^n$ is $l_2$-average smooth where
\begin{align*}
    l_1 &= \frac{45 (\sqrt{3} - 1) \alpha\lambda}{\beta^2} = \frac{45 (\sqrt{3} - 1)L'}{16\sqrt{n}} \alpha 
    \le \frac{45 (\sqrt{3} - 1) L'}{16\sqrt{n}} \frac{8(\sqrt{3} + 1)\sqrt{n}\mu}{45 L'} = \mu, \\
    l_2 &= 4\sqrt{n + 4050 \alpha^2}\frac{\lambda}{\beta^2} = \frac{L'}{4\sqrt{n}} \sqrt{n + 4050 \alpha^2} \le L'.
\end{align*}
\item By Proposition \ref{prop:nonconvex:prop:base}, we know that 
\begin{align*}
    f_{\mathrm{NC}}(\vzero_{}) - \min_{\vx \in \BR^{m+1}} f_{\mathrm{NC}}(\vx) &\le \lambda (\sqrt{\alpha}/2 + 10 \alpha m) = \frac{10368 \sqrt{n} \eps^2}{L' \alpha} + \frac{207360 \sqrt{n} \eps^2}{L' \sqrt{\alpha}} m \\
    &\le \frac{10368}{217728} \Delta + \frac{207360}{217728} \Delta = \Delta.
\end{align*}
\item Since $\alpha \le 1$, we have $\frac{\Delta L' \sqrt{\alpha}}{217728 \sqrt{n} \eps^2 } \ge \frac{\Delta L' \alpha}{217728 \sqrt{n} \eps^2} $ and consequently $m \ge 2$.
By Proposition \ref{prop:nonconvex:prop:base}, we know that 
\begin{align*}
    \min_{\vx \in \fF_{M}} \norm{\nabla f_{\mathrm{NC}}(\vx)} \ge \frac{\alpha^{3/4}\lambda}{4\beta} 
    = \frac{\alpha^{3/4}\lambda}{4\sqrt{16\lambda \sqrt{n}/L'}} = \frac{\sqrt{\lambda L'}}{\sqrt[4]{n}} \frac{\alpha^{3/4}}{16} = 9 \eps.
\end{align*}
\end{enumerate}
This completes the proof.
\end{proof}
Next we prove Theorem \ref{thm:average:nonconvex}.

\begin{proof}[Proof of Theorem \ref{thm:average:nonconvex}]
By Lemma \ref{lem:base} and the third property of Proposition \ref{prop:average:nonconvex}, in order to find $\hat{\vx} \in \BR^{m+1}$ such that $\E \norm{\nabla f_{\mathrm{NC} }(\hat{\vx})} < \eps$, PIFO algorithm $\fA$ needs at least $N$ queries 
, where
$
    N 
    = \frac{n m}{4(1+c_0)} = \Omega \left(\frac{\Delta L \sqrt{n \alpha}}{\eps^2}\right).$
Since $\eps^2 \le \frac{\Delta L \alpha}{435456 \sqrt{n}}$ and $\alpha \le 1$, we have $ \Omega \left(\frac{\Delta L \sqrt{n \alpha}}{\eps^2}\right) = \Omega \left( n + \frac{\Delta L \sqrt{n \alpha}}{\eps^2}\right) $.
\end{proof}

\subsection{Proofs of Proposition \ref{prop:base} and Lemma \ref{lem:jump}}\label{appendix:min:proof:base}

We use $\Norm{\mA}$ to denote the spectral radius of $\mA$.
Recall that $\vb_{l-1}^{\top}$ is the $l$-th row of $\mB$,
$G(\vx) = \sum\limits_{i=1}^{m-1} \Gamma (x_i)$ and $$\fL_i = \{ l: 0 \le l \le m, l \equiv i - 1 (\bmod n) \}, i = 1, 2, \dots, n.$$
For simplicity, we omit the parameters of $\mB$, $\vb_l$ and $r_i$.

For $1 \le i \le n$, let $\mB_i$ be the submatrix whose rows are $\big\{ \vb_l^\top \big\}_{l \in \fL_i}$.
Then $r_i$ can be written as
\begin{align*}
r_i(\vx) &= \frac{n}{2} \norm{\mB_i \vx}^2 + \frac{c_1}{2} \norm{\vx}^2 + c_2 G(\vx) - c_3 n \inner{\ve_1}{\vx} \bone_{\{i=1\}}.
\end{align*}
\begin{proof}[Proof of Proposition \ref{prop:base}]
\begin{enumerate}
\item 

For the convex case,
\begin{align*}
r_i(\vx) &= \frac{n}{2} \norm{\mB_i \vx}^2 + \frac{c_1}{2} \norm{\vx}^2 - c_3 n \inner{\ve_1}{\vx} \bone_{\{i=1\}}.
\end{align*}
Obviously, $r_i$ is $c_1$-strongly convex.
Note that 
\begin{align*}
    \inner{\vu}{\mB_i^{\top}\mB_i \vu} &= \norm{\mB_i \vu}^2 \\
    &= \sum_{l \in \fL_i} (\vb_l^{\top} \vu)^2 \\\
    &= \sum_{l \in \fL_i\backslash \{0,m\}} (u_{l} - u_{l+1})^2 + \omega^2 u_1^2 \bone_{\{0 \in \fL_i\}} + \zeta^2 u_m^2 \bone_{\{m \in \fL_i\}} \\
    &\le 2 \norm{\vu}^2,
\end{align*}
where the last inequality is according to $(x + y)^2 \le 2(x^2 + y^2)$, and $|l_1 - l_2| \ge n \ge 2$ for $l_1, l_2 \in \fL_i$. 
Hence, $\Norm{\mB_i^{\top}\mB_i} \le 2$, and
\begin{align*}
    \Norm{\nabla^2 r_i(\vx)} = \Norm{n \mB_i^{\top} \mB_i + c_1 \mI} \le 2n  + c_1.
\end{align*}

\vskip 10pt
Next, observe that
\begin{align*}
    \norm{\nabla r_i(\vx_1) - \nabla r_i(\vx_2)}^2 
    = \norm{(n \mB_i^{\top} \mB_i + c_1 \mI) (\vx_1 - \vx_2)}^2
\end{align*}
Let $\vu = \vx_1 - \vx_2$.
Note that
\begin{align*}
    \vb_l \vb_l^{\top} \vu = 
    \begin{cases}
    (u_{l} - u_{l+1}) (\ve_{l} - \ve_{l+1}), &0 < l < m, \\
    \omega^2 u_1 \ve_1, &l = 0, \\
    \zeta^2 u_m \ve_m, &l = m.
    \end{cases}
\end{align*}

Thus,
\begin{align*}
    & \norm{(n \mB_i^{\top} \mB_i + c_1 \mI) \vu}^2 \\
    =& \norm{n \sum_{l \in \fL_i \backslash \{0,m\} } (u_{l} - u_{l+1}) (\ve_{l} - \ve_{l+1}) + n\omega^2 u_1^2 \bone_{\{0 \in \fL_i\}} + n\zeta^2 u_m^2 \bone_{\{m \in \fL_i\}}  + c_1 \vu}^2 \\
    =& \sum_{l \in \fL_i \backslash \{0,m\} } \left[(n (u_{l} - u_{l+1}) + c_1 u_{l})^2 + (-n (u_{l} - u_{l+1}) + c_1 u_{l+1})^2 \right] \\
    &+ (n \omega^2 + c_1)^2 u_1^2 \bone_{\{0 \in \fL_i\}}
    + (n \zeta^2 + c_1)^2 u_m^2 \mathds{1}_{\{m \in \fL_i\}}
    + \sum_{\substack{l-1, l \notin \fL_i \\ l \neq 0, m}} c_1^2 u_{l}^2 \\
    \le& 2\left[ (n + c_1)^2 + n^2 \right] \left[ \sum_{l \in \fL_i \backslash \{0, m\} } (u_l^2 + u_{l+1}^2) + u_1^2 \bone_{\{0 \in \fL_i\}} + u_m^2 \bone_{\{m \in \fL_i\}} \right]
    + c_1^2 \norm{\vu}^2,
\end{align*}
where we have used $(2n + c_1)^2 \le 2 \left[ (n + c_1)^2 + n^2 \right]$.

Therefore, we have
\begin{align*}
    &\quad \frac{1}{n} \sum_{i=1}^n \norm{\nabla r_i(\vx_1) - \nabla r_i(\vx_2)}^2 \\
    &\le \frac{1}{n}  \sum_{l=0}^{m} 4\left[ (n + c_1)^2 + n^2 \right] u_l^2  + c_1^2 \norm{\vu}^2 \\
    &\le \frac{4}{n} \left[ \left[ (n + c_1)^2 + n^2 \right] \right] \norm{\vu}^2 + c_1^2 \norm{\vu}^2,
\end{align*}
In summary, we get that $\{r_i\}_{i=1}^n$ is $L'$-average smooth, where
\begin{align*}
    L' = \sqrt{\frac{4}{n} \left[ (n + c_1)^2 + n^2 \right] + c_1^2}.
\end{align*}
\item
The results of the non-convex case follow from the above proof, Proposition \ref{prop:nonconvex:prop:base} and the inequality $(a+b)^2 \le 2(a^2 + b^2)$.
\end{enumerate}
This completes the proof.
\end{proof}

\begin{proof}[Proof of Lemma \ref{lem:jump}]
\begin{enumerate}
\item
For the convex case,
\begin{align*}
r_j(\vx) &= \frac{n}{2} \norm{\mB_j \vx}^2 + \frac{c_1}{2} \norm{\vx}^2 - c_3 n \inner{\ve_1}{\vx} \bone_{\{j=1\}}.
\end{align*}
Recall that
\begin{align*}
    \vb_l \vb_l^{\top} \vx = 
    \begin{cases}
    (x_{l} - x_{l+1}) (\ve_{l} - \ve_{l+1}), &0 < l < m, \\
    \omega^2 x_1 \ve_1, &l = 0, \\
    \zeta^2 x_m \ve_m, &l = m.
    \end{cases}
\end{align*}
For $\vx \in \fF_0$, we have $\vx = \vzero$, and
\begin{align*}
\nabla r_1(\vx) = c_3 n \ve_1 \in \fF_1, \\
\nabla r_j(\vx) = \vzero ~(j \ge 2).
\end{align*}
For $\vx \in \fF_k ~(1 \le k < m)$, we have 
\begin{align*}
    \vb_l \vb_l^{\top} \vx \in
    \begin{cases}
    \fF_{k}, &l \neq k, \\
    \fF_{k+1}, &l = k.
    \end{cases}
\end{align*}
Moreover, we suppose $k \in \fL_{i}$. Since
\begin{align*}
    \nabla r_j(\vx) &= n \mB_j^{\top} \mB_j \vx + c_1 \vx - c_3 n \ve_1 \bone_{\{j=1\}} \\
    &= n \sum_{l \in \fL_j} \vb_l \vb_l^{\top} \vx + c_1 \vx - c_3 n \ve_1 \bone_{\{j=1\}},
\end{align*}
it follows that $\nabla r_{i} (\vx) \in \fF_{k+1}$ and $\nabla r_j (\vx) \in \fF_{k} ~(j \neq i)$.

\vskip 5pt
Now, we turn to consider $\vu = \prox_{r_j}^{\gamma} (\vx)$. We have
\begin{align*}
    \left(n \mB_j^{\top} \mB_j + \left(c_1 + \frac{1}{\gamma}\right)\mI  \right) \vu = c_3 n \ve_1 \bone_{\{j=1\}} + \frac{1}{\gamma} \vx,
\end{align*}
i.e.,
\begin{align*}
    \vu = d_1 (\mI + d_2 \mB_j^{\top} \mB_j)^{-1} \vy,
\end{align*}
where $d_1 = \frac{1}{c_1 + 1/\gamma}$, $d_2 = \frac{n}{c_1 + 1/\gamma}$, and $\vy = c_3 n \ve_1 \bone_{\{j=1\}} + \frac{1}{\gamma} \vx$. 

Note that 
\begin{align*}
    (\mI + d_2 \mB_j^{\top} \mB_j)^{-1} = \mI - \mB_j^{\top} \left( \frac{1}{d_2}\mI + \mB_j \mB_j^{\top} \right)^{-1} \mB_j.
\end{align*}
If $k = 0$ and $j > 1$, we have $\vy = \vzero$ and $\vu = \vzero$. \\
If $k = 0$ and $j = 1$, we have $\vy = c_3 n \ve_1$. Since $\omega=0$, $\mB_1 \ve_1 = \vzero$, so $\vu = c_1 \vy \in \fF_1$. 

For $k \ge 1$, we know that $\vy \in \fF_k$.
And observe that if $|l - l'| \ge 2$, then $\vb_{l}^{\top} \vb_{l'} = 0$, and consequently $\mB_j \mB_j^{\top}$ is a diagonal matrix, 
so we can assume that $\frac{1}{d_2}\mI + \mB_j \mB_j^{\top} = \diag(\beta_{j,1}, \dots, \beta_{j, |\fL_j|})$.
Therefore,
\begin{align*}
    \vu = d_1 \vy - d_1 \sum_{s = 1}^{|\fL_j|} \beta_{j, s} \vb_{l_{j, s}} \vb_{l_{j, s}}^{\top} \vy, 
\end{align*}
where we assume that $\fL_j = \{l_{j, 1}, \dots, l_{j, |\fL_j|}\}$.

Thus, we have $\prox_{r_{i}}^{\gamma} (\vx) \in \fF_{k+1}$ for $k \in \fL_i$
and $\prox_{r_{j}}^{\gamma} (\vx) \in \fF_{k} ~(j \neq i)$.

\item
For the non-convex case,
\begin{align*}
r_j(\vx) &= \frac{n}{2} \norm{\mB_j \vx}^2 + c_2 G(\vx)- c_3 n \inner{\ve_1}{\vx} \bone_{\{j=1\}}.
\end{align*}
Let $\Gamma'(x)$ be the derivative of $\Gamma(x)$.
First note that $\Gamma'(0) = 0$, so if $\vx \in \fF_k$, then 
\[\nabla G(\vx) = \big( \Gamma'(x_1), \Gamma'(x_2), \dots, \Gamma'(x_{m-1}), 0 \big)^{\top} \in \fF_k. \]

For $\vx \in \fF_0$, we have $\vx = \vzero$, and
\begin{align*}
\nabla r_1(\vx) = c_3 n \ve_1 \in \fF_1, \\
\nabla r_j(\vx) = \vzero ~(j \ge 2).
\end{align*}

For $\vx \in \fF_k ~(1 \le k < m)$, recall that
\begin{align*}
    \vb_l \vb_l^{\top} \vx = 
    \begin{cases}
    (x_{l} - x_{l+1}) (\ve_{l} - \ve_{l+1}), &0 < l < m, \\
    \omega^2 x_1 \ve_1, &l = 0, \\
    \zeta^2 x_m \ve_m, &l = m.
    \end{cases}
\end{align*}
Suppose $k \in \fL_i$. Since
\begin{align*}
    \nabla r_j(\vx) &= n \mB_j^{\top} \mB_j \vx + c_2 \nabla G(\vx) - c_3 n \ve_1 \bone_{\{j=1\}} \\
    &= n \sum_{l \in \fL_j} \vb_l \vb_l^{\top} \vx + c_2 \nabla G(\vx) - c_3 n \ve_1 \bone_{\{j=1\}},
\end{align*}
it follows that $\nabla r_{i} (\vx) \in \fF_{k+1}$ and $\nabla r_j (\vx) \in \fF_{k} ~(j \neq i)$.

\vskip 5pt
Now, we turn to consider $\vu = \prox_{r_j}^{\gamma} (\vx)$. 
We have
\begin{align*}
    \nabla r_j(\vu) + \frac{1}{\gamma} (\vu - \vx) = \vzero,
\end{align*}
that is
\begin{align*}
    \left(n \sum_{l \in \fL_j} \vb_l \vb_l^{\top} + \frac{1}{\gamma} \mI \right) \vu + c_2 \nabla G(\vu) = \vy,
\end{align*}
where $\vy = c_3 n \ve_1 \bone_{\{j=1\}} + \frac{1}{\gamma} \vx$. Since $\gamma < \frac{\sqrt{2} + 1}{60 c_2}$, we have the following claims.

\begin{enumerate}
\item If $0 < l < m-1$ and $l \in \fL_j$, we have
\begin{equation}
\label{proof:nonconvex:case1}
\begin{aligned}
    n(u_l - u_{l+1}) + \frac{1}{\gamma} u_l + 120 c_2 \frac{ u_l^2 (u_l - 1) }{ 1 + u_l^2 } & =  y_l\\
    n(u_{l+1} - u_{l}) + \frac{1}{\gamma} u_{l+1} + 120 c_2 \frac{ u_{l+1}^2 (u_{l+1} - 1) }{ 1 + u_{l+1}^2 } & =  y_{l+1}.
\end{aligned}
\end{equation}
By Lemma \ref{lem:solution:z1z2}, $y_l = y_{l+1} = 0$ implies $u_l = u_{l+1} = 0$.
\item If $m-1 \in \fL_j$, we have
\begin{equation}
\label{proof:nonconvex:case2}
\begin{aligned}
    n(u_{m-1} - u_{m}) + \frac{1}{\gamma} u_{m-1} + 120 c_2 \frac{ u_{m-1}^2 (u_{m-1} - 1) }{ 1 + u_{m-1}^2 } & =  y_{m-1}\\
    n(u_{m} - u_{m-1}) + \frac{1}{\gamma} u_{m} & =  y_{m}.
\end{aligned}
\end{equation}
If $y_{m-1} = y_m = 0$, we obtain
\begin{equation*}
\begin{aligned}
    \frac{ 1 + 2 \gamma n }{\gamma (1 + \gamma n) } u_{m-1} +  120 c_2 \frac{ u_{m-1}^2 (u_{m-1} - 1) }{ 1 + u_{m-1}^2 }  & =  0 \\
    \left( n + \frac{1}{\gamma} \right) u_m - \frac{1}{\gamma} u_{m-1} & = 0.
\end{aligned}
\end{equation*}
By Lemma \ref{lem:solution:z}, $u_{m-1} = u_m = 0$.
\item If $m \in \fL_j$, we have
\begin{align}
\label{proof:nonconvex:case3}
    n \zeta^2 u_m + \frac{1}{\gamma} u_m = y_m.
\end{align}
$y_m = 0$ implies $u_m = 0$.
\item If $l > 0$ and $l-1, l  \notin \fL_j$, we have
\begin{align}
\label{proof:nonconvex:case4}
    \frac{1}{\gamma} u_l + 120 c_2 \frac{ u_l^2 (u_l - 1) }{ 1 + u_l^2 } \bone_{\{ l < m \}}  = y_l.
\end{align}
By Lemma \ref{lem:solution:z}, $y_l = 0$ implies $u_l = 0$.
\end{enumerate}

For $\vx \in \fF_0$ and $j = 1$, we have $\vx = \vzero$ and $\vy = n \omega^2 \ve_1$. Since $n \ge 2$, we have $1 \notin \fL_1$. If $2 \in \fL_1$, we can consider the solution to Equations (\ref{proof:nonconvex:case1}), (\ref{proof:nonconvex:case2}) or (\ref{proof:nonconvex:case3}) and conclude that $u_2 = 0$. If $2 \notin \fL_1$, we can consider the solution to Equation (\ref{proof:nonconvex:case4}) and conclude that $u_2=0$. Similarly, we can obtain $u_l = 0$ for $l \ge 2$, which implies $\vu \in \fF_1$.

For $\vx \in \fF_0$ and $j > 1$, we have $\vy=\vzero$ and $0 \notin \fL_j $. If $1 \in \fL_j$, we can consider the solution to Equations (\ref{proof:nonconvex:case1}) or (\ref{proof:nonconvex:case2}) and conclude that $u_1 = 0$. If $1 \notin \fL_j$, we can consider the solution to Equation (\ref{proof:nonconvex:case4}) and conclude that $u_1=0$. Similarly, we can obtain $u_l = 0$ for all $l$, which implies $\vu = \vzero \in \fF_0$.

For $k \ge 1$, we know that $\vy \in \fF_k$. Suppose $k \in \fL_i$. 

If $j = i$, we have $k+1 \notin \fL_i$. If $k = m-1$, clearly we have $\vu \in \fF_{k+1}$. Now we suppose $k < m-1$. If $k+2 \in \fL_i$, we can consider the solution to Equations (\ref{proof:nonconvex:case1}), (\ref{proof:nonconvex:case2}) or (\ref{proof:nonconvex:case3}) and conclude that $u_{k+2} = 0$. If $k+2 \notin \fL_1$, we can consider the solution to Equation (\ref{proof:nonconvex:case4}) and conclude that $u_{k+2}=0$. Similarly, we can obtain $u_l = 0$ for $l \ge k+2$, which implies $\vu \in \fF_{k+1}$.

If $j \neq i$, we have $k \notin \fL_j$. If $k+1 \in \fL_j$, we can consider the solution to Equations (\ref{proof:nonconvex:case1}), (\ref{proof:nonconvex:case2}) or (\ref{proof:nonconvex:case3}) and conclude that $u_{k+1} = 0$. If $k+1 \notin \fL_j$, we can consider the solution to Equation (\ref{proof:nonconvex:case4}) and conclude that $u_{k+1}=0$. Similarly, we can obtain $u_l = 0$ for $l \ge k+1$, which implies $\vu \in \fF_k$.
\end{enumerate}
This completes the proof.
\end{proof}

\end{document}